\documentclass[10pt,oneside]{amsart}
\usepackage{amssymb, amsmath, amsthm, amsfonts}
\usepackage{stmaryrd}
\usepackage{mathrsfs,comment}
\usepackage{graphicx}
\usepackage{placeins}
\usepackage{rotating}
\usepackage{tikz}
\usepackage{float}
\usepackage{subfigure}
\usepackage{caption}
\usepackage{pdflscape}
\usepackage{notoccite}
\usepackage{hyperref}  
\usepackage{url}
\usepackage[all,arc,2cell]{xy}
\UseAllTwocells
\usepackage{enumerate}
\usepackage{chngcntr}
 \usepackage{lineno}
 \usepackage{blindtext}
\usepackage{verbatim}
\usepackage{soul}
\usepackage[normalem]{ulem}
\usepackage{lscape}
\usepackage{geometry}
\usepackage{pifont}
\usepackage{todonotes}

\usepackage{tikz}
\usepackage{tikz-cd} 
\usepackage[makeroom]{cancel}

\hypersetup{%
  bookmarksnumbered=true,%
  bookmarks=true,%
  colorlinks=true,%
  linkcolor=blue,%
  citecolor=blue,%
  filecolor=blue,%
  menucolor=blue,%
  urlcolor=blue,%
  pdfnewwindow=true,%
  pdfstartview=FitBH}

\def\@url#1{{\tt\def~{\lower3.5pt\hbox{\char'176}}\def\_{\char'137}#1}}

\def\makeautorefname#1#2{\expandafter\def\csname#1autorefname\endcsname{#2}}
\makeautorefname{equation}{Equation}%
\makeautorefname{footnote}{footnote}%
\makeautorefname{item}{item}%
\makeautorefname{figure}{Figure}%
\makeautorefname{table}{Table}%
\makeautorefname{part}{Part}%
\makeautorefname{appendix}{Appendix}%
\makeautorefname{chapter}{Chapter}%
\makeautorefname{section}{Section}%
\makeautorefname{subsection}{Section}%
\makeautorefname{subsubsection}{Section}%
\makeautorefname{paragraph}{Paragraph}%
\makeautorefname{subparagraph}{Paragraph}%
\makeautorefname{theorem}{Theorem}%
\makeautorefname{theo}{Theorem}%
\makeautorefname{thm}{Theorem}%
\makeautorefname{addendum}{Addendum}%
\makeautorefname{addend}{Addendum}%
\makeautorefname{add}{Addendum}%
\makeautorefname{maintheorem}{Main theorem}%
\makeautorefname{mainthm}{Main theorem}%
\makeautorefname{corollary}{Corollary}%
\makeautorefname{claim}{Claim}%
\makeautorefname{corol}{Corollary}%
\makeautorefname{coro}{Corollary}%
\makeautorefname{cor}{Corollary}%
\makeautorefname{lemma}{Lemma}%
\makeautorefname{lemm}{Lemma}%
\makeautorefname{lem}{Lemma}%
\makeautorefname{sublemma}{Sublemma}%
\makeautorefname{sublem}{Sublemma}%
\makeautorefname{subl}{Sublemma}%
\makeautorefname{proposition}{Proposition}%
\makeautorefname{proposit}{Proposition}%
\makeautorefname{propos}{Proposition}%
\makeautorefname{propo}{Proposition}%
\makeautorefname{prop}{Proposition}%
\makeautorefname{property}{Property}
\makeautorefname{proper}{Property}
\makeautorefname{scholium}{Scholium}%
\makeautorefname{step}{Step}%
\makeautorefname{conjecture}{Conjecture}%
\makeautorefname{conject}{Conjecture}%
\makeautorefname{conj}{Conjecture}%
\makeautorefname{question}{Question}
\makeautorefname{questn}{Question}
\makeautorefname{quest}{Question}
\makeautorefname{ques}{Question}
\makeautorefname{qn}{Question}
\makeautorefname{definition}{Definition}%
\makeautorefname{defin}{Definition}%
\makeautorefname{defi}{Definition}%
\makeautorefname{def}{Definition}%
\makeautorefname{defn}{Definition}%
\makeautorefname{dfn}{Definition}%
\makeautorefname{notation}{Notation}
\makeautorefname{nota}{Notation}
\makeautorefname{notn}{Notation}
\makeautorefname{remark}{Remark}%
\makeautorefname{rema}{Remark}%
\makeautorefname{rem}{Remark}%
\makeautorefname{rmk}{Remark}%
\makeautorefname{rk}{Remark}%
\makeautorefname{remarks}{Remarks}%
\makeautorefname{rems}{Remarks}%
\makeautorefname{rmks}{Remarks}%
\makeautorefname{rks}{Remarks}%
\makeautorefname{example}{Example}%
\makeautorefname{examp}{Example}%
\makeautorefname{exmp}{Example}%
\makeautorefname{exmps}{Examples}%
\makeautorefname{exam}{Example}%
\makeautorefname{exa}{Example}%
\makeautorefname{algorithm}{Algorith}%
\makeautorefname{algo}{Algorith}%
\makeautorefname{alg}{Algorith}%
\makeautorefname{axiom}{Axiom}%
\makeautorefname{axi}{Axiom}%
\makeautorefname{ax}{Axiom}%
\makeautorefname{case}{Case}%
\makeautorefname{claim}{Claim}%
\makeautorefname{clm}{Claim}%
\makeautorefname{assumption}{Assumption}%
\makeautorefname{assumpt}{Assumption}%
\makeautorefname{conclusion}{Conclusion}%
\makeautorefname{concl}{Conclusion}%
\makeautorefname{conc}{Conclusion}%
\makeautorefname{condition}{Condition}%
\makeautorefname{condit}{Condition}%
\makeautorefname{cond}{Condition}%
\makeautorefname{construction}{Construction}%
\makeautorefname{construct}{Construction}%
\makeautorefname{const}{Construction}%
\makeautorefname{cons}{Construction}%
\makeautorefname{criterion}{Criterion}%
\makeautorefname{criter}{Criterion}%
\makeautorefname{crit}{Criterion}%
\makeautorefname{exercise}{Exercise}%
\makeautorefname{exer}{Exercise}%
\makeautorefname{exe}{Exercise}%
\makeautorefname{problem}{Problem}%
\makeautorefname{problm}{Problem}%
\makeautorefname{probm}{Problem}%
\makeautorefname{prob}{Problem}%
\makeautorefname{solution}{Solution}%
\makeautorefname{soln}{Solution}%
\makeautorefname{sol}{Solution}%
\makeautorefname{summary}{Summary}%
\makeautorefname{summ}{Summary}%
\makeautorefname{sum}{Summary}%
\makeautorefname{operation}{Operation}%
\makeautorefname{oper}{Operation}%
\makeautorefname{observation}{Observation}%
\makeautorefname{observn}{Observation}%
\makeautorefname{obser}{Observation}%
\makeautorefname{obs}{Observation}%
\makeautorefname{ob}{Observation}%
\makeautorefname{convention}{Convention}%
\makeautorefname{convent}{Convention}%
\makeautorefname{conv}{Convention}%
\makeautorefname{cvn}{Convention}%
\makeautorefname{warning}{Warning}%
\makeautorefname{warn}{Warning}%
\makeautorefname{note}{Note}%
\makeautorefname{fact}{Fact}%
\makeautorefname{notation}{Notation}

  \makeatletter
                   \let\c@lemma\c@theorem
                  \makeatother

\newtheorem{thm}{Theorem}[section]
\newtheorem{cor}{Corollary}[section]
\newtheorem{prop}{Proposition}[section]
\newtheorem{lem}{Lemma}[section]
\newtheorem{notation}{Notation}[section]

\theoremstyle{definition}
\newtheorem{defn}{Definition}[section]
\newtheorem{exmp}{Example}[section]

\newtheorem{rem}{Remark}[section]
\newtheorem{rmk}{Remark}[section]

\makeatletter
\let\c@lem=\c@thm
\let\c@cor=\c@thm
\let\c@prop=\c@thm
\let\c@lem=\c@thm
\let\c@defn=\c@thm
\let\c@exmps=\c@thm
\let\c@exmp=\c@thm
\let\c@rem=\c@thm
\let\c@warn=\c@thm
\let\c@claim=\c@thm
\let\c@quest=\c@thm
\let\c@rmk=\c@thm
\let\c@rem=\c@thm
\let\c@notation=\c@thm
\let\c@equation=\c@thm
\makeatother

\numberwithin{equation}{section}

\newcommand{\Z}{\mathbb{Z}}

\newcommand{\R}{\mathbb{R}}

\newcommand{\cC}{\mathcal{C}}
\newcommand{\cD}{\mathcal{D}}

\newcommand{\cM}{\mathcal{M}}
\newcommand{\cN}{\mathcal{N}}

\DeclareSymbolFontAlphabet{\scr}{rsfs}

\newcommand{\Ncyc}{\mathrm{N}^{\mathrm{cyc}}}
\newcommand{\speccat}{\mathrm{SpecCat}}
\newcommand{\eTHH}{\mathrm{eTHH}}

\newcommand{\EG}{\mathcal{E}G}
\newcommand{\Set}{\mathrm{Set}}

\def\quickop#1{\expandafter\newcommand\csname #1\endcsname{\operatorname{#1}}}
\quickop{Hom} \quickop{End} \quickop{Aut} \quickop{Tel} \quickop{Mic} 
\quickop{Ext} \quickop{Tor} \quickop{Cotor} \quickop{Id} \quickop{Coker} \quickop{Ker}
\quickop{Lim} \quickop{Colim} \quickop{Holim} \quickop{Hocolim}
\quickop{id} \quickop{tel} \quickop{mic} \quickop{coker} 
\quickop{colim} 

\DeclareMathOperator{\Mod}{Mod}

\DeclareMathOperator{\Ho}{Ho}

\DeclareMathOperator{\op}{op}
\DeclareMathOperator{\perf}{Perf}
\DeclareMathOperator{\stperf}{stPerf}

\DeclareMathOperator{\SpecCat}{SpecCat}
\DeclareMathOperator{\Sp}{Sp}
\DeclareMathOperator{\Fun}{Fun}

\DeclareMathOperator{\bimod}{Bimod}
\newcommand{\cyc}{\mathrm{cyc}}

\DeclareMathOperator{\THH}{THH}
\DeclareMathOperator{\ETHH}{ETHH}

\newcommand{\Mack}{\mathrm{Mack}_G}

\renewcommand\Id{\textrm{Id}}

\definecolor{darkspringgreen}{rgb}{0.09, 0.45, 0.27}
\definecolor{darkterracotta}{rgb}{0.8, 0.31, 0.36}
	\definecolor{darkcoral}{rgb}{0.8, 0.36, 0.27}
	\definecolor{indiagreen}{rgb}{0.07, 0.53, 0.03}
	\definecolor{mountainmeadow}{rgb}{0.19, 0.73, 0.56}
	\definecolor{mountbattenpink}{rgb}{0.6, 0.48, 0.55}
	\definecolor{palatinatepurple}{rgb}{0.41, 0.16, 0.38}
	\definecolor{cinnamon}{rgb}{0.82, 0.41, 0.12}
	\definecolor{chocolate}{rgb}{0.82, 0.41, 0.12}

\usepackage{scalerel}
\newcommand{\sm}{\wedge}

\newcommand{\TC}{\textnormal{TC}}

\newcommand{\abs}[1]{\lvert #1\rvert}

\newcommand{\Spec}{\textnormal{Sp}}

\usepackage[utf8]{inputenc}

\title{Trace methods for equivariant algebraic $K$-theory}

\author[Chan]{David Chan}
\address[Chan]{Department of Mathematics, Michigan State University, East Lansing, MI 48824 }

\author[Gerhardt]{Teena Gerhardt}
\address[Gerhardt]{Department of Mathematics, Michigan State University, East Lansing, MI 48824 }

\author[Klang]{Inbar Klang}
\address[Klang]{Department of Mathematics,
Vrije Universiteit Amsterdam - Faculty of Science,
De Boelelaan 1111,
1081 HV Amsterdam,
The Netherlands}

\begin{document}

\begin{abstract}
In the past decades, one of the most fruitful approaches to the study of algebraic $K$-theory has been trace methods, which construct and study trace maps from algebraic $K$-theory to topological Hochschild homology and related invariants. In recent years, theories of equivariant algebraic $K$-theory have emerged, but thus far few tools are available for the study and computation of these theories. In this paper, we lay the foundations for a trace methods approach to equivariant algebraic $K$-theory. For $G$ a finite group, we construct a Dennis trace map from equivariant algebraic $K$-theory to a $G$-equivariant version of topological Hochschild homology; for $G$ the trivial group this recovers the ordinary Dennis trace map. We show that upon taking fixed points, this recovers the trace map of Adamyk--Gerhardt--Hess--Klang--Kong, and gives a trace map from the fixed points of coarse equivariant $A$-theory to the free loop space. We also establish important properties of equivariant topological Hochschild homology, such as Morita invariance, and explain why it can be considered as a multiplicative norm.
\end{abstract}

\maketitle
\markboth{CHAN, GERHARDT, AND KLANG}{Trace methods for equivariant algebraic $K$-theory}  
\renewcommand*\contentsname{}
\tableofcontents

\section{Introduction}
  
Algebraic $K$-theory is a way of extracting interesting invariants from rings.  While these invariants can be difficult to compute in general, they carry a tremendous amount of information. For instance, when $R$ is a Dedekind domain the zeroth (reduced) algebraic $K$-group is isomorphic to the ideal class group of $R$, and the higher algebraic $K$-groups encode even more number theoretic information.  In topology, initial interest in algebraic $K$-theory grew out of its relationship with geometric invariants of manifolds. If $M$ is a connected manifold the zeroth and first algebraic $K$-groups of the group ring $\mathbb{Z}[\pi_1(M)]$ are the homes of geometric invariants of $M$, namely Wall's finiteness obstruction and Whitehead torsion.
  
  In the 1990's, new techniques from stable homotopy theory allowed for the definition of algebraic $K$-theory for ring spectra. Let $\Omega M$ denote the based loop space of $M$; taking the suspension spectrum yields a ring spectrum $\mathbb{S}[\Omega M]=\Sigma^{\infty}_+(\Omega M)$, and the algebraic $K$-theory of this ring spectrum recovers the $A$-theory of $M$, denoted by $A(M)$.  The spectrum $A(M)$ was first constructed by Waldhausen \cite{Waldhausen}, and plays a central role in the stable parametrized h-cobordism theorem \cite{WaldhausenJahrenRognes}.  One can think of the spectrum $\mathbb{S}[\Omega M]$ as a topological enrichment of the group ring $\mathbb{Z}[\pi_1(M)]$, and this explains the presence of earlier geometric invariants in the algebraic $K$-groups of the group ring. 

A very successful approach to the study and computation of algebraic $K$-theory is through trace methods, in which algebraic $K$-theory is approximated by topological versions of classical invariants from homological algebra. One key such invariant is topological Hochschild homology (THH), a topological analogue of the classical theory of Hochschild homology for algebras. For a ring spectrum $R$ there is a Dennis trace map 
\[
    tr\colon K(R) \mapsto \THH(R)
\]
relating algebraic $K$-theory and topological Hochschild homology. From THH one can define topological cyclic homology (TC) \cite{BHM, NS}. The Dennis trace factors through topological cyclic homology, $ K(R) \to \TC(R)$. This cyclotomic trace
 is close to an equivalence in many situations (see, for instance, \cite{DGM, GeHe06, CMM21}), giving an approach to algebraic $K$-theory which has been extremely fruitful. 

Recent years have seen an explosion of interest in equivariant homotopy theory and equivariant algebra.  Roughly speaking, these fields study the homotopical invariants and structures of spaces and spectra with an action by a compact Lie group.  Given the central role of $K$-theory in much of modern stable homotopy theory, a natural question to consider is the extent to which $K$-theoretic tools can be applied to analyze equivariant multiplicative structures.  Foundational work of Merling, Barwick, Barwick--Glasman--Shah, and Malkiewich--Merling provides constructions of equivariant algebraic $K$-theory $G$-spectra which serve this purpose \cite{Merling,Barwick,BGS,MM1}. 

These $G$-spectra have many interesting applications.  If $F\subset L$ is a Galois extension with Galois group $G$, then we can consider the $G$-equivariant algebraic $K$-groups of $L$. These groups tie together the algebraic $K$-groups of the fixed fields of $L$, and provide a collection of interesting operations which compare the $K$-groups of different fixed fields \cite{Merling}.  When $L$ is a number field, work of Elmanto--Zhang relates the sizes of the equivariant algebraic $K$-groups to special values of certain Artin L-functions \cite{ElmantoZhang}.  

In another direction, Malkiewich--Merling produce two versions of $A$-theory associated to a space $X$ with action by a finite group $G$. Both of these constructions produce a genuine $G$-spectrum whose underlying spectrum is Waldhausen's $A(X)$. The first is denoted by $A_G(X)$, known as genuine $A$-theory, and is expected to fit into the equivariant parametrized $h$-cobordism theorem \cite{MM1,MM2,GoodwillieIgusaMalkiewichMerling}.  The second version of equivariant $A$-theory is called \emph{coarse} $A$-theory, denoted $A_G^{\mathrm{coarse}}(X)$.  The fixed points of this spectrum can be interpreted using the \emph{bivariant $A$-theory}  of Williams \cite{Williams,MM1}.  

One obstacle to further understanding equivariant algebraic $K$-theory is limited computational tools. Recent work of Elmanto--Zhang \cite{ElmantoZhang} and Chan--Vogeli \cite{ChanVogeli} makes significant computational advances in this area, however these approaches have only been applied to the case of field extensions. In a more general setting, it would be advantageous to have a trace methods approach to the study of equivariant algebraic $K$-theory. 

In recent years, several equivariant versions of topological Hochschild homology have emerged, including Real topological Hochschild homology (THR), twisted topological Hochschild homology ($\THH_{C_n}$), and quaternionic topological Hochschild homology (THQ) \cite{DMPR, AnBlGeHiLaMa, AKMP}. Twisted topological Hochschild homology, constructed by Angeltveit, Blumberg, Gerhardt, Hill, Lawson, and Mandell \cite{AnBlGeHiLaMa}, takes as input a $C_{n}$-ring spectrum. In \cite{aghkk:shadows}, the authors show that for a $C_n$-ring spectrum $R$ there is a non-equivariant map from the fixed points of equivariant algebraic $K$-theory to twisted THH:
\[
K_{C_n}(R)^{C_n} \to \THH_{C_n}(R).
\]
It is conjectured in \cite{aghkk:shadows} that this map arises as the fixed points of an equivariant map from the equivariant algebraic $K$-theory $K_{C_n}(R)$ to another equivariant form of topological Hochschild homology. In this work we prove a version of this conjecture, constructing the appropriate equivariant version of THH, as well as an equivariant Dennis trace map. 

Our approach to this conjecture is motivated by considering Hochschild constructions as equivariant norms. In the course of their work on the Kervaire invariant one problem, Hill, Hopkins, and Ravenel \cite{HHR} define multiplicative norm functors $N_H^G$ from $H$-spectra to $G$-spectra, for $H \leq G$ finite groups. These functors now play a central role in equivariant homotopy theory and its applications. While the work in \cite{HHR} is restricted to finite groups, norms to compact Lie groups have been constructed in special cases in, for instance, \cite{AnBlGeHiLaMa, AKGH}. A more general framework for norms in the compact Lie group setting is considered in \cite{BlumbergHillMandellNorms}.

In \cite{AnBlGeHiLaMa}, it is shown that the ordinary THH of a ring spectrum $A$ can be viewed as the norm from the trivial group to $S^1$, $\THH(A) = N_e^{S^1}(A)$. Further, twisted THH for $C_n$-ring spectra also arises as an equivariant norm. Indeed, for a $C_n$-ring spectrum $R$, $\THH_{C_n}(R) = N_{C_n}^{S^1}(R).$ In the current work we define a norm from $G$-ring spectra to $G \times S^1$-spectra for a finite group $G$. We take this norm as a definition of a new equivariant form of THH: \[
\ETHH(R) := N_G^{G \times S^1}(R). 
\]

For finite groups, the computability of the norm functors is aided by their compatibility with the geometric fixed points functor $\Phi^K$ for subgroups $K\leq G$ \cite{HHR,AnBlGeHiLaMa}. An analogous relationship between fixed points and norms is expected to hold in the setting of compact Lie groups \cite{BlumbergHillMandellNorms}. We prove a special case, giving a relationship between ETHH and twisted THH.
 \begin{thm}[{\autoref{thm: geometric fixed points of ETHH is twisted THH for rings}}]\label{theorem: intro theorem twisted THH is fixed points of ETHH}
    Let $R$ be a cofibrant $C_{n}$-ring spectrum, and let $\Delta_{C_n} \leq C_n \times S^1$ denote the diagonal subgroup isomorphic to $C_n$. 
 There is an isomorphism of orthogonal $C_n$-spectra
    \[
        \Phi^{\Delta_{C_n}}\ETHH(R) \cong \THH_{C_n}(R).
    \]
\end{thm}
We also prove some foundational properties for ETHH. 

\begin{thm}[{\autoref{theorem: Morita invariance} and \autoref{cor-ETHH-additivity}}]
    Equivariant topological Hochschild homology is Morita invariant and satisfies additivity.
\end{thm}

Having established these properties, we then turn to our main objective: developing the foundations of trace methods for equivariant algebraic $K$-theory. The first step in this program is to define a trace map from equivariant algebraic $K$-theory to an equivariant form of topological Hochschild homology, lifting the non-equivariant trace map in \cite{aghkk:shadows}
\[
    K_{C_n}(R)^{C_n}\to \THH_{C_n}(R).
\]
 We use \autoref{theorem: intro theorem twisted THH is fixed points of ETHH} to show that this trace map can be lifted to a map of spectra with $G$-action. We denote by $\eTHH$ a version of $\ETHH$ indexed on a trivial universe, rather than a complete one. 
\begin{thm}[{\autoref{cor-trace-R}}]\label{tracethm}
 Let $R$ be a cofibrant ring $G$-spectrum.  Then the ordinary Dennis trace refines to an equivariant Dennis trace in the homotopy category of spectra with $G$-action 
 \[
 tr: \mathcal{I}_{\mathcal{V}} ^{\R^\infty} K_G(R) \to \eTHH(R).
 \]
 where $\mathcal{V}$ denotes a complete $G$-universe.
\end{thm}

Forthcoming work of Marc Gotliboym addresses cyclotomic structures for equivariant topological Hochschild homology, and the development of equivariant topological cyclic homology. 
 
 When $G=C_n$, this equivariant Dennis trace yields the desired lift of the trace map from \cite{aghkk:shadows}.

\begin{thm}[{\autoref{thm-trace-to-twisted}}]\label{thm-trace-to-twisted-intro}
    Let $R$ be a cofibrant $C_n$-ring spectrum. Upon taking fixed points of the equivariant Dennis trace map of \autoref{tracethm}, we obtain the trace map 
    $$K_{C_n}(R)^{C_n} \to \THH_{C_n}(R)$$
    of \cite{aghkk:shadows}.
\end{thm}

It is natural to wonder about whether the trace map of \autoref{tracethm} refines to a map of genuine $G$-spectra $K_G(R)\to \ETHH(R)$.  Changing universes to genuine $G$-spectra, this is a question about whether the derived counit 
\[
     \mathcal{I}^{\mathcal{V}}_{\mathbb{R}^{\infty}}(\mathcal{I}_{\mathcal{V}} ^{\R^\infty} K_G(R))^c\xrightarrow{\epsilon} K_G(R)
\]
admits a section, where $(-)^c$ denotes cofibrant replacement in the category of spectra with $G$-action.  After taking fixed points, the map $\epsilon$ admits a canonical section, and thus on fixed point there exist trace maps
\[
    K_G(R)^H\to \ETHH(R)^H
\]
for any subgroup $H\leq G$; see \autoref{remark: trace on genuine fixed points}.  On the other hand, in \autoref{example: no section} we show that that the map $\epsilon$ need not admit a section as genuine $G$-spectra; indeed it almost never will.  To summarize, while we can construct trace maps between the fixed points of equivariant algebraic $K$-theory and the fixed points of $\ETHH(R)$, this approach does not refine all the way to a map of genuine $G$-spectra. Since the equivariant algebraic $K$-theory of a $G$-ring spectrum depends only on the category of perfect modules over the underlying ring with $G$-action, it is not surprising that it does not naturally compare to $\ETHH$, which is a construction that takes place in genuine $G$-spectra. 

As an application, we show that the equivariant Dennis trace map yields trace maps from the fixed points of both forms of equivariant $A$-theory to the free loop space. 
\begin{prop}[\autoref{prop-trace-A-loop}]
    Let $H \leq G$. The equivariant Dennis trace map gives trace maps
\[
    (A^{\mathrm{coarse}}_G(X))^H \to \Sigma^{\infty}(\mathcal{L}X^H)_+ \hspace{.5cm} \textrm{and} \hspace{.5cm} A_G(X)^H\to \Sigma^{\infty}(\mathcal{L}X^H)_+.
\] 
in the homotopy category of spectra. 
\end{prop}

\subsection{Organization}
In \autoref{sec:background} we recall necessary background, reviewing some foundations for equivariant homotopy theory, equivariant algebraic $K$-theory, twisted topological Hochschild homology, and spectral categories. In \autoref{sec:ETHH} we define equivariant topological Hochschild homology for $G$-ring spectra, and show that for a $C_n$-ring spectrum, twisted THH can be recovered as the geometric fixed points of ETHH. We define $G$-spectral categories in \autoref{sec-GSp}, and define ETHH of a $G$-spectral category. In \autoref{sec:SpecWaldG} we develop a theory of spectral Waldhausen $G$-categories. In \autoref{sec: Morita} we prove that ETHH is Morita invariant, and establish a Morita adjunction which will be essential to the definition of our equivariant trace map. Another key ingredient for the equivariant trace is the additivity of equivariant topological Hochschild homology, which is established in \autoref{sec:Additivity}. Finally, in \autoref{sec:DennisTrace}, we define the equivariant Dennis trace, and prove that it lifts the non-equivariant trace map from \cite{aghkk:shadows}. We also consider applications of the equivariant Dennis trace to equivariant $A$-theory.

\subsection{Acknowledgments} The authors would like to thank Mike Hill, Cary Malkiewich, Michael Mandell, Andres Mejia, Mona Merling, and Maximilien P\'eroux for helpful conversations. The first author was supported by NSF grant DMS-2135960. The second author was supported by NSF grants DMS-2104233 and DMS-2404932. This work was also partially supported by a grant from the Simons Foundation. The authors would like to thank the Isaac Newton Institute for Mathematical Sciences, Cambridge, for support and hospitality during the programme ``Equivariant homotopy theory in context" where work on this paper was undertaken. This work was supported by EPSRC grant no EP/R014604/1. This material is also in part based on work supported by
the National Science Foundation under DMS-1928930, while the authors were in residence at the Simons Laufer Mathematical Sciences Institute (previously
known as MSRI) in Berkeley, California, during the Fall 2022 semester.

\section{Background}\label{sec:background}

\subsection{Equivariant homotopy theory and norms}\label{sec:norms}

In this section we recall the basic homotopy theory of orthogonal $G$-spectra.  We refer the reader to \cite{mandell-may} or \cite[Section 2]{AnBlGeHiLaMa} for more details.  Throughout we fix a compact Lie group $G$.

\begin{defn}
    A \emph{$G$-universe} is a countably infinite dimensional real orthogonal $G$-representation $U$ such that
    \begin{enumerate}
        \item the $G$-fixed points of $U$ are infinite dimensional,
        \item every irreducible real orthogonal representation of $G$ which is contained in $U$ is contained infinitely many times.
    \end{enumerate}
    We call a $G$-universe $U$ \emph{trivial} if the $G$-action is trivial.  We call a $G$-universe $U$ \emph{complete} if it contains at least one copy of every irreducible real orthogonal $G$-representation.
\end{defn}

We write $\mathcal{T}_G$ for the category of based $G$-spaces and based equivariant maps.  This is a closed symmetric monoidal category with the smash product and internal mapping spaces. A map in $\mathcal{T}_G$ is a $G$-weak equivalence if the induced map on fixed points is a weak equivalence for all closed subgroups $H\leq G$.

We write $\mathcal{J}^U_G$ for the $\mathcal{T}_G$-enriched category whose objects are finite dimensional representations $V\subset U$ and morphisms $\mathcal{J}^U_G(V,W)$ given by the $G$-space of (possibly non-equivariant) linear isomorphisms $V\xrightarrow{\cong} W$, where the $G$-action is by conjugation.  To make these based $G$-spaces we freely add a disjoint base point to all morphism sets.

\begin{defn}
    An orthogonal $G$-spectrum, indexed on a $G$-universe $U$, is a functor of $\mathcal{T}_G$-enriched categories $F\colon \mathcal{J}_G^U\to \mathcal{T}_G$ together which choices of maps
    \[
        \sigma_{V,W}\colon F(V)\wedge S^W\to F(V\oplus W)
    \]
    which are natural, associative, and unital.  A morphism of orthogonal $G$-spectra is a natural transformation of functors which commutes the maps $\sigma_{V,W}$.  An orthogonal $G$-spectrum is an $\Omega$-$G$-spectrum if the maps
    \[
        F(V)\to \Omega^WF(V\oplus W)
    \]
    adjoint to the maps $\sigma_{V,W}$ are all $G$-weak equivalences.
    
\end{defn}

We write $\Sp^G_U$ for the category of orthogonal $G$-spectra indexed on the universe $U$.  When the universe $U$ is complete we will often omit it from the notation and write simply $\Sp^G$. When $G=e$ is the trivial group we have the category of orthogonal spectra which we denote by $\Sp$.  Note when $U = \mathbb{R}^{\infty}$ is the trivial $G$-universe that $\Sp^G_{\mathbb{R}^{\infty}}$ is isomorphic to the category of $G$-objects in $\Sp$. We refer to these as spectra with $G$-action.

For any universe $U$ the category $\Sp^G_U$ is closed symmetric monoidal.  We write $\wedge$ for the smash product, $F(-,-)$ for the internal hom, and $\mathbb{S}$ for the equivariant sphere spectrum which acts as the unit.  There is a suspension functor $\Sigma^{\infty}_U\colon \mathcal{T}_G\to \Sp^G_U$ which sends a $G$-space $X$ to the orthogonal spectrum $\Sigma^{\infty}_U(X)(V) = S^V\wedge X$.  This functor is strong monoidal; in particular $\mathbb{S} = \Sigma^{\infty}_U(S^0)$.

For two $G$-universes $U$ and $U'$ we translate between $U$-spectra and $U'$-spectra by means of the \emph{change of universe} functors 
\[
    \mathcal{I}_{U}^{U'}\colon \Sp^G_U\to \Sp^{G}_{U'}
\]
defined on objects by
\[
    \mathcal{I}_{U}^{U'}(F)(W) = \mathcal{J}^{U'}_G(\mathbb{R}^{n},W)\wedge_{O(n)} F(\mathbb{R}^{n}).
\]
where $n$ is the dimension of $W$.  The next lemma records the necessary facts about these functors.

\begin{lem}[{\cite[Theorem 1.5]{mandell-may}}]\label{lemma: change of universe facts}
    Let $U$, $V$, and $W$ be $G$-universes.  The change of universe functors satisfy the following:
    \begin{enumerate}
        \item there is a natural isomorphism $\mathcal{I}^V_U\circ \Sigma^{\infty}_U\cong \Sigma^{\infty}_V$,
        \item there is a natural isomorphism $\mathcal{I}^W_V\circ \mathcal{I}^V_U\cong \mathcal{I}^W_U$,
        \item $\mathcal{I}^U_U$ is isomorphic to the identity functor,
        \item the functor $\mathcal{I}_U^V$ is strong monoidal.
    \end{enumerate}
\end{lem}

\begin{rem}
    Properties (2) and (3) above imply the change of universe functors are equivalences of categories.  Nevertheless, each category of orthogonal $G$-spectra has a model structure for which the change of universe functors are \emph{not} Quillen equivalences.
\end{rem}

Suppose that $H\leq G$ is a closed subgroup.  For any $G$-universe $U$ we can forget the $G$-action to an $H$-action and obtain an $H$-universe $i^*_HU$.  Note that even when $U$ is complete $i^*_HU$ may fail to be complete.

\begin{defn}
    The \emph{restriction functor}
    \[
        i^*_H\colon \Sp^G_U\to \Sp^H_{i^*_HU}
    \]
    is given by \[i^*_H(F)(V) = \mathcal{J}_H^{i^*_HU}(\mathbb{R}^n,V)\wedge_{O(n)}i^*_HF(\mathbb{R}^n)\]
    where $n$ is the dimension of $V$.
\end{defn}

The restriction functor is strong monoidal and has both left and right adjoints.  In particular, restriction preserves both limits and colimits.  Often it will be convenient to assume that $i_H^*$ is landing in the category of $H$-spectra indexed on a complete universe.  Following \cite{mandell-may} we can do this by post composition with a change of universe, although we will omit it from the notation.

\subsubsection{Fixed points}  We will consider two kinds of fixed points functors in equivariant stable homotopy theory.  In both cases we will be particularly interested in fixed points with respect to a normal subgroup $H\trianglelefteq G$ so that $G/H$ is a group. In this case, for any $G$-universe $U$, the fixed points $U^H$ have an action by $G/H$ which makes $U^H$ into a $G/H$-universe.  Again, this may fail to be complete even if $U$ is complete.

To produce our fixed point functors we introduce an intermediate category $\mathcal{J}_G^H$ whose objects are sub-representations of a $G$-universe $U$ and whose mappings spaces are the $H$-fixed points of $\mathcal{J}_G(V,W)$.  The category $\mathcal{J}_G^H$ is naturally enriched in $\mathcal{T}^{G/H}$.  There are two enriched functors 
\begin{align*}
    \mathrm{Fix}^H & \colon \mathcal{J}^H_G\to \mathcal{J}_{G/H}^{U^H}\\
    \rho & \colon \mathcal{J}_{G/H}\to \mathcal{J}^H_G
\end{align*}
where $\mathrm{Fix}^H(V) = V^H$ and $\rho(W) = W$, treated as a $G$-space by pulling back along the quotient $G\twoheadrightarrow G/H$.

For any $F\in \Sp^G_U$, the collection $F(V)^H$ gives the objects of a functor $\mathcal{J}^H_G\to \mathcal{T}^{G/H}$.  We briefly denote this functor by $\widetilde{F}$.

\begin{defn}
    For a normal subgroup $H\trianglelefteq G$ the \emph{categorical fixed points} functor 
    \[
        (-)^H\colon \Sp^G_U\to \Sp^{G/H}_{U^H}
    \]
    sends $F\in \Sp^G_U$ to $\rho^*(\widetilde{F})$.  In particular, we have $F^H(V) = F(V)^H$.
\end{defn}

The categorical fixed points are a natural construction which are used to define equivariant homotopy groups.  On the other hand, when the universe $U$ is not trivial the categorical fixed points do not have the property of extending fixed points of $G$-spaces.  That is, $(\Sigma^{\infty}X)^H\ncong \Sigma^{\infty}(X^H)$ for a based $G$-space $X$.  The desire for this property is the motivation for the \emph{geometric fixed points}.

\begin{defn}
    For a normal subgroup $H\trianglelefteq G$ the \emph{geometric fixed points} functor
    \[
        \Phi^H\colon \Sp^G_{U}\to \Sp^{G/H}_{U^H}
    \]
    is given by composing $F\mapsto \widetilde{F}$ with left Kan extension along $\mathrm{Fix}^H$.
\end{defn}

\begin{prop}[{\cite[Proposition V.4.7]{mandell-may}, and \cite[Theorem A.1]{blumberg-mandell:cyclotomic}}]\label{prop: fixed points are monoidal}
    For $H$ a normal subgroup of $G$ both $(-)^H$ and $\Phi^H$ are lax monoidal functors. When either $X$ or $Y$ is a cofibrant $G$-spectrum (in the stable model structure, see below) the map $\Phi^H(X)\wedge \Phi^H(Y)\to \Phi^H(X\wedge Y)$ is an isomorphism.
\end{prop}
\begin{rem}\label{rem: geom=cat fixed points for trivial universe}
When $U = \mathbb{R}^{\infty}$ is the trivial $G$-universe note that $\mathcal{J}^H_G$, $\mathcal{J}_{G/H}$,  and $\mathcal{J}^{U^H}_{G/H}$ are the same category and the functors $\rho$ and $\mathrm{Fix}^H$ are both the identity.  In particular, the geometric and categorical fixed point functors agree for the trivial $G$-universe and are both strong monoidal always.
\end{rem}

\subsubsection{The norm}

In this subsection we restrict to $G$ a finite group.  In \cite{HHR} the authors develop the \emph{norm functor}, a powerful tool for understanding equivariant spectra.  We write
\[
    \wedge^G_H\colon \Sp^{H}_{\mathbb{R}^{\infty}}\to \Sp^{G}_{\mathbb{R}^{\infty}}
\]
for the functor which sends $X$ to 
\[
    \wedge_H^G(X) = \bigwedge\limits_{\gamma\in G/H} X
\]
where the product on the right is indexed by a choice of coset representatives.  The group $G$ acts first on the cosets $G/H$ and then has a residual $H$-action on the smash product terms.

\begin{defn}
    Let $G$ be a finite group and $H$ a subgroup.  Let $U$ be a $G$-universe and $V$ an $H$-universe.  The \emph{point-set norm functor} $N_H^G\colon \Sp^H_{V}\to \Sp^G_U$ is the composite
\[\Sp^H_{V}\xrightarrow{\mathcal{I}_V^{\mathbb{R}^{\infty}}}\Sp^{H}_{\mathbb{R}^{\infty}}\xrightarrow{\wedge_H^G} \Sp^{G}_{\mathbb{R}^{\infty}}\xrightarrow{\mathcal{I}_{\mathbb{R}^{\infty}}^U}\Sp^G_{U}.
    \]
\end{defn}

The norm functor is strong monoidal and, in particular, it restricts to a functor on commutative and associative algebras.  In the commutative case the norm is left adjoint to the forgetful functor.

\begin{prop}[{\cite[Appendix B]{HHR}}]\label{prop: norm is left adjoint finite}
    For any finite group $G$ and subgroup $H\leq G$ there is an adjunction 
    \[
            \begin{tikzcd}
                \mathrm{Comm}^H \ar[r,shift left,"N^G_H"] &  \ar[l,shift left,"i_H^*"]\mathrm{Comm}^G.
            \end{tikzcd}
        \]
\end{prop}

\subsubsection{Homotopy theory of orthogonal $G$-spectra}

In this subsection we outline the model structures at play in this paper.  To do so, we need to first define equivariant homotopy groups.  Throughout this section we fix a compact Lie group $G$ and a $G$-universe $U$.

\begin{defn}
    Let $X\in \Sp^G_U$.  For any $H\leq G$ and any integer $n$ we define 
    \[
        \pi^H_n(X) = \begin{cases}
            \colim\limits_{V< U} \pi_n((\Omega^VX(V))^H) & n\geq 0 \\
            \colim\limits_{\mathbb{R}^{-n}< V <U} \pi_0((\Omega^{V - \mathbb{R}^{-n}}X(V))^H) & n\leq 0 
        \end{cases}
    \]
    where in either case the colimit is over finite dimensional sub-representations $V<U$.
\end{defn}

As usual, a map of $G$-spectra induces a map on $\pi_n^H$ for all $n$ and $H$.  We call a map $X\to Y$ a \emph{stable equivalence} if it induces an isomorphism on $\pi_n^H$ for all $n$ and $H$.  

\begin{prop}[{\cite[Theorem III.4.2]{mandell-may}}]
    There is a symmetric monoidal compactly generated model structure, called the \emph{stable model structure}, on $\Sp^G_U$ with weak equivalences the stable equivalences. The fibrant objects are the $\Omega$-$G$-spectra.
\end{prop}
\begin{rem}
    When the universe $U = \mathbb{R}^{\infty}$ is a trivial $G$-universe recall that $\Sp^{G}_{\mathbb{R}^{\infty}}$ is isomorphic to the category $\Fun(BG,\Sp)$ of orthogonal spectra with $G$-action.  This category is often given the projective model structure, where weak equivalences are equivariant maps of spectra which are underlying equivalences.  With the projective model structure this is called the category of Borel $G$-spectra.  We will not make any use of Borel $G$-spectra in this paper.  In particular, when we say ``a weak equivalence of spectra with $G$-action'' we mean a map $X\to Y$ which induces equivalences of spectra $X^H\to Y^H$ for all $H\leq G$.
\end{rem}

The stable equivalences are also the weak equivalences in the \emph{positive complete stable model structure} for $\Sp^G_U$ \cite[\S B.4]{HHR}.  This model structure is convenient for working with multiplicative structures and especially the norm.  

\begin{prop}[{\cite[Proposition B.63]{HHR}}]
    There is a symmetric monoidal cofibrantly generated model structure, called the \emph{positive complete stable model structure}, on $\Sp^G_U$ with weak equivalences the stable equivalences.  All $\Omega$-$G$-spectra are fibrant.  Additionally, the positive complete stable model structure lifts to model structures on associative and commutative algebras in $\Sp^G_U$ with the same weak equivalences and fibrations.  
\end{prop}

When $X$ is a fibrant $G$-spectrum indexed on a complete $G$-universe we can also compute it's homotopy groups by taking fixed points.
\begin{prop}[{\cite[V.3.2]{mandell-may}}]
    Let $X$ be a $G$-spectrum indexed on a complete universe which is fibrant in either the stable or positive complete stable model structure.  Then there is an isomorphism $\pi_n^H(X)\cong \pi_n(X^H)$.
\end{prop}

Whenever we say that a ring orthogonal $G$-spectrum is cofibrant we mean that is cofibrant in the positive complete stable model structure. This implies that the underlying $G$-spectrum is cofibrant in the stable model structure.

The following propositions summarize all the necessary facts we need about the interplay between the change of universe, fixed points, and norm functors and the homotopy theory of equivariant spectra. The first proposition tells us that the adjunction of \autoref{prop: norm is left adjoint finite} is a Quillen adjunction.

\begin{prop}\cite[Theorem 10.2.4]{HHR:Book}\label{prop: norm is left quillen adjoint finite} 
    Let $G$ be a finite group, $H\leq G$ a subgroup, $U$ a complete $G$-universe, and $V$ a complete $H$-universe. Then the norm $N_H^G\colon \Sp^H_V\to \Sp^G_U$ preserves cofibrations and weak equivalences between cofibrant objects.  There is Quillen adjunction 
        \[
            \begin{tikzcd}
                \mathrm{Comm}^H \ar[r,shift left,"N^G_H"] &  \ar[l,shift left,"i_H^*"]\mathrm{Comm}^G.
            \end{tikzcd}
        \]
\end{prop}
\begin{prop}[{\cite{mandell-may}}]
    For any $G$-universes $U'\leq U$ there is a Quillen adjunction
        \[
            \begin{tikzcd}
                \Sp^G_{U'} \ar[r,shift left,"\mathcal{I}^U_{U'}"] &  \ar[l,shift left,"\mathcal{I}^{U'}_U"]\Sp^G_U
            \end{tikzcd}.
        \]
\end{prop}
The following proposition captures the interplay between geometric fixed points and equivariant norms. 
\begin{prop}[{\cite[Appendix B]{HHR}}]\label{lem: norm and fixed point lemmas}
    Let $G$ be a finite group, $H\leq G$ a subgroup, $U$ a $G$-universe, and $V$ an $H$-universe.  With the positive complete stable model structures, we have:
    \begin{enumerate}
        \item when $H$ is normal, $\Phi^H\colon \Sp^G_U\to \Sp^{G/H}_{U^H}$ preserves weak equivalences and cofibrations between cofibrant objects,
        \item for a cofibrant $H$-spectrum $X$, there is an isomorphism $\Phi^HX\cong \Phi^GN^G_HX$,
        \item if $X$ is a cofibrant commutative orthogonal ring spectrum then there is an isomorphism $X\cong \Phi^GN^G_eX$.
       \item \cite[Proposition 9.11.11]{HHR:Book} if $f\colon X\to Y$ is a map of $G$-spectra such that $\Phi^H(f)\colon \Phi^H(X)\to \Phi^H(Y)$ is a stable equivalence of orthogonal spectra for all $H\leq G$ then $f$ is a stable equivalence of orthogonal $G$-spectra,
                \end{enumerate}
        \end{prop}
       
        We also recall the following result of Lewis. 
        \begin{prop}[{\cite[Corollary 2.5(b)]{Lewis:Splitting}}]\label{prop:LewisSplit} If $X$ is a cofibrant orthogonal $G$-spectrum indexed on a trivial universe $\mathbb{R}^{\infty}$ and $V$ is a complete $G$-universe, then there is a stable equivalence of non-equivariant spectra
        \[
            (\mathcal{I}_{\mathbb{R}^{\infty}}^V(X))^G \simeq  \bigvee_{(H)\leq G} (X^H)_{hW_G(H)}
        \]
        where $(H)$ runs over the conjugacy classes of subgroups of $G$ and $W_G(H) = N_G(H)/H$ is the Weyl group.
   \end{prop}

\begin{rem}\label{remark: fixed points are retract of change universe fixed points}
    In \autoref{prop:LewisSplit} above, note that the summand corresponding to $H=G$ is precisely $X^G$.  Hence, in the stable homotopy category, $X^G$ is always a retract of $(\mathcal{I}_{\mathbb{R}^{\infty}}^V(X))^G$.
\end{rem}
Finally we relate change of universe and fixed point constructions.
\begin{prop}\label{proposition: geometric fixed points and change of universe}
    Let $H\leq G$ a normal group, and let $U$ be a complete $G$-universe.  If $X$ is any connected spectrum with $G$-action there is an isomorphism of $G/H$-spectra
    \[
        \Phi^H(\mathcal{I}^{U}_{\mathbb{R}^{\infty}}(X)) \cong \mathcal{I}^{U^H}_{\mathbb{R}^{\infty}}(X^H)
    \]
    indexed on the universe $U^H$.
\end{prop}
\begin{proof}
    First, note that this result is true when $X$ is a suspension spectrum with $G$-action, because geometric fixed points commute with suspension spectra.  Note also that in the case of suspension spectra this isomorphism is natural in maps of $G$-spaces.  The general case follows because every connected spectrum is a colimit of suspension spectra, and change of universe and fixed points commute with these colimits.
\end{proof}

We end with a brief discussion of homotopy groups and the change of universe functor.  Let $\Set^G$ denote the category of finite $G$-sets for a finite group $G$.
\begin{defn}
    A \emph{$G$-coefficient system} is a product preserving functor $F\colon (\Set^G)^{\op}\to \mathrm{Ab}$.  We write $\mathrm{Coeff}_G$ for the category of $G$-coefficient systems and natural transformations.
\end{defn}
Since the category $\Set^G$ is generated under coproducts by the orbits $G/H$, the category $(\Set^G)^{\op}$ is generated by the orbits $G/H$ under products.  Thus a coefficient system is entirely determined, up to isomorphism, by its values on the orbits $G/H$ and the maps between them.

\begin{prop}
    If $X$ is an orthogonal $G$-spectrum indexed on a trivial $G$-universe then the assignment $G/H\mapsto \pi_n^H(X)$ admits the structure of a coefficient system which we denote by $\underline{\pi}_n(X)$.
\end{prop}

We obtain more structure when the universe is complete.  Let $\mathcal{A}^G = \mathrm{Span}(\Set^G)$ be the category of spans in $\Set^G$.  It has the same objects as $\Set^G$, but morphisms are given by isomorphism classes of spans $[X\leftarrow A\rightarrow Y]$.  Composition is given by pullback of spans. The Burnside category has bi-products, given on objects by disjoint union of finite $G$-sets.

\begin{defn}
    A Mackey functor is a product preserving functor $F\colon \mathcal{A}^G\to \mathrm{Ab}$.  We write $\Mack$ for the category of $G$-Mackey functors and natural transformations.
\end{defn}
\begin{prop}
    If $X$ is an orthogonal $G$-spectrum indexed on a complete $G$-universe then the assignment $G/H\mapsto \pi^H_n(X)$ is a $G$-Mackey functor which we denote by $\underline{\pi}_n(X)$.
\end{prop}

There is an embedding $j\colon (\Set^G)^{\op}\to \mathcal{A}^G$ which is the identity on objects and sends a map $f\colon X\to Y $ in $\Set^G$ to the span $[Y\xleftarrow{f} X\xrightarrow{=} X]$.  Since this functor preserves products, we see that any Mackey functor $M$ determines a coefficient system $j^*(M)$.  The functor $j^*\colon \Mack\to \mathrm{Coeff}$ has a left adjoint $j_!\colon \mathrm{Coeff}\to \Mack$ given by left Kan extension.
\begin{exmp}\label{example: free Mackey functor}
    Let $M$ be the $C_2$-coefficient system which has $M(C_2/C_2) = M(C_2/e)=\mathbb{Z}$, with trivial conjugation and identity restriction map.  Using the universal mapping property of $j_!$, one sees that $j_!(M)$ can be depicted by the diagram
\[\begin{tikzcd}
	{\Z^2} \\
	\Z
	\arrow["{(1,2)}"', shift right, from=1-1, to=2-1]
	\arrow["{i_2}"', shift right, from=2-1, to=1-1]
\end{tikzcd}\]
where $j_!(M)(C_2/C_2) = \mathbb{Z}^2$,  $j_!(M)(C_2/e) =\mathbb{Z}$, and $i_2\colon \mathbb{Z}\to \mathbb{Z}^2$ is the inclusion of the second component.  
\end{exmp}

\begin{prop}\label{proposition: computing change of universe on pi zero}
    If $X$ is cofibrant connective $G$-spectrum indexed on a trivial $G$-universe $\R^\infty$ then
    \[
        j_!(\underline{\pi}_0(X))\cong \underline{\pi}_0(\mathcal{I}^{V}_{\mathbb{R}^{\infty}}(X)).
    \]
    If $Y$ is any fibrant connective $G$-spectrum indexed on a complete $G$-universe $V$ then 
    \[
        j^*(\underline{\pi}_n(X))\cong \underline{\pi}_n(\mathcal{I}^{\mathbb{R}^{\infty}}_V(X))
    \]
    for all $n\geq 0$.
\end{prop}
\begin{proof}
    The homotopy category of $G$-spectra indexed on a complete $G$-universe can be presented as the homotopy category of the $\infty$-category of functors $\Fun(\mathcal{A}^{G}_{\mathrm{eff}},\Sp)$, where $\mathcal{A}^G_{\mathrm{eff}}$ is the \emph{effective} Burnside $\infty$-category of \cite{Barwick}; see also \cite{Guilloumay,CMNN}. Similarly, the homotopy category of $G$-spectra indexed on a trivial $G$-universe can be presented as the homotopy category of the $\infty$-category of functors $\Fun((\Set^G)^{\op},\Sp)$. Both of these identifications restrict to identifications on connective objects.
    
    The effective Burnside $\infty$-category has the property that its homotopy category is equivalent to the ordinary category $\mathcal{A}^G$ \cite[3.8]{Barwick}. The functor $j\colon (\Set^G)^{\op}\to \mathcal{A}^G$ can be refined to a map of $\infty$-categories $J\colon N_{\bullet}((\Set^G)^{\op})\to \mathcal{A}_{\mathrm{eff}}^G$, where $N_{\bullet}$ denotes the nerve, such that $J$ induces $j$ on homotopy categories. 

    The derived change of universe functor is identified with the map induced on homotopy categories by left Kan extension along $J$.  Thus its zero truncation can be identified with left Kan extension on homotopy categories, which is left Kan extension along $j$, which is precisely the first claim.  The second statement follows from the fact that categorical fixed points and change of universe $\mathcal{I}^{\mathbb{R}^{\infty}}_\mathcal{V}$ commute on fibrant spectra since $\mathcal{I}^{\mathbb{R}^{\infty}}_\mathcal{V}$ is a right Quillen adjoint.
\end{proof}

\subsection{Equivariant algebraic $K$-theory}

We begin by reviewing Waldhausen's construction of algebraic $K$-theory.  Recall that a Waldhausen category consists of a category $\cC$ with a distinguished zero object and two chosen classes of morphisms, called weak equivalences and cofibrations.  These collections of morphisms must both be closed under composition and contain all isomorphisms. We require that the unique map $0\hookrightarrow A$ is always a cofibration. Finally, given a span $X\xleftarrow{r} A\xrightarrow{t} Y$ where $t$ is a cofibration we must have that the pushout $X\cup_A Y$ exists and the map $X\to X\cup_A Y$ is a cofibration.  

A functor $F\colon \cC\to \cD$ between Waldhausen categories is \emph{exact} if it preserves the zero object, cofibrations, weak equivalences, and pushouts along cofibrations. We write $\mathrm{Wald}$ for the category of small Waldhausen categories and exact functors.

Given a Waldhausen category $\cC$, the $S_{\bullet}$ construction of $\cC$ is the simplicial set whose $k$-simplices are composites
\[
    A_{0,1}\hookrightarrow A_{0,2}\hookrightarrow\dots A_{0,k}
\]
in $\cC$ where each map is a cofibration, together with choices of pushout squares
\[
\begin{tikzcd}
    A_{0,i}\ar[hook, r] \ar[d] & A_{0,j} \ar[d]\\
    0 \ar[r] & A_{i,j}
    \end{tikzcd}
\]
For instance, we can visualize an element in $S_{3}\cC$ as a grid
\[\begin{tikzcd}
	0 & {A_{0,1}} & {A_{0,2}} & {A_{0,3}} \\
	& 0 & {A_{1,2}} & {A_{1,3}} \\
	&& 0 & {A_{2,3}} \\
	&&& 0
	\arrow[hook, from=1-1, to=1-2]
	\arrow[hook, from=1-2, to=1-3]
	\arrow[from=1-2, to=2-2]
	\arrow[hook, from=1-3, to=1-4]
	\arrow[from=1-3, to=2-3]
	\arrow[from=1-4, to=2-4]
	\arrow[from=2-2, to=2-3]
	\arrow[hook, from=2-3, to=2-4]
	\arrow[from=2-3, to=3-3]
	\arrow[from=2-4, to=3-4]
	\arrow[from=3-3, to=3-4]
	\arrow[from=3-4, to=4-4]
\end{tikzcd}\]
where any rectangle with top face in the first row and bottom left equal to $0$ is a pushout.  The face maps $d_i\colon S_{n}(\cC)\to S_{n-1}(\cC)$ for $i>0$ are given by deleting the appropriate column and composing horizontal maps. The zeroth face map deletes the first row. 
 The degeneracy maps are given by inserting an extra column which is equal to the one which precedes it.  

For any fixed $k$, $S_k(\cC)$ is not just a set but a category. The morphisms are given by collections of map $A_{i,j}\to A_{i,j}'$ which makes the evident diagram of grids commute.  In fact, $S_{k}(\cC)$ is itself a Waldhausen category where the weak equivalences and cofibrations are just collections of maps $A_{i,j}\to A_{i,j}'$ where each constituent map is either a weak equivalence or a cofibration.  We write $wS_{\bullet}(\cC)$ for the simplicial category where $\mathrm{ob}(wS_k(\cC)) = \mathrm{ob}(S_k(\cC))$ but with only weak equivalences as morphisms.  
\begin{defn}
    The algebraic $K$-theory of a Waldhausen category $\cC$ is $K(\cC) = \Omega |wS_{\bullet}(\cC)|$.
\end{defn}
The algebraic $K$-theory of a Waldhausen category is a group-like $E_{\infty}$-space, hence determines a connective orthogonal spectrum via one of the many equivalent delooping machines \cite{MayThomason}.

 The $S_{\bullet}$ construction is functorial in exact functors, hence an exact functor induces a map of $K$-theory spectra $K(F)\colon K(\cC)\to K(\cD)$.

We will also need a variant of the $S_{\bullet}$ construction which produces a symmetric spectrum instead of an orthogonal spectrum.  Since for any fixed $k$ the category $S_{k}(\cC)$ is a Waldhausen category, we can apply the $S_{\bullet}$-construction to obtain a simplicial Waldhausen category $S_{\bullet}S_{k}(\cC)$.  Allowing $k$ to vary as well, we obtain a bisimplicial category $S_{\bullet,\bullet}(\cC)$.  One can repeat this procedure, and for all $n$ we obtain an $n$-simplicial category $S^{(n)}_{\bullet,\dots,\bullet}(\cC)$ called the iterated $S_{\bullet}$-construction. For any $n$, $S^{(n)}_{\bullet,\dots,\bullet}(\cC)$ has an action from the symmetric group $\Sigma_n$ obtained via permuting the simplicial indices.

If $\mathcal{C}$ is any Waldhausen category we can also define a simplicial Waldhausen category $w_{\bullet}\cC$ whose $k$-simplices are the $k$-chains of weak equivalences in $\cC$; morphisms in this category are isomorphisms of chains. 
\begin{defn}\label{definition: iterated S dot}
    The symmetric spectrum $K$-theory of a Waldhausen category $\cC$ is the symmetric spectrum whose $n$-space is $K^{\mathrm{sym}}(\cC) = \Omega|\mathrm{diag}(w_{\bullet}S^{(n)}_{\bullet,\dots,\bullet}(\cC))|$.
\end{defn}
The categories of orthogonal and symmetric spectra, with their respective stable model structures, are related by a Quillen equivalence
\[
    \begin{tikzcd}
        \Sp^{\Sigma} \ar[r,shift left,"\mathbb{P}"] & \Sp \ar[l,shift left,"\mathbb{U}"]
    \end{tikzcd}
\]
and these can be used to compare the two models of $K$-theory of a Waldhausen category.
\begin{prop}[{\cite[Theorem 3.11]{Malkiewich:Coassembly}}]
    Let $\cC$ be any Waldhausen category.  There is a zig-zag of stable equivalences between $K(\cC)$ and $\mathbb{P}K^{\mathrm{sym}}(\cC)$.
\end{prop}

We now recall Malkiewich--Merling's definition of equivariant algebraic $K$-theory from \cite{MM1}. For a finite group $G$ we let $BG$ denote the one object groupoid of $G$, and let $\mathrm{Wald}$ denote the category of Waldhausen categories and exact functors.
\begin{defn}
    A \emph{Waldhausen $G$-category} is a a functor $F\colon BG\to \mathrm{Wald}$.
\end{defn}
Explicitly, a Waldhausen $G$-category consists of a Waldhausen category $\cC$ together with exact functors $g\colon \cC\to \cC$ for all $g\in G$ such that $h\circ g = hg$ and the map associated to the unit in $G$ is the identity functor.

\begin{defn}\label{defn: EG}
    Let $\mathcal{E}G$ denote the category with object set $G$ and a unique morphism
    \[
        h_g\colon g\to hg
    \]
    for all $g,h\in G$.  This category has a natural right action by $G$ given by $g\cdot k = gk$.
\end{defn}

If $\cC$ is a Waldhausen $G$-category then the category of all functors and natural transformations $\Fun(\EG,\mathcal{C})$ is a Waldhausen $G$-category where the $G$-action is via conjugation of functors: $(g\cdot F)(x) = gF(g^{-1}x)$. Since this is still a Waldhausen category, we can apply the $wS_{\bullet}$-construction to obtain a $G$-space.

\begin{defn}
    The equivariant algebraic $K$-theory space of a Waldhausen $G$-category $\mathcal{C}$ is the $G$-space 
    \[
        K_G(\mathcal{C}) := \Omega |wS_{\bullet} \Fun(\EG,\mathcal{C})|.
    \]
\end{defn}
Malkiewich and Merling show this space is an infinite loop $G$-space, meaning that is admits deloopings for all real orthogonal $G$-representations.

\begin{thm}[{\cite[Theorem 2.21]{MM1}}]
    For any Waldhausen $G$-category $\mathcal{C}$, the equivariant algebraic $K$-theory $K_G(\mathcal{C})$ is an infinite loop $G$-space.
\end{thm}

In particular, the equivariant algebraic $K$-theory determines a genuine $G$-spectrum.  

\begin{defn}[Malkiewich--Merling]
    The equivariant algebraic $K$-theory of a $G$-Waldhausen category $\cC$ is the genuine orthogonal $G$-spectrum determined, using the Guillou--May machine \cite{GuillouMay:IteratedLoopSpaceTheory}, by the infinite loop $G$-space $K_G(\cC)$.  We will denote this spectrum also by $K_G(\cC)$, it will be clear from context whether we are discussing a space or a genuine orthogonal $G$-spectrum.
\end{defn}

\subsection{Twisted topological Hochschild homology}

Classically, topological Hochschild homology plays a key role in the trace method approach to algebraic $K$-theory. In this section we recall the definition of topological Hochschild homology (THH) for a ring spectrum, as well as an equivariant variant: twisted topological Hochschild homology for $C_n$-equivariant ring spectra, $\THH_{C_n}(R)$. In Section \ref{sec:ETHH} we will develop another equivariant version of THH, denoted ETHH, which will also be crucial to our trace method approach in the equivariant setting. 

Classically, topological Hochschild homology is defined via a cyclic bar construction. 

\begin{defn}
 Let $R$ be an associative orthogonal ring spectrum, and $M$ an $(R,R)$-bimodule. The \emph{cyclic bar construction with coefficients} $\Ncyc_{\bullet}(R;M)$, is a simplicial spectrum which in degree $q$ is $R^{\wedge q} \wedge M.$ 

 Let $\rho\colon R\wedge M \to M$ denote the left module structure map for $M$, and $\psi: M\wedge R \to M$ denote the right module structure map. Then the face and degeneracy maps of $\Ncyc_{\bullet}(R;M)$ are given as follows: 
 $$
d_i = \left\{\begin{array}{ll} 
(\Id^{\wedge (q-1)} \wedge \psi) \circ \tau &: i=0 \\
\Id^{\wedge i} \wedge \mu \wedge \Id^{\wedge (q-i-1)} &: 0< i<q \\ \Id^{\wedge (q-1)} \wedge \rho &: i=q,
\end{array}\right.
$$
and 
$$
s_i =  \Id^{\wedge (i+1)} \wedge \eta \wedge \Id^{\wedge (q-i)} \hspace{1cm}:  0\leq i \leq q.
$$
Here, $\tau$ rotates the first factor of $R$ to the end. This yields a simplicial orthogonal spectrum $\Ncyc_{\bullet}(R;M)$. 
\end{defn}

\begin{rmk}
The usual convention for the cyclic bar construction is to place the $R$-bimodule $M$ on the far left i.e.\ the $q$ simplices are $M\wedge R^{\wedge q}$.  Our choice to put the bimodule on the right instead of the left is to keep our notation consistent with that of \cite{BlumbergMandell:localizationLongOne}.  There, Blumberg--Mandell consider an extension of the cyclic bar construction to spectral categories where the multiplication maps used in the definition of the face maps are replaced by the categorical composition maps.  In this setting, the choice to put the bimodule on the right is in line with the usual convention that function composition is read right-to-left instead of left-to-right.
\end{rmk}

When $M=R$, considered as an $(R,R)$-bimodule, the cyclic bar construction is denoted $\Ncyc_{\bullet}(R)$. Note that in this case the simplicial spectrum $\Ncyc_{\bullet}(R)$ is further a \emph{cyclic} spectrum, and thus upon realization has an action of $S^1$. 

\begin{defn}
For an associative ring spectrum $R$ and an $(R, R)$-bimodule $M$, the \emph{topological Hochschild homology} of $R$ with coefficients in $M$ is
\[
\THH(R;M) := |\Ncyc_{\bullet}(R;M)|.  
\]
When $M=R$ we write 
\[
\THH(R) : = |\Ncyc_{\bullet}(R)|.
\]

\end{defn}

There is a also perspective on topological Hochschild homology via equivariant norms. As discussed in Section \ref{sec:norms}, Hill, Hopkins, and Ravenel developed multiplicative norm functors $N_H^G$ from $H$-spectra to $G$-spectra for $H\leq G$ finite groups \cite{HHR}. In \cite{AnBlGeHiLaMa}, the authors extend this norm construction from finite groups to $S^1$, defining a functor $N_e^{S^1}$ and showing that the cyclic bar construction is a model for this equivariant norm. In particular for a ring spectrum $R$ and a complete $S^1$-universe $U$,
\[ N_e^{S^1}(R):=\mathcal{I}_{\mathbb{R}^{\infty}}^U|\Ncyc_{\bullet}(R)|.
\]
Thus, topological Hochschild homology is an equivariant norm. This norm model of THH lends itself to generalizations. Indeed, in \cite{AnBlGeHiLaMa} the authors define a notion of $C_n$-twisted topological Hochschild homology of $C_n$-ring spectra. For $R$ a $C_n$-ring spectrum, 
\[
\THH_{C_n}(R):= N_{C_n}^{S^1}(R). 
\]
There is a model for the norm $N_{C_n}^{S^1}$ as a twisted cyclic bar construction. 

\begin{defn}
Let $R$ be an associative orthogonal $C_n$-ring spectrum indexed on the trivial universe $\mathbb{R}^{\infty}$ and let $\sigma$ denote the generator $e^{2\pi i/n}$ of $C_n$. The $C_n$-\emph{twisted cyclic bar construction} $\mathrm{N}^{\cyc, C_n}_{\bullet}(R)$, is a simplicial $C_n$-spectrum which in degree $q$ is $R^{\wedge (q+1)}.$ 

 Let $\alpha_q: R^{\wedge (q+1)}\to R^{\wedge (q+1)}$ be the map which rotates the first factor to the end, and acts on the new last factor by $\sigma^{-1}$. Then the face and degeneracy maps of $\mathrm{N}^{\cyc, C_n}_{\bullet}(R)$ are given as follows: 
 $$
d_i = \left\{\begin{array}{ll} 
(\Id^{\wedge (q-1)} \wedge \mu) \circ \alpha_q &: i=0 \\
\Id^{\wedge i} \wedge \mu \wedge \Id^{\wedge (q-i-1)} &: 0< i\leq q,
\end{array}\right.
$$
and 
$$
s_i =  \Id^{\wedge (i+1)} \wedge \eta \wedge \Id^{\wedge (q-i)} \hspace{1cm}:  0\leq i \leq q.
$$
This yields a simplicial object $\Ncyc_{\bullet}(R)$. While this is not a cyclic object, it is a $\Lambda_n^{op}$-object, in the sense of B\"okstedt-Hsiang-Madsen \cite{BHM}. Thus, upon realization, it has an $S^1$-action. 
\end{defn}

\begin{defn}For $U$ a complete $S^1$-universe, and $R$ a $C_n$-ring spectrum indexed on $\widetilde{U}= i^*_{C_n}U$, 
\[
\THH_{C_n}(R):= N_{C_n}^{S^1}(R):= \mathcal{I}_{\mathbb{R}^{\infty}}^U|\mathrm{N}^{\cyc, C_n}_{\bullet}(\mathcal{I}_{\widetilde{U}}^{\mathbb{R}^{\infty}}R)|.
\]
\end{defn}
For more on twisted THH, see, for example, \cite{AnBlGeHiLaMa}, \cite{aghkk}, or \cite{BlGeHiLa}. 

\subsection{Spectral categories}\label{sec:SpecCat}

In this section we briefly recall the definition of spectral categories, as well as the cyclic bar construction in this setting. See \cite{BlumbergMandell:localizationLongOne}, \cite{CLMPZ}, or \cite{CLMPZ2} for more details. 

\begin{defn}\label{def-speccat}
A spectral category $\mathcal{C}$ is a category enriched in orthogonal spectra. In more detail, this consists of the data of 
\begin{itemize}
    \item for any two objects $a,b \in \mathrm{ob}(\mathcal{C})$, a spectrum $\mathcal{C}(a,b)$,
    \item for any object $a \in \mathrm{ob}(\mathcal{C})$, a unit map from the sphere spectrum to $\mathcal{C}(a,a)$, and
    \item for any three objects $a,b,c \in \mathrm{ob}(\mathcal{C})$, a strictly associative and unital composition map $\mathcal{C}(b,c) \sm \mathcal{C}(a,b) \to \mathcal{C}(a,c)$.
\end{itemize}
\end{defn}

\begin{exmp}
    Let $R$ be an orthogonal ring spectrum. Then there is a spectral category $\mathcal{C}_R$ with one object $\bullet$ such that $\mathcal{C}_R(\bullet, \bullet) = R$. The unit and composition maps are given by the unit and multiplication maps of $R$, respectively.
\end{exmp}
\begin{exmp}\label{ex:SpecCatMods}
    Let $R$ be an orthogonal ring spectrum.  There is a spectral category $\Mod_R$ whose objects are right $R$-modules $M$ and morphism spectra given by the internal hom $F_R(M,M')$.  The unit and composition are given by usual spectral enrichment of identity and composition.
\end{exmp}

\begin{defn}\label{def-ptwise-cof}
    We say that a spectral category $\mathcal{C}$ is pointwise cofibrant if all its mapping spectra are cofibrant in the stable model structure on orthogonal spectra.
\end{defn}

For example, if $R$ is cofibrant as a spectrum, then $\mathcal{C}_R$ is pointwise cofibrant.

\begin{defn}\label{def-spectral-functor}
Let $\mathcal{C}, \mathcal{D}$ be spectral categories. A functor of spectral categories associates to each object $a \in \mathrm{ob}(\mathcal{C})$ an object $F(a) \in \mathrm{ob}(\mathcal{D})$, and to every $a,b \in \mathrm{ob}(\mathcal{C})$ a map of spectra $F_{a,b}: \mathcal{C}(a,b) \to \mathcal{D}(F(a), F(b))$ which respects composition and units.
\end{defn}

We denote by $\SpecCat$ the category whose objects are small spectral categories and whose morphisms are spectral functors.

\begin{defn}\label{def-op-sp}
For a spectral category $\mathcal{C}$, let $\mathcal{C}^{\mathrm{op}}$ denote the spectral category with the same objects as $\mathcal{C}$, mapping spectra given by $\mathcal{C}^{\mathrm{op}}(x,y) = \mathcal{C}(y,x)$, and composition given by: 
\[
    \mathcal{C}^{\op}(y,z)\wedge \mathcal{C}^{\op}(x,y) = \mathcal{C}(z,y)\wedge \mathcal{C}(y,x) \xrightarrow{\sigma} \mathcal{C}(y,x)\wedge\mathcal{C}(z,y) \xrightarrow{\circ} \mathcal{C}(z,x) = \mathcal{C}^{\op}(x,z)  
\]
where $\sigma$ is the swap map and $\circ$ is the composition map of $\mathcal{C}$.
\end{defn}

If $\mathcal{C}$ and $\mathcal{D}$ are two spectral categories, we can define a new spectral category $\mathcal{D}^{\op}\wedge \mathcal{C}$ which has objects $\mathrm{ob}(\mathcal{D})\times \mathrm{ob}(\mathcal{C})$.  The morphisms are given by
\[
    (\mathcal{D}^{\op}\wedge \mathcal{C})((d_1,c_1),(d_2,c_2)) =\mathcal{D}^{\op}(d_1,d_2)\wedge \mathcal{C}(c_1,c_2) =\mathcal{D}(d_2,d_1)\wedge \mathcal{C}(c_1,c_2),
\]

\begin{defn}
    For $\mathcal{C}$ and $\mathcal{D}$ two spectral categories, a $(\mathcal{C},\mathcal{D})$-bimodule $M$ is a spectral functor $M\colon \mathcal{D}^{\mathrm{op}} \wedge \mathcal{C} \to \Spec$.
\end{defn}

Explicitly, a $(\cC,\cD)$-bimodule consists of a collection of spectra $\cM(d,c)$ for all objects $(d,c)\in \mathrm{ob}(\cD)\times \mathrm{ob}(\cC)$, together with structure maps
\[
    \alpha\colon \mathcal{C}(c,e)\wedge M(d,c)\wedge \mathcal{D}(f,d)\to M(f,e)
\]
which are associative and unital with respect to the composition and units in $\cC$ and $\cD$.

\begin{defn}
For a small spectral category $\mathcal{C}$ and a $(\mathcal{C}, \mathcal{C})$-bimodule $\mathcal{M}$, the cyclic bar construction $\mathrm{N}^{cyc}_{\bullet}(\mathcal{C};\mathcal{M})$ is the simplicial spectrum with $q$-simplices
\[
    \Ncyc_{q}(\mathcal{C};\mathcal{M}) = \bigvee\limits_{(c_0,c_1,\dots, c_q)} \mathcal{C}(c_{q-1},c_q)\wedge \mathcal{C} (c_{q-2},c_{q-1})\wedge\dots\wedge \mathcal{C}(c_0,c_1)\wedge \mathcal{M}(c_q,c_0)
\]
where the wedge is over all $(q+1)$-tuples of objects $c_i\in \mathcal{C}$. The  zeroth face map cycles the leftmost smash summand to the far right and then uses the right-module structure of $\mathcal{M}$. The remaining face maps are given by the composition in $\mathcal{C}$ and the left module structure of $\mathcal{M}$.  
\end{defn}

When the bimodule $\mathcal{M}$ is $\mathcal{C}$ we will write $\Ncyc_{\bullet}(\mathcal{C})$.  Note that in this case the cyclic bar construction is a cyclic object in the sense of Connes. When $\mathcal{M} = \mathcal{C}$, the geometric realization of the cyclic bar construction has an $S^1$-action. Let $U$ be a complete $S^1$-universe.

\begin{defn}
For a small spectral category $\mathcal{C}$,
    $$\THH(\mathcal{C}) = \mathcal{I}_{\mathbb{R}^{\infty}}^{U}|\Ncyc_{\bullet}(\mathcal{C};\mathcal{C})|.$$
\end{defn}

\section{Equivariant topological Hochschild homology}\label{sec:ETHH}

In this section introduce a version of topological Hochschild homology for genuine $G$-ring spectra, for a finite group $G$.  We call this theory equivariant THH, denoted ETHH, and it will play a central role in our trace methods approach to equivariant algebraic $K$-theory constructed later in the paper.  

Aside from its place in the theory of trace maps, ETHH is also notable because it is a model for a norm from $G$-spectra to $G\times S^1$-spectra.  Since the introduction of norms for finite groups in \cite{HHR}, norms have played a central role in equivariant homotopy theory.  Despite this, the construction of norms for compact Lie groups is in active development; see, for instance,  \cite{BlumbergHillMandellNorms}.  

Let $\mathcal{U}$ be a complete $(G\times S^1)$-universe, and let $\mathcal{V} = \iota_G^*\mathcal{U}$.  

\begin{defn}
Let $G$ be a finite group. The equivariant THH of a genuine orthogonal ring $G$-spectrum $R$ is the $G\times S^1$-spectrum $\ETHH(R) :=N^{G\times S^1}_G(R):=\mathcal{I}^{\mathcal{U}}_{\mathcal{V}}|\mathrm{N}^{\cyc}_{\bullet}(R)|$.
\end{defn}

\begin{exmp}\label{example: free loop space}
    Let $X$ be a based $G$-space such that $X^H$ is connected for all $H\leq G$.  Then the genuine equivariant suspension spectrum $\Sigma^{\infty}_G (\Omega X_+)$ is an associative $G$-ring spectrum and there is a weak equivalence \[\ETHH(\Sigma^{\infty}_G \Omega X_+)\simeq \Sigma^{\infty}_{G\times S^1} \mathcal{L}(X)_+,\] where $\mathcal{L}$ denotes the free loop space. This equivalence follows from the fact that suspension spectrum is strong monoidal, and therefore commutes with the cyclic bar construction. The space $\mathcal{L}(X)$ is considered as a $G\times S^1$-space via the action of $S^1$ on itself and the action of $G$ on $X$.  Explicitly, if $f\colon S^1\to X$ is a free loop then $((g,t)\cdot f)(s) = g\cdot f(t^{-1}s).$
\end{exmp}

Let $R\in \Spec^G_{\mathcal{V}}$ be a $G$-ring spectrum so that $\mathcal{I}_\mathcal{V}^{\mathbb{R}^{\infty}}(R)$ is a ring object in $\Spec^G_{\mathbb{R}^{\infty}}$.  The cyclic bar construction of $\mathcal{I}_\mathcal{V}^{\mathbb{R}^{\infty}}(R)$ is then an $S^1$-object in $\Spec^G_{\mathbb{R}^{\infty}}\cong \Sp^{BG}:= \Fun(BG,\Sp)$ where $BG$ is the one object groupoid of $G$.  In particular, this yields an object in $(\Spec^{BG})^{BS^1}\cong \Spec^{B(G\times S^1)}\cong \Spec^{G\times S^1}_{\mathbb{R}^{\infty}}$.

\begin{lem}
    If $R$ is a genuine orthogonal $G$-ring spectrum then there is an isomorphism of genuine $G\times S^1$-spectra $\ETHH(R)\cong \mathcal{I}_{\mathbb{R}^{\infty}}^{\mathcal{U}}|\mathrm{N}^{\cyc}_{\bullet}(\mathcal{I}^{\mathbb{R}^{\infty}}_{\mathcal{V}}(R))|.$ 
\end{lem}
\begin{proof}
    Since change of universe is a strong monoidal equivalence of categories we have 
    \[  \mathcal{I}_{\mathbb{R}^{\infty}}^{\mathcal{U}}|\mathrm{N}^{\cyc}_{\bullet}(\mathcal{I}^{\mathbb{R}^{\infty}}_{\mathcal{V}}(R))|  \cong \mathcal{I}_{\mathbb{R}^{\infty}}^{\mathcal{U}}\mathcal{I}^{\mathbb{R}^{\infty}}_{\mathcal{V}}|\mathrm{N}^{\cyc}_{\bullet}(R)| \cong \mathcal{I}^{\mathcal{U}}_{\mathcal{V}}|\mathrm{N}^{\cyc}_{\bullet}(R)|
    \]
    where the second isomorphism uses \autoref{lemma: change of universe facts}.
\end{proof}

\begin{rem}\label{rem-eTHH-rings}
    We will sometimes find it convenient to refer to the cyclic bar construction as a spectrum with $G \times S^1$-action. We will denote this by $\eTHH$. That is,
    $$\eTHH(R) = |\mathrm{N}^{\cyc}_{\bullet}(\mathcal{I}^{\mathbb{R}^{\infty}}_{\mathcal{V}}(R))|.$$
\end{rem}

We next show that $N_G^{G\times S^1}$ has the expected adjoint property of a norm, as in \autoref{prop: norm is left quillen adjoint finite}.
\begin{prop}
    The functor $N_G^{G\times S^1}$ restricts to a functor
    \[
        N^{G\times S^1}_G\colon \mathrm{Comm}_\mathcal{V}^{G}\to\mathrm{Comm}^{G\times S^1}_\mathcal{U}.
    \]
    Moreover, this functor is left adjoint to the forgetful functor $i_G^*\colon \mathrm{Comm}_\mathcal{U}^{G\times S^1}\to\mathrm{Comm}_\mathcal{V}^G$, and this is a Quillen adjunction, where the categories of commutative algebras are given the positive complete stable model structures.
\end{prop}
\begin{proof}
    Since the change of universe functors are equivalences of categories it suffices to prove that the functor which sends $A\in \mathrm{Comm}^G_{\mathbb{R}^{\infty}}$ to $|\Ncyc_{\bullet}(A)|$ in $\mathrm{Comm}^{G\times S^1}_{\mathbb{R}^{\infty}}$ is left adjoint to the forgetful functor. As in \cite[Proposition 4.3]{AnBlGeHiLaMa}, this follows from the observation that, as $S^1$-objects in $\mathrm{Comm}^{G}_{\mathbb{R}^{\infty}}$, we can identify $|\Ncyc_{\bullet}(A)|$  with  $S^1\otimes A$, the tensoring of $A$ with the simplicial circle $S^1$ considered as a $G$-simplicial set with trivial action. The fact that this is a Quillen adjunction is immediate from the fact that the forgetful functor preserves positive fibrations and positive acyclic fibrations.
\end{proof}

While $\ETHH$ is not a left adjoint on associative ring spectra, we note that it still preserve weak equivalences between cofibrant objects.

\begin{prop}
    If $f\colon R\xrightarrow{\sim}S$ is a weak equivalence of ring orthogonal $G$-spectra which are positive cofibrant then the induced map $\ETHH(f)\colon \ETHH(R)\to \ETHH(S)$ is a weak equivalence of $G\times C_n$-spectra for all $n$.
\end{prop}
\begin{proof}
        We will show that for all $K\leq G\times C_n$ we have a weak equivalence
        \[
            \Phi^K(\ETHH(R))\to \Phi^K(\ETHH(S))
        \]
        on all geometric fixed points, hence $f$ is a weak equivalence by \autoref{lem: norm and fixed point lemmas}(4).

        Since we are only interested in the $G\times C_n$-structure we may replace $\ETHH(R)$ with the spectrum obtained by geometric realization of the $n$-th cyclic subdivision of $N^{\mathrm{cyc}}(R)$.  After passing the change of universe inside the geometric realization the $q$-simplices of this simplicial $G$-spectrum are
        \[
            [q]\mapsto (N^{G\times C_n}_{G\times 1}(R))^{\wedge(q+1)}
        \]
        and the map $\ETHH(R)\to \ETHH(S)$ is induced by the simplicial map given on $q$-simplices by the smash product of the maps
        \[
            N^{G\times C_n}_{G\times 1}(f)\colon N^{G\times C_n}_{G\times 1}(R)\to N^{G\times C_n}_{G\times 1}(S).
        \]
        These maps are weak equivalences between cofibrant objects, by \autoref{prop: norm is left quillen adjoint finite}, and thus the smash product of these maps is also a weak equivalence between cofibrant objects.  Applying any geometric fixed points, and using \autoref{lem: norm and fixed point lemmas}(2), we see that the map $\ETHH(R)\to \ETHH(S)$ is the realization of a levelwise equivalence between levelwise cofibrant simplicial $G\times C_n$-spectra which proves the claim.
\end{proof}

If $G = C_n = \langle \sigma| \sigma^n\rangle$, we fix the inclusion of $C_n\hookrightarrow S^1$, which sends $\sigma$ to $e^{2\pi i/n}$.  Let $\Delta_{C_n}$ be the diagonal subgroup of $C_n\times S^1$ isomorphic to $C_n$. The goal for the remainder of this section is to prove there is an equivalence of $C_n$-spectra:
\[\THH_{C_n}(R)\simeq \Phi^{\Delta_{C_n}} \ETHH(R).\]
The spectrum on the right is a $(C_n\times S^1)/(\Delta_{C_n})$-spectrum.  To make sense of this equivalence we need to pick an isomorphism $(C_n\times S^1)/\Delta_{C_n}\xrightarrow{\cong}S^1$; we pick the map $(e^{2k\pi i/n},e^{\theta i})\mapsto e^{(2k\pi/n-\theta)i}$. This choice makes the inclusion 
\[C_n\cong C_n\times 1 \cong ((C_n\times 1)\Delta_{C_n})/\Delta_{C_n}\to (C_n\times S^1)/\Delta_{C_n}\cong S^1\]
the usual embedding $C_n\to S^1$ and sends $(1,\sigma)$ to $\sigma^{-1}$.  

Given a $G$-ring spectrum $R$, an $(R,R)$-bimodule $M$, and an element $g\in G$ we define the right-twisted bimodule $M^{g}$ to be the $(R,R)$-bimodule with the same left $R$-action as $M$ but with right $R$-action given by
\[  
    M\wedge R\xrightarrow{1\wedge g} M\wedge R\to M
\]
where the unlabeled map is the original right $R$-action.  The following lemma is helpful for computing the geometric fixed points of $\ETHH(R)$. One can define the left-twisted bimodule ${^{g}M}$ analogously.

\begin{lem}[{c.f.\ \cite[Theorem 4.4]{AnBlGeHiLaMa}}]\label{lemma:TCviaNorm4.4}
    Let $R$ be an orthogonal $C_n$-ring spectrum.  There is an isomorphism of genuine $(C_n\times C_n)$-spectra
    \[
        i_{C_n\times C_n}^* N^{C_n\times S^1}_{C_n}(R)\cong |\Ncyc_{\bullet}(N_{C_n\times 1}^{C_n\times C_n}(R),N_{C_n\times 1}^{C_n\times C_n}(R)^{(1,\sigma)})|
    \]
    where the right term is the cyclic bar construction with coefficients in a twisted bimodule. 
\end{lem}
\begin{proof}
    The proof is almost identical that of \cite[Theorem 4.4]{AnBlGeHiLaMa}.  The result also follows from a more general result, \autoref{lemma:TCviaNorm4.4spectral}, proven below.
\end{proof}

\begin{thm}\label{thm: geometric fixed points of ETHH is twisted THH for rings}
    Let $R$ be a positive cofibrant $C_{n}$-ring spectrum. 
 There is an isomorphism of orthogonal $C_n$-spectra 
    \[
        \Phi^{\Delta_{C_n}}\ETHH(R) \cong \THH_{C_n}(R).
    \]
\end{thm}
\begin{proof}

Using \autoref{lemma:TCviaNorm4.4}, we have
\begin{align*}
    i_{C_n}^*\Phi^{\Delta_{C_n}}\ETHH(R)& \cong \Phi^{\Delta_{C_n}}i^{*}_{C_n\times C_n}\ETHH(R) \\
    & \cong \Phi^{\Delta_{C_n}}|\mathrm{N}^{\cyc}(N_{C_n\times 1}^{C_n\times C_n}(R), (N_{C_n\times 1}^{C_n\times C_n}(R))^{(1,\sigma)})| \\
    &\cong  |\Phi^{\Delta_{C_n}}\mathrm{N}^{\cyc}(N_{C_n\times 1}^{C_n\times C_n}(R), (N_{C_n\times 1}^{C_n\times C_n}(R))^{(1,\sigma)})|
\end{align*}
where the last spectrum is the realization of the simplicial spectrum obtained from the twisted bar construction by applying geometric fixed points levelwise.  The last isomorphism follows from \cite[Lemma 4.7]{DottoMalkPatchSagaveWoo}.  The $k$-simplices of the resulting simplicial $C_n$-spectrum are
\begin{equation}\label{equation:norm diagonal}
    \Phi^{\Delta_{C_n}}\left(N^{C_n\times C_n}_{C_n\times 1}(R) \right)^{\wedge(k+1)}  \cong   \left(\Phi^{\Delta_{C_n}}N^{C_n\times C_n}_{C_n\times 1}(R) \right)^{\wedge(k+1)} \cong R^{\wedge(k+1)}
\end{equation}
where the first isomorphism uses that geometric fixed points are strong monoidal on cofibrant spectra and the second uses the norm diagonal, \cite[Theorem 2.35]{AnBlGeHiLaMa}. 

Thus we have, up to isomorphism, a simplicial $C_n$-spectrum with the same $k$-simplices as $\mathrm{N}^{\cyc,C_n}(R)$ which defines $\THH_{C_n}(R)$. It remains to check that the face and degeneracy maps are correct.  Lax monoidality of geometric fixed points implies that the last $k$ face maps are given by the multiplication of $R$, as expected.  For the face map $d_0$, it suffices to identify 
\[
    \Phi^{\Delta_{C_n}}(1,\sigma)\colon \Phi^{\Delta_{C_n}}(N^{C_n\times C_n}_{C_n\times 1}(R))\to \Phi^{\Delta_{C_n}}(N^{C_n\times C_n}_{C_n\times 1}(R))
\]
with $\sigma^{n-1}\colon R\to R$.  This follows from the observations that the isomorphism \eqref{equation:norm diagonal} is an isomorphism of $(C_n\times C_n)/(\Delta_{C_n})$-spectra, and that the isomorphism $(C_n\times C_n)/(\Delta_{C_n})\cong C_n\times 1$ identifies $(1,\sigma)$ with $\sigma^{n-1}$.
\end{proof}

\begin{exmp}
    Consider the case of $R = \Sigma^{\infty}_G \Omega X_+$ from \autoref{example: free loop space}.  There is an equivalence 
\[
    \ETHH(\Sigma^{\infty}_G \Omega X_+)\simeq \Sigma^{\infty}_{G\times S^1} \mathcal{L}(X)_+
\]
and so when $G=C_n$ we can compute $\THH_{C_n}$ by taking geometric fixed points of the suspension of the free loop space.  Since geometric fixed points commute with equivariant suspension spectra we have
\[
    \THH_{C_n}(\Sigma^{\infty}_{C_n} (\Omega X)_+)\simeq \Sigma^{\infty} (\mathcal{L}(X)^{\Delta_{C_n}}_+).
\]

The $\Delta_{C_n}$-fixed points of $\mathcal{L}(X)$ can be computed as the twisted free loop space $\Lambda^{\sigma}(X)$, the space of paths $\varphi\colon I\to X$ such that $\sigma\varphi(0)  = \varphi(1)$.  This space is of interest because its homology is the home for fixed point invariants such as the Reidemeister trace \cite{MalkiewichPonto}.
\end{exmp}

\section{Hochschild homology theories for $G$-spectral categories}\label{sec-GSp}

Several modern constructions of the Dennis trace require extending the cyclic bar construction from ring spectra to spectral categories, which can be thought of as ring spectra ``with many objects.''  In this section, in order to facilitate the construction of the equivariant Dennis trace,  we introduce the notion of $G$-spectral categories. The notion of equivariant bimodules and ETHH of $G$-spectral categories is studied.  We also study the homotopy theory of $G$-spectral categories by constructing a model structure on the category of $G$-spectral categories.  This model structure allows us to construct cofibrant replacements of $G$-spectral categories, a necessary technical step in constructing a well defined Dennis trace.

Let $\SpecCat$ denote the category of spectral categories, as in \autoref{sec:SpecCat}.
 
\begin{defn}
A $G$-spectral category is a functor
\[
BG \to \mathrm{SpecCat}.
\]
A morphism of $G$-spectral categories is a natural transformation of functors. 
\end{defn}

Explicitly, a $G$-spectral category $\mathcal{C}$ consists of
\begin{enumerate}
\item A collection of objects ob$\mathcal{C}$ together with an assignment
\[
g\colon \textup{ob} \mathcal{C} \to \textup{ob} \mathcal{C}
\]
for each $g\in G$, such that $h \circ g  = hg.$
\item For each pair of objects $x,y \in \textup{ob}\mathcal{C}$, an orthogonal spectrum $\mathcal{C}(x,y)$ and a map of spectra
\[
g\colon \mathcal{C}(x,y)\to \mathcal{C}(gx,gy).
\]
for each $g\in G$, such that $h \circ g  = hg.$
\item A unit map $\eta_x\colon \mathbb{S} \to \mathcal{C}(x,x)$ for each $x \in \textup{ob}\mathcal{C}$, such that $g \circ \eta_x = \eta_{gx}$.
\item For each triple of objects $x,y,z$ a composition map
\[
\circ\colon \mathcal{C}(y,z) \wedge \mathcal{C}(x,y) \to \mathcal{C}(x,z),
\]
which is unital and associative, and such that the diagram
\[
    \begin{tikzcd}
        \mathcal{C}(y,z)\wedge \mathcal{C}(x,y) \ar[r,"g\wedge g"] \ar[d,"\circ"] & \mathcal{C}(gy,gz)\wedge \mathcal{C}(gx,gy) \ar[d,"\circ"] \\
        \mathcal{C}(x,z) \ar[r,"g"]& \mathcal{C}(gx,gz)
    \end{tikzcd}
\]
commutes for all choices of $g,x,y,z$.
\end{enumerate}

\begin{exmp}\label{example: G-ring as G-spec cat}
    Let $R$ be a ring spectrum with $G$-action.  There is a $G$-spectral category $\mathcal{C}_R$ with a single object $\bullet$ and morphism spectrum $\mathcal{C}_R(\bullet, \bullet)=R$.  The composition is given by the multiplication $R\wedge R\to R$ and the maps $g\colon R\to R$ are given by the $G$-action on $R$.
\end{exmp}

\begin{exmp}\label{example: modules as G-spec cat}
    Recall from \autoref{ex:SpecCatMods} that for an orthogonal ring spectrum $R$ the category $\Mod_R$ of right $R$-modules is a spectral category. If $R$ has a $G$-action then $\Mod_R$ is a $G$-spectral category with $G$-action defined as follows.  For a right $R$-module $M$ the $R$-module $gM = M^{g^{-1}}$ is the spectrum $M$ with $R$-action map
    \[
        M\wedge R\xrightarrow{1\wedge g^{-1}}M\wedge R\to M
    \]
    where the second map is the action of $R$ on $M$.  The morphism spectra are given by 
    \[
        \Mod_R(M,N) = F_R(M,N)
    \]
    where the right hand side is the internal mapping spectrum of right $R$-modules. 
    
    For $g\in G$ the maps on morphism spectra
    \[
        g\colon F_R(M,N)\to F_R(M^{g^{-1}},N^{g^{-1}})
    \]
    arise as follows: if $f\colon M\to N$ is a map of spectra which is also a map of right $R$-modules, then a straightforward diagram chase shows that the same map of spectra $f\colon M^{g^{-1}}\to N^{g^{-1}}$ is also a map of right $R$-modules (recall that $M$ and $M^{g^{-1}}$ have the same underlying spectrum). Thus the underlying set of right $R$-modules maps $M\to N$ and $R$-module maps $M^{g^{-1}}\to N^{g^{-1}}$ are the same.  To make this precise, on morphism spectra one translates the above argument into the argument that the canonical map $\phi\colon F_{R}(M,N)\to F(M,N) = F(M^{g^{-1}},N^{g^{-1}})$ equalizes the maps which define $F_{R}(M^{g^{-1}},N^{g^{-1}})$, hence $\phi$ factors as $F_{R}(M,N)\xrightarrow{g} F_{R}(M^{g^{-1}},N^{g^{-1}}) \to F(M,N)$.  We note that if $M=M^{g^{-1}}$ and $N = N^{g^{-1}}$, the map $g: F_{R}(M,N) \to F_{R}(M^{g^{-1}},N^{g^{-1}})$ must be the identity, by the universal property of the equalizer. 
\end{exmp}

\begin{exmp}\label{example: perf as G-spec Cat}
    A right module $M$ over a ring spectrum $R$ is called \emph{perfect} if it is contained in the smallest thick subcategory of $\Mod_R$ containing $R$.  We will write $\perf_R$ for the spectral category of perfect right $R$-modules.  To see that $\perf_R$ is a $G$-spectral category it suffices to observe that if $P$ is a perfect right $R$-module then $gP = P^{g^{-1}}$ is also perfect.  

    To see this, note that if $P$ is perfect, then $gP$ is in the smallest thick subcategory of $\Mod_R$ containing $gR = R^{g^{-1}}$.  But there is an isomorphism of right $R$-modules $gR\cong R$ given by the mutually inverse maps $g^{-1}\colon R\to (gR)$ and $g\colon (gR)\to R$ and so $gP$ must also be perfect.
\end{exmp}

\begin{rem}
    The examples above can, of course, be carried out for left modules and left perfect modules, respectively.  Indeed the $G$-action in this case is actually a bit more natural.  The reason for spelling out the case of right modules instead of left modules is that right modules play a more central role in our study of Morita equivalences in \autoref{sec: Morita}.
\end{rem}

Later we will need to understand the interaction between mapping objects and the action of $G$ on $\Mod_R$.  We record these interactions in a few lemmas.
\begin{lem}\label{lemma: right twisted mapping set}
    Let $R$ be a ring orthogonal spectrum with $G$-action. For any right $R$-modules $M$ and $N$ we have a natural bijection
    \[
        \Mod_R(M^g,N)\leftrightarrow \Mod_R(M,N^{g^{-1}}).
    \]
    Moreover, there is a canonical isomorphism of mapping spectra $F_{R}(M^g,N)\cong F_R(M,N^{g^{-1}})$
\end{lem}
\begin{proof}
    The elements in the first mapping set consist of maps of spectra $f\colon M\to N$ such that the diagram
    \[
        \begin{tikzcd}
            M\wedge R \ar[r,"1\wedge g"] \ar[d,"f\wedge 1"'] & M\wedge R\ar[r,"\alpha_M"] & M\ar[d,"f"]\\
            N\wedge R \ar[rr,"\alpha_N"'] & & N
        \end{tikzcd}
    \]
    commutes, where $\alpha$ denotes a module action map.  Given such a map $f$, the diagram
    \[
        \begin{tikzcd}[column sep = large]
            M\wedge R \ar[dd,"f\wedge 1"'] \ar[rr,"\alpha_M"] & & M \ar[dd,"f"] \\
            & M\wedge R \ar[ul,"1\wedge g"] \ar[d,"f\wedge 1"]&\\
            N\wedge R \ar[r,"1\wedge g^{-1}"'] & N\wedge R \ar[r,"\alpha_N"'] & N
        \end{tikzcd}
    \]
    must also commute; the right side is the first diagram and the left side commutes because the smash product is monoidal.  But the outside of the second diagram is exactly the data of a morphism in $\Mod_{R}(M,N^{g^{-1}})$ so these morphism sets consists of the same maps of spectra.  The result about mapping spectra is essentially the same with the added difficulty of needing to work with the mapping adjunctions.  We leave the details to the reader.
\end{proof}
\begin{lem}\label{lemma: left twisted mapping set}
    Let $R$ be an orthogonal ring spectrum with $G$-action. For any right $R$-module $M$ we consider $F_R(M,R)$ as a left $R$-module via the left $R$-module structure of $R$.  There is an isomorphism of $R$-module spectra $^{g}F_R(M,R)\cong F_R(M,{^{g}R)}$.
\end{lem}
\begin{proof}
    Explicitly, the left $R$-module structure on $F_R(M,R)$ is given by the adjunct of the composite
    \[
        R\wedge F_R(M,R)\wedge M\xrightarrow{1\wedge \mathrm{ev}} R\wedge R\xrightarrow{\mu} R
    \]
    where $\mathrm{ev}$ is the evaluation map and $\mu$ is the multiplication on $R$.  Note that the diagram
    \[
        \begin{tikzcd}
            R\wedge F_R(M,R)\wedge M \ar[r,"1\wedge \mathrm{ev}"] \ar[d,"g\wedge 1\wedge 1"]& R\wedge R \ar[d,"g\wedge 1"] & \\
            R\wedge F_R(M,R)\wedge M \ar[r,"1\wedge \mathrm{ev}"] & R\wedge R \ar[r,"\mu"] & R
        \end{tikzcd}
    \]
    commutes.  Since the two ways of going around the diagram represent the adjuncts of the left $R$-module structure maps on $F_{R}(M,^{g}R)$ and $^{g}F_R(M,R)$ respectively these module structure maps are the same.
\end{proof}

To study the category of $G$-spectral categories, we introduce the notion of a $G$-spectral functor.

\begin{defn}
    If $\cC$ and $\cD$ are $G$-spectral categories then a \emph{$G$-spectral functor} $F\colon \cC\to \cD$ is a natural transformation of functors $BG\to \SpecCat$.
\end{defn}

Explicitly, a $G$-spectral functor $F\colon \cC\to \cD$ consists of the following data:
\begin{enumerate}
    \item for every $x\in \mathrm{ob}\cC$, an object $F(x)\in \mathrm{ob}\cD$,
    \item for every pair of objects $x,y\in \mathrm{ob}\cC$, a map of spectra
    \[
        F\colon \cC(x,y)\to \cD(F(x),F(y))
    \]
    which respects composition on the nose,
    \item for every $g\in G$, $F(gc) = gF(c)$ and the diagram
    \[
        \begin{tikzcd}
            \cC(x,y) \ar[r,"g"] \ar[d,"F"] & \cC(gx,gy) \ar[d,"F"]\\
            \cD(F(x),F(y)) \ar[r,"g"] & \cD(F(gx),F(gy))
        \end{tikzcd}
    \]
    commutes for all $x,y\in \mathrm{ob(}C)$.
\end{enumerate}

\begin{defn}
For a $G$-spectral category $\mathcal{C}$, let $\mathcal{C}^{\mathrm{op}}$ denote the $G$-spectral category with the same objects as $\mathcal{C}$, mapping spectra given by $\mathcal{C}^{\mathrm{op}}(x,y) = \mathcal{C}(y,x)$, and composition given by: 
\[
    \mathcal{C}^{\op}(y,z)\wedge \mathcal{C}^{\op}(x,y) = \mathcal{C}(z,y)\wedge \mathcal{C}(y,x) \xrightarrow{\sigma} \mathcal{C}(y,x)\wedge\mathcal{C}(z,y) \xrightarrow{\circ} \mathcal{C}(z,x) = \mathcal{C}^{\op}(x,z)  
\]
where $\sigma$ is the swap map and $\circ$ is the composition map of $\mathcal{C}$.
\end{defn}

If $\mathcal{C}$ and $\mathcal{D}$ are two $G$-spectral categories we can define a new $G$-spectral category $\mathcal{D}^{\op}\wedge \mathcal{C}$ which has objects $\mathrm{ob}(\mathcal{D})\times \mathrm{ob}(\mathcal{C})$, and $g$ acts diagonally.  The morphisms are given by
\[
    (\mathcal{D}^{\op}\wedge \mathcal{C})((d_1,c_1),(d_2,c_2)) =\mathcal{D}^{\op}(d_1,d_2)\wedge \mathcal{C}(c_1,c_2) =\mathcal{D}(d_2,d_1)\wedge \mathcal{C}(c_1,c_2),
\]
and the map 
\[
g:(\mathcal{D}^{\op}\wedge \mathcal{C})((d_1,c_1),(d_2,c_2))  \to (\mathcal{D}^{\op}\wedge \mathcal{C})(g(d_1,c_1),g(d_2,c_2)) 
\]
is the map
\[
g \wedge g: \mathcal{D}(d_2,d_1)\wedge \mathcal{C}(c_1,c_2) \to \mathcal{D}(gd_2,gd_1)\wedge \mathcal{C}(gc_1,gc_2).
\]

We now define equivariant bimodules over $G$-spectral categories $(\cC,\cD)$.  Recall that a non-equivariant $(\cC,\cD)$-bimodule is spectral functor $\cD^{\op}\wedge \cC\to \Sp$.  This entails the existence of maps of spectra
    \[
    \alpha\colon \mathcal{C}(c,e)\wedge M(d,c)\wedge \mathcal{D}(f,d)\to M(f,e)
\]
which are compatible with the units and strictly associative with the compositions of $\mathcal{C}$ and $\mathcal{D}$.

\begin{defn}
    For $\mathcal{C}$ and $\mathcal{D}$ two $G$-spectral categories, an equivariant $(\mathcal{C},\mathcal{D})$-bimodule $M$ is a spectral functor $M: \mathcal{D}^{\mathrm{op}} \wedge \mathcal{C} \to \Spec$, with designated maps
    \[
g: M(x,y) \to M(gx, gy)
    \]
    such that that $h \circ g = hg$ for all $h, g \in G$, and the map $e$ is the identity.  We further require that the diagram
\[
    \begin{tikzcd}
        \mathcal{C}(c,e)\wedge M(d,c)\wedge \mathcal{D}(f,d) \ar[r,"\alpha"] \ar[d,"g\wedge g\wedge g"] & M(f,e) \ar[d,"g"] \\
        \mathcal{C}(gc,ge)\wedge M(gd,gc)\wedge \mathcal{D}(gf,gd)\ar[r,"\alpha"]  & M(gf,ge)
    \end{tikzcd}
\]
commutes for all choices of $a,b,c,d,e,f$ and $g\in G$. A map of equivariant bimodules is a natural transformation of spectral functors.  We denote the category of equivariant $(\mathcal{C},\mathcal{D})$-bimodules by $\mathrm{Bimod}_G(\mathcal{C},\mathcal{D})$.
\end{defn}
\begin{exmp}
    Any $G$-spectral category $\cC$ is an equivariant bimodule over itself via the canonical functor $\cC^{\op}\wedge\cC\to \Sp$ which sends a pair $(x,y)$ to $\cC(x,y)$.
\end{exmp}

\begin{exmp}\label{example: bimodules over one object spec cats}
    Let $R$ be a ring spectrum with $G$-action, let $\mathcal{C}_R$ be the single object $G$-spectral category associated to $R$, and let $\mathcal{D}$ be any $G$-spectral category.  A $(\mathcal{D},\mathcal{C}_R)$-bimodule consists of a spectrum $\mathcal{M}(d) = \mathcal{M}(\bullet,d)$ for all $d\in \mathcal{D}$, together with maps
    \begin{align*}
        & \mathcal{M}(d)  \xrightarrow{g} \mathcal{M}(gd)\\
       &  \mathcal{M}(d)\wedge R  \xrightarrow{\alpha} \mathcal{M}(d)\\
        & \mathcal{D}(d,d')\wedge \mathcal{M}(d)  \xrightarrow{\beta} \mathcal{M}(d')
    \end{align*}
    for all $g\in G$ and $d'\in \mathcal{D}$ such that the diagrams
    \[
        \begin{tikzcd}
            \mathcal{M}(d)\wedge R\ar[r,"g\wedge g"] \ar[d,"\alpha"]& \mathcal{M}(gd)\wedge R  \ar[d,"\alpha"]& \mathcal{D}(d,d')\wedge \mathcal{M}(d)  \ar[r,"g\wedge g"] \ar[d,"\beta"] & \mathcal{D}(gd,gd')\wedge \mathcal{M}(gd)\ \ar[d,"\beta"] \\
            \mathcal{M}(d) \ar[r,"g"] & \mathcal{M}(gd) & \mathcal{M}(d')\ar[r,"g"] & \mathcal{M}(gd')
        \end{tikzcd}
    \]
    commute.  The map $\alpha$ turns $\mathcal{M}(d)$ into a right $R$-module for all $d\in \mathcal{D}$. A quick diagram chase shows that the commutativity of the left diagram is equivalent to the claim that the map $g\colon \mathcal{M}(d)\to \mathcal{M}(gd)^{g}$ is a map of right $R$-modules for all $g\in G$.  Thus we see that a $(\mathcal{D},\mathcal{C}_R)$-bimodule consists of a spectral functor $\mathcal{D}\to \Mod_R$ together which choices of isomorphisms $g\colon \mathcal{M}(d)\to \mathcal{M}(gd)^g$ in $\Mod_R$ for all $g\in G$ such that the right diagram above commutes.  Similarly, a $(\mathcal{C}_R,\mathcal{D})$-bimodule can be described as a spectral functor $\mathcal{N}\colon \cD^{\op}\to {_{R}\Mod}$ together with isomorphisms $g\colon \mathcal{N}(d)\to {^{g}\mathcal{N}(gd)}$ in $_{R}\Mod$ for all $g\in G$.
\end{exmp}

\begin{defn}
For an abelian group $G$ with $g \in G$, $G$-spectral categories $\mathcal{C}$ and $\mathcal{D}$, and $\cM$ a $(\mathcal{C},\mathcal{D})$-bimodule, we define a new $(\mathcal{C},\mathcal{D})$-bimodule $\cM^g$ by the composite
\[
\cM^g: \mathcal{D}^{op} \wedge \mathcal{C} \xrightarrow{g^{\op}\wedge 1} \mathcal{D}^{op} \wedge \mathcal{C} \xrightarrow{M} \Spec.
\]
The designated map
\[
h\colon \cM^g(x,y) \to \cM^g(hx, hy)
\]
is the map
\[
\cM^g(x,y)= \cM(gx,y) \to \cM(hgx,hy) = \cM^g(hx, hy),
\]
where the last equality holds because $G$ is abelian. 
\end{defn}

We now define the cyclic bar construction for a $G$-spectral category.

\begin{defn}
For a small $G$-spectral category $\mathcal{C}$ and a $(\mathcal{C}, \mathcal{C})$-bimodule $\mathcal{M}$, the cyclic bar construction $\Ncyc_{\bullet}(\mathcal{C};\mathcal{M})$ is the simplicial spectrum with $G$-action with $q$-simplices
\[
    \Ncyc_{q}(\mathcal{C};\mathcal{M}) = \bigvee\limits_{(c_0,c_1,\dots, c_q)} \mathcal{C}(c_{q-1},c_q)\wedge \mathcal{C} (c_{q-2},c_{q-1})\wedge\dots\wedge \mathcal{C}(c_0,c_1)\wedge \mathcal{M}(c_q,c_0)
\]
where the wedge is over all $(q+1)$-tuples of objects $c_i\in \mathcal{C}$.  The action of $G$ is given as follows:  the summand corresponding to $(c_0,\dots c_q)$ is sent to the summand $(gc_0,gc_1,\dots gc_q)$ via the map
\[
    \mathcal{C}(c_{q-1},c_q)\wedge\dots\wedge \mathcal{C}(c_0,c_1)\wedge \mathcal{M}(c_q,c_0) \xrightarrow{g\wedge\dots\wedge g} \mathcal{C}(gc_{q-1},gc_q)\wedge \dots\wedge \mathcal{C}(gc_0,gc_1)\wedge \mathcal{M}(gc_q,gc_0).
\]
The  zeroth face map cycles the leftmost smash summand to the far right and then uses the right-module structure of $\mathcal{M}$. The remaining face maps are given by the composition in $\mathcal{C}$ and the left module structure of $\mathcal{M}$.  When the bimodule $\mathcal{M}$ is $\mathcal{C}$ we will write $\Ncyc_{\bullet}(\mathcal{C})$.  Note that in this case the cyclic bar construction is a cyclic object in the sense of Connes.
\end{defn}

Since the cyclic bar construction $\Ncyc_{\bullet}(\cC)$ of a $G$-spectral category $\cC$ is a cyclic object in spectra with $G$-action, its geometric realization is naturally a spectrum with $G\times S^1$-action.  Let $\mathcal{U}$ be a complete $G\times S^1$-universe.

\begin{defn}
    The equivariant THH of a $G$-spectral category $\mathcal{C}$ is the genuine $G\times S^1$-spectrum
    \[
        \ETHH(\mathcal{C}) := \mathcal{I}_{\mathbb{R}^{\infty}}^\mathcal{U}|\Ncyc_{\bullet}(\mathcal{C})|.
    \]
    We will sometimes find it convenient to refer to the cyclic bar construction as a spectrum with $G \times S^1$ action. As in \autoref{rem-eTHH-rings}, we denote it by $\ETHH(\mathcal{C}) : = |\Ncyc_{\bullet}(\mathcal{C})|$.
\end{defn}
Let $\mathcal{V} = i^*_G\mathcal{U}$ be the complete $G$-universe obtained by restricting $\mathcal{U}$.   

\begin{defn}
For $\mathcal{C}$ a $G$-spectral category, and $\mathcal{M}$ a $(\mathcal{C},\mathcal{C})$-bimodule, we define the ETHH of $\mathcal{C}$ with coefficients in $\mathcal{M}$ to be the $G$-spectrum 
\[
\ETHH(\mathcal{C}; \mathcal{M}) = \mathcal{I}_{\mathbb{R}^{\infty}}^\mathcal{V} |\Ncyc_{\bullet}(\mathcal{C};\mathcal{M})|.
\]
\end{defn}
\medskip

When $G = C_n$ we are also able to define a version of twisted $\THH$ for $C_n$-spectral categories.  Let $\Delta_{C_n}\subset C_n\times C_n$ be the diagonal copy of $C_n$, and consider $\Delta_{C_n}$ as a subgroup of $C^n\times S^1$ via the usual embedding of $C_n$ into $S^1$.  As above we use the fixed isomorphism of groups
\[
    (C_n\times S^1)/\Delta_{C_n}\cong S^1
\]
via the map $(e^{2k\pi i/n},e^{\theta i})\mapsto e^{(2k\pi/n-\theta)i}$.  Note that this isomorphism sends $(1,\sigma)$ to $\sigma^{-1}$.

\begin{defn}
For a $C_n$-spectral category $\mathcal{C}$, we define the $C_n$-twisted topological Hochschild homology of $\mathcal{C}$ to be the $S^1$-spectrum
\[
\THH_{C_n}(\mathcal{C}):= \Phi^{\Delta_{C_n}}\ETHH(\mathcal{C}). 
\]
\end{defn}

Note that we have defined $C_n$-twisted THH so that the analog of \autoref{thm: geometric fixed points of ETHH is twisted THH for rings} in the setting of $C_n$-spectral categories is true by definition.  That is, when $\mathcal{C}$ is the one object spectral category of a $C_n$-ring spectrum this agrees with the definition of twisted THH from \cite{AnBlGeHiLaMa}. The advantage of this approach is that $\THH_{C_n}(\mathcal{C})$ is automatically an $S^1$-spectrum. The following reassuring proposition says that twisted THH could have been obtained by a twisted cyclic bar construction.  

\begin{prop}\label{prop: twisted fixed points of cyclic nerve}
   Suppose that $\mathcal{C}$ is a $C_n$-spectral category.  There is an isomorphism of spectra with $C_n$-action
    \[
        i^*_{C_n}|\Ncyc_{\bullet}(\cC)|^{\Delta_{C_n}}\cong |\Ncyc_{\bullet}(\cC,\cC^{\sigma^{-1}})|.
    \]
\end{prop}

Changing universe and applying \autoref{proposition: geometric fixed points and change of universe} yields a version of the above for ETHH.
\begin{cor}\label{cor:twisted THH for spec cats is correct}
Suppose that $\mathcal{C}$ is a $C_n$-spectral category. There is an isomorphism of $C_n$-spectra
\[
    i^*_{C_n}\THH_{C_n}(\mathcal{C}) \cong \ETHH(\mathcal{C};\mathcal{C}^{\sigma^{-1}}). 
\]
where $\sigma$ is the chosen generator of $C_n$.
\end{cor}

Before proving \autoref{prop: twisted fixed points of cyclic nerve} we need a version of \autoref{lemma:TCviaNorm4.4}.  This requires a brief diversion to set up some notation. For a $C_n$-spectral category $\mathcal{C}$ and any $r>1$ we let $\mathcal{C}^{\wedge r}$ denote the $C_n$-spectral category of \cite[Definition 7.4]{CLMPZ} with objects $r$-tuples $(c_1,\dots,c_r)$ of objects in $\mathcal{C}$ and morphism spectra
\[
    \mathcal{C}^{\wedge r}\left((a_1,\dots, a_r),(b_1,\dots, b_r)\right) =\bigwedge\limits_{i=1}^r \mathcal{C}(a_i,b_i).
\]
Note that in addition to the cyclic action of $C_r$ on the objects of $\mathcal{C}^{\wedge r}$, we have maps of morphism spectra
\[
    \mathcal{C}^{\wedge r}\left((a_1,\dots, a_r),(b_1,\dots, b_r)\right)\to \mathcal{C}^{\wedge r}\left((a_r,\dots, a_{r-1}),(b_r,\dots, b_{r-1})\right)
\] given by rotating smash summands. The following lemma is immediate from checking the definitions. 

\begin{lem}
    The assignment $\mathcal{C}\mapsto \mathcal{C}^{\wedge r}$ is a functor from $C_n$-spectral categories to $C_n\times C_r$-spectral categories.  
\end{lem}

For the remainder of this section we focus on the case $r=n$, so that $\mathcal{C}^{\wedge n}$ is a $(C_n \times C_n)$-spectral category.  For any $a\in \mathcal{C}$ let $\widehat{a} = (a,\sigma a,\dots,\sigma^{n-1}a)\in \mathcal{C}^{\wedge n}$. Then for any $a,b\in \mathcal{C}$  we have
\[
    \mathcal{C}^{\wedge n}(\widehat{a},\widehat{b}) = \bigwedge\limits_{i=0}^{n-1} \mathcal{C}(\sigma^ia,\sigma^ib).
\]
The subgroup $\Delta_{C_n} \leq C_n \times C_n$ acts on the spectrum $\mathcal{C}^{\wedge n}(\widehat{a},\widehat{b})$, where the generator $(\sigma,\sigma)$ acts according to
\begin{equation}\label{eq: action of sigma on Cr(a,b)}
    \mathcal{C}^{\wedge n}(\widehat{a},\widehat{b}) = \bigwedge\limits_{i=0}^{n-1} \mathcal{C}(\sigma^ia,\sigma^ib) \to  \bigwedge\limits_{i=0}^{n-1} \mathcal{C}(\sigma^{i-1}a,\sigma^{i-1}b)\xrightarrow{\wedge \sigma} \bigwedge\limits_{i=0}^{n-1} \mathcal{C}(\sigma^ia,\sigma^ib)
\end{equation}
where the first map is a cyclic rotation and the second uses the structure maps
\[
    \sigma\colon \mathcal{C}(\sigma^{i-1}a,\sigma^{i-1}b)\to \mathcal{C}(\sigma^ia,\sigma^ib).
\]

\begin{prop}\label{prop:specCat norm diagonal}
    For any $a,b\in \mathcal{C}$, we have
    \[
        \mathcal{C}^{\wedge n}(\widehat{a},\widehat{b})^{\Delta_{C_n}}\cong \mathcal{C}(a,b).
    \]
\end{prop}
\begin{proof}
    The spectra $\mathcal{C}(\sigma^ia,\sigma^ib)$ and $\mathcal{C}(a,b)$ are isomorphic via the map $\sigma^i$.  Thus one can rewrite $\mathcal{C}^{\wedge r}(\widehat{a},\widehat{b})$ as the smash product of $\mathcal{C}(a,b)$ with itself $n$-times.  With this identification the maps \eqref{eq: action of sigma on Cr(a,b)} become just cyclic permutations. Thus there is an isomorphism of $C_n$-spectra
    \[
        \mathcal{C}^{\wedge n}(\widehat{a},\widehat{b})\cong \wedge_{e}^{C_n}\mathcal{C}(a,b)
    \]
    and the result is immediate upon taking fixed points.
\end{proof}

The construction $\mathcal{C}\mapsto \mathcal{C}^{\wedge n}$ is functioning like a norm in this context.   The next lemma is a generalization of \autoref{lemma:TCviaNorm4.4} for $C_n$-spectral categories.

\begin{lem}\label{lemma:TCviaNorm4.4spectral}
    There is an isomorphism
    \[
        i^*_{C_n\times C_n} |\Ncyc _\bullet(\mathcal{C})|\cong |\Ncyc _\bullet(\mathcal{C}^{\wedge n},(\mathcal{C}^{\wedge n}){^{(1,\sigma)}})|.
    \]
\end{lem}
\begin{proof}
    Since $\Ncyc_{\bullet}$  $(\mathcal{C})$ is a cyclic object in spectra with $C_n$-action, its realization is  isomorphic, as a spectrum with $C_n\times C_n$-action, to the realization of its subdivision $\mathrm{sd}_n\Ncyc_{\bullet}(\mathcal{C})$.  Hence it suffices to show an isomorphism of simplicial spectra with $C_n\times C_n$-action
    \[
        \mathrm{sd}_n\Ncyc_{\bullet}(\mathcal{C})\cong \Ncyc_{\bullet}(\mathcal{C}^{\wedge n}, (\mathcal{C}^{\wedge n}){^{(1,\sigma)}}).
    \]

    Fix a $k$ and let $k_+ = k+1$.  The $k$-simplices of $\mathrm{sd}_n\Ncyc_{\bullet}(\mathcal{C})$ are given by
    \begin{equation}\label{eq: subdivision summand 1}
        \bigvee\limits_{(c_1,\dots c_{nk_+})} \mathcal{C}(c_{nk_+-1},c_{nk_+})\wedge\dots\wedge \mathcal{C}(c_1,c_2)\wedge\mathcal{C}(c_{nk_+},c_1) 
    \end{equation}
    while the $k$-simplices of $\Ncyc_{\bullet}(\mathcal{C}^{\wedge n}, (\mathcal{C}^{\wedge n}){^{(1,\sigma)}})$ are given by
    \begin{equation}\label{eq: subdivision summand 2}
        \bigvee\limits_{(\vec{a}_0,\dots, \vec{a}_k)} \mathcal{C}(a_{k-1}^1,a_k^1)\wedge\mathcal{C}(a_{k-1}^2,a_k^2)\wedge  \dots \wedge \mathcal{C}(a_0^{n},a_1^n)\wedge \mathcal{C}(a_k^{n},a_0^1)\wedge\mathcal{C}(a_k^{1},a_0^2)\wedge\dots\wedge \mathcal{C}(a_k^{n-1},a_0^n)
    \end{equation}
    where the $\vec{a}_i = (a_i^1,\dots a_i^n)$ are objects in $\mathcal{C}^{\wedge n}$. Note that the last $n$-terms of this smash product are unlike the other terms because the superscripts in the source and target are offset by one; this shift is precisely because we took a bar construction in the twisted bimodule $(\mathcal{C}^{\wedge n}){^{(1,\sigma)}}$. 

    The isomorphism between \eqref{eq: subdivision summand 2} and \eqref{eq: subdivision summand 1} is given by sending the summand corresponding to $(\vec{a}_0,\dots, \vec{a}_k)$ in \eqref{eq: subdivision summand 2} to the summand 
    \[
        (a_0^1,a_1^1,a_2^1,\dots,a_k^1,a_0^2,a_1^2,\dots,a_{k-1}^n,a_k^n)
    \]
    in \eqref{eq: subdivision summand 1}.  With this choice the tensor summands are identical up to a reordering isomorphism which is canonical up to isomorphism because the smash product is symmetric monoidal.  Moreover, it is a straightforward, albeit tedious, exercise to see that the face maps in either simplicial spectrum are identified under this reordering.
\end{proof}

\begin{proof}[proof of \autoref{prop: twisted fixed points of cyclic nerve}]

    After applying \autoref{lemma:TCviaNorm4.4spectral} it suffices to show there is an isomorphism
    \[
        |\Ncyc _\bullet(\mathcal{C}^{\wedge n},(\mathcal{C}^{\wedge n}){^{(1,\sigma)}})|^{\Delta_{C_n}}\cong |\Ncyc _\bullet(\mathcal{C},\mathcal{C}^{\sigma})|.
    \]
To show this, note that we can commute the fixed point functor through geometric realization and so we may instead produce an isomorphism of simplicial spectra with $C_n$-action
    \[
        \Ncyc _\bullet(\mathcal{C}^{\wedge n},(\mathcal{C}^{\wedge n})^{(1,\sigma)})^{\Delta_{C_n}}\cong \Ncyc _\bullet(\mathcal{C},\mathcal{C}^{\sigma}).
    \]
The $k$-simplices on the left are given by 
    \begin{equation}\label{eq:filler name}
        \bigvee\limits_{(\vec{a}_0,\dots, \vec{a}_k)} \mathcal{C}^{\wedge n}(\vec{a}_{k-1},\vec{a}_k)\wedge \dots\wedge \mathcal{C}^{\wedge n}(\vec{a}_{0},\vec{a}_1)\wedge (\mathcal{C}^{\wedge n})^{(1,\sigma)}(\vec{a}_{k},\vec{a}_0).
    \end{equation}
    Since $\Delta_{C_n}$ is permuting the tuples $(\vec{a}_0,\dots, \vec{a}_k)$, our first task is to figure out which tuples are fixed.
    
    The action of $\Delta_{C_n}$ on $(\vec{a}_0,\dots, \vec{a}_k)$ is determined by
    \[
        (\sigma,\sigma)\cdot (\vec{a}_0,\dots, \vec{a}_k) = (\vec{b}_0,\dots, \vec{b}_k)
    \]
    with $\vec{b}_i = (\sigma a_i^n,\sigma a_i^1,\sigma a_i^2,\dots,\sigma a_{i}^{n-1})$.  If a summand is going to be fixed under the $\Delta_{C_n}$-action then it must be a summand corresponding to a tuple $(\vec{a}_0,\dots, \vec{a}_k)$ where for all $i$ \[\vec{a}_i   = \widehat{a_i}:= (a_i,\sigma a_i,\sigma^2a_i,\dots,\sigma^{n-1}a_i)\]
    for some $a_i\in \mathcal{C}$.  It follows that the $k$-simplices of $\Ncyc _\bullet(\mathcal{C}^{\wedge n},(\mathcal{C}^{\wedge n})^{(1,\sigma)})^{\Delta_{C_n}}$ are isomorphic to 
    \[
        \bigvee\limits_{(a_0,\dots, a_k)}\left(\mathcal{C}^{\wedge n}(\widehat{a}_{k-1},\widehat{a}_k)\wedge \dots\wedge \mathcal{C}^{\wedge n}(\widehat{a}_{0},\widehat{a}_1)\wedge (\mathcal{C}^{\wedge n}){}^{(1,\sigma)}(\widehat{a}_{k},\widehat{a}_0)\right)^{\Delta_{C_n}}
    \]
    where the wedge is now over any $(k+1)$-tuples in $\mathcal{C}$. Using the strong monoidality of fixed points, this is isomorphic to 
    \[
        \bigvee\limits_{(a_0,\dots, a_k)}\left(\mathcal{C}^{\wedge n}(\widehat{a}_{k-1},\widehat{a}_k)\right)^{\Delta_{C_n}}\wedge \dots\wedge \left(\mathcal{C}^{\wedge n}(\widehat{a}_{0},\widehat{a}_1)\right)^{\Delta_{C_n}}\wedge \left(\mathcal{C}^{\wedge n}(\widehat{\sigma^{n-1}a}_{k},\widehat{a}_0)\right)^{\Delta_{C_n}}
    \]
    where in the last term we used the observation that $(1,\sigma)\cdot\widehat{a} = \widehat{\sigma^{n-1}a}$ for any $a\in \mathcal{C}$. Observing that the action of $\Delta_{C_n}$ on $\mathcal{C}^{\wedge n}(\widehat{a},\widehat{b})$ is precisely \eqref{eq: action of sigma on Cr(a,b)}, we apply \autoref{prop:specCat norm diagonal} and see that the $k$-simplices are 
    \[
    \bigvee\limits_{(a_0,\dots, a_k)}\mathcal{C}(a_{k-1},a_k)\wedge\dots\wedge  \mathcal{C}(a_0,a_1)\wedge \mathcal{C}(\sigma^{-1} a_k,a_0)
    \]
    as desired.  Since the face maps are given entirely in terms of the composition structures of $\mathcal{C}$, monoidality of fixed points implies that the face maps are those of $\Ncyc_{\bullet}(\mathcal{C},\mathcal{C}^{\sigma^{-1}})$. 
\end{proof}
\subsection{Fixed points of the cyclic nerve}

In this subsection we compute the $H$-fixed points of the cyclic bar construction $\Ncyc_{\bullet}(\mathcal{C};\mathcal{M})$ for a $G$-spectral category $\mathcal{C}$ and $(\mathcal{C},\mathcal{C})$-bimodule $M$.  As a first step, note that if $c,d\in \mathcal{C}$ are $H$-fixed then the structure maps
\[
    h\colon \mathcal{C}(c,d)\to \mathcal{C}(hc,hd) = \mathcal{C}(c,d)
\]
make $\mathcal{C}(c,d)$ a spectrum with $H$-action.  We define a spectral category $\mathcal{C}^H$ whose objects are the $H$-fixed objects of $\mathcal{C}$ and whose morphism spectra are $\mathcal{C}^H(c,d) = \mathcal{C}(c,d)^H$.  The composition laws are obtained using the lax monoidality of categorical fixed points (\autoref{prop: fixed points are monoidal}).

For a $(\mathcal{C},\mathcal{C})$-bimodule $\mathcal{M}$ we also see that $\mathcal{M}(c,d)$ is  spectrum with $H$-action whenever $c$ and $d$ are $H$-fixed.  We obtain a $(\mathcal{C}^H,\mathcal{C}^H)$-bimodule $\mathcal{M}^H$ defined by $\mathcal{M}^H(c,d) = \mathcal{M}(c,d)^H$.

\begin{exmp}\label{example: the fixed points of bimodules over perf}
Let $\mathcal{C}=\Mod_R$ be the category of right modules over a ring orthogonal spectrum $R$ with $G$-action.  This is a $G$-spectral category as described in \autoref{example: modules as G-spec cat}.  We consider $\mathcal{C}$ as a bimodule over itself, i.e. $\mathcal{C}(M,N) = F_{R}(M,N)$ is the internal mapping spectrum.  When $M$ and $N$ are $H$-fixed for some $H\leq G$ we have that $F_R(M,N) = F_{R}(M^h,N^h)$ for all $h\in H$ and the action of $H$ on this spectrum is trivial. In particular, the spectrum $(\Mod_R)^H(M,N)$ is just $F_R(M,N)$ again.  The same remarks carry over verbatim if we replace $\Mod_R$ with $\perf_R$, the $G$-spectral category of perfect right $R$-modules.
\end{exmp}

\begin{prop}\label{proposition: fixed points of cyclic nerve}
    Suppose that $\mathcal{C}$ is a spectral $G$-category.  For $H$ a subgroup of $G$ there is an isomorphism of simplicial spectra
    \[
        \mathrm{N}^{\cyc}_{\bullet}(\mathcal{C};\mathcal{M})^H \cong \mathrm{N}^{\cyc}_{\bullet}(\mathcal{C}^H;\mathcal{M}^H)
    \]
    where the fixed points on the left are computed levelwise.
\end{prop}
\begin{proof}
    Before taking $H$-fixed points the $q$-simplices on the left are given by
    \[
         \mathrm{N}^{\cyc}_{q}(\mathcal{C};\mathcal{M}) = \bigvee\limits_{(c_0,c_1,\dots, c_q)} \mathcal{C}(c_{q-1},c_q)\wedge \mathcal{C} (c_{q-2},c_{q-1})\wedge\dots\wedge \mathcal{C}(c_0,c_1)\wedge \mathcal{M}(c_q,c_0).
    \]
    Since $H$ acts first by permuting tuples $(c_0,\dots,c_q)$, we see that the $H$-fixed points of the $q$-simplices must be
    \[
        (\Ncyc_{q}(\mathcal{C};\mathcal{M}))^H = \bigvee\limits_{(c_0,c_1,\dots, c_q)\in (\mathcal{C}^H)^{q+1}} \left(\mathcal{C}(c_{q-1},c_q)\wedge \mathcal{C} (c_{q-2},c_{q-1})\wedge\dots\wedge \mathcal{C}(c_0,c_1)\wedge \mathcal{M}(c_q,c_0)\right)^H.
    \]
    The result now follows from the fact that $H$-fixed points of spectra with $G$-action are strong monoidal.  That is, these fixed points are isomorphic to
    \[
        \Ncyc_{q}(\mathcal{C}^H;\mathcal{M}^H) = \bigvee\limits_{(c_0,c_1,\dots, c_q)\in (\mathcal{C}^H)^{q+1}} \mathcal{C}^H(c_{q-1},c_q)\wedge \mathcal{C}^H (c_{q-2},c_{q-1})\wedge\dots\wedge \mathcal{C}^H(c_0,c_1)\wedge \mathcal{M}^H(c_q,c_0).
    \]
    Since the face and degeneracy maps in the cyclic nerve come from the composition laws in $\mathcal{C}$ and $\mathcal{C}^H$, and composition in the latter is defined using the monoidality of fixed points, we see that these maps assemble into an isomorphism of simplicial spectra.  
\end{proof}

\medskip

\subsection{A model structure on $G$-spectral categories}

Given a $G$-spectral functor $F\colon \cC\to \cD$ there is an induced map $F^H\colon \cC^H\to \cD^H$ for any subgroup $H\leq G$.  

\begin{defn}\label{definition: weak equivalence of G spec cats}
    We say that $F$ is a \emph{weak equivalence} of $G$-spectral categories if it is a bijection on objects and all the functors $F^H$ induced equivalences on morphism spectra.
\end{defn}
 In this section we observe that these weak equivalences are the weak equivalences of a model structure on the category of $G$-spectral categories.  We use this observation to produce useful examples of $G$-spectral categories which are sufficiently cofibrant for our purposes. We begin by recalling the following useful observation, which follows from a result of Schwede--Shipley \cite[Proposition 6.3]{schwede-shipley:equivmonmodcat}.

\begin{prop}
    The category $\speccat$ of spectral categories and spectral functors forms a cofibrantly generated model category where the weak equivalences are spectral functors which induce a bijection on objects and stable equivalences on morphism spectra. If $\mathcal{C}$ is cofibrant in this model structure, it has cofibrant morphism spectra.
\end{prop}

Since the category of $G$-spectral categories is just $\Fun(BG,\speccat)$ we could obtain a model structure on $G$-spectral categories using the projective model structure, but the weak equivalences would be incorrect. 
 Instead, we need a ``more genuine'' model structure. For a finite group $G$, let $\mathcal{O}_G$ denote the orbit category of $G$. We consider the category of presheaves $\Fun(\mathcal{O}_G^{\op},\speccat)$ with the projective model structure.

Any $G$-spectral category $\cC$ determines a functor $i_*(\mathcal{C})\colon \mathcal{O}_G^{\op}\to \speccat$ which is defined by $i_*(\cC)(G/H) = \cC^H$.  If $K\leq H$ then there is a functor $\cC^H\to \cC^K$ which is the inclusion on objects.  The functor $i_*\colon \Fun(BG,\speccat)\to \Fun(\mathcal{O}_G^{\op},\speccat)$ has a left adjoint which sends a functor $F\colon \mathcal{O}_G^{\op}\to \speccat$ to $F(G/e)$, which has an action by $G$ since $\mathcal{O}_{G}^{\op}(G/e,G/e)\cong G$. To produce the desired model structure on $G$-spectral categories, we use a result of Stephan \cite{Stephan}; see also \cite{BMOOP:modelStructureGCat}.

 \begin{thm}\label{thm-model-str-G-spec-cat}
     The category of $G$-spectral categories $\Fun(BG,\speccat)$ has a cofibrantly generated model structure right induced by $i_*\colon \Fun(BG,\speccat)\to \Fun(\mathcal{O}_G ^{op},\speccat)$.  The weak equivalences are those of \autoref{definition: weak equivalence of G spec cats}.   If $\cC$ is a cofibrant object in this model structure then $\cC^H$ is a pointwise cofibrant spectral category for all $H$.
 \end{thm}

 \begin{proof}
     We use Proposition 2.6 of \cite{Stephan}. Fixed points of spectra with $G$-action commute with pushouts and with directed colimits along cofibrations, and therefore the first two conditions of the proposition hold. For a spectral category $\cC$ and a $G$-set $G/K$, the tensor $G/K \otimes \cC \in \Fun(BG, \speccat)$ is defined using an indexed wedge product, which implies that the third condition of the proposition is satisfied as well. Therefore this right-induced model structure exists and is cofibrantly generated as in Proposition 2.6 of \cite{Stephan}. It remains to show that if $\cC$ is a cofibrant object in this model structure, then $\cC^H$ is a pointwise cofibrant spectral category for all $H$. It suffices to show that if $\phi$ is a generating cofibration in $\Fun(BG, \speccat)$, then $\phi^H$ is a cofibration of spectral categories for all $H \leq G$. As in the proof of Proposition 2.6 of \cite{Stephan}, the generating cofibrations in $\Fun(BG, \speccat)$ are of the form $G/K \otimes f$, where $f$ is a generating cofibration in $\speccat$. Thus $(G/K \otimes f)^H = (G/K)^H \otimes f$, which is a wedge of cofibrations and therefore a cofibration.
 \end{proof}
 \begin{notation}\label{notation: cofibrant replacement}
     If $\cC$ is a $G$-spectral category then we write $Q\cC$ for a cofibrant replacement in the model structure of the theorem.
 \end{notation}

It is not clear whether or not the realization of the cyclic nerve of a cofibrant $G$-spectral category is a cofibrant spectrum with $G$-action.  On the other hand, it has the property of being flat, in the sense that smashing with it is homotopical.  There are additional conditions under which $|\Ncyc_{\bullet}(\cC)|$ is a cofibrant spectrum with $G$-action.

\begin{defn}
    We say that a $G$-spectral category $\cC$ is \emph{$G$-pointwise cofibrant} if for all $H\leq G$ and all $H$-fixed pairs of objects $c,d\in \cC$ the mapping $H$-spectrum $\cC(c,d)$ is a cofibrant spectrum with $H$-action.
\end{defn}
\begin{exmp}
    If $R$ is a cofibrant spectrum with $G$-action then $\cC_R$ is $G$-pointwise cofibrant.
\end{exmp}

\begin{prop}
    If $\cC$ is $G$-pointwise cofibrant then $|\Ncyc_{\bullet}(\cC)|$ is a cofibrant spectrum with $G$-action
\end{prop}
\begin{proof}
    It suffices to check that the cyclic nerve is a levelwise cofibrant spectrum with $G$-action.  The $k$-simplices are 
    \[
        \bigvee_{(c_0,\dots,c_k)} \cC(c_{k-1},c_k)\wedge \dots \wedge \cC(c_{0},c_1)\wedge  \cC(c_{k},c_0)
    \]
    The collection of $(k+1)$-tuples $(c_0,\dots,c_k)$ is a $G$-set, and we may pick representatives $(c_0^i,\dots,c_k^i)$ for all the $G$-orbits.  Let $H_i$ be the stabilizer of the tuple $(c_0^i,\dots,c_{k}^i)$, then we have that the $k$-simplices decompose as
    \[
        \bigvee_{i} G_+\wedge_{H_i} (\cC(c^i_{k-1},c^i_k)\wedge \dots \wedge \cC(c^i_{0},c^i_1)\wedge \cC(c^i_{k},c^i_0))
    \]
    where $\cC(c^i_{k-1},c^i_k)\wedge \dots \wedge \cC(c^i_{0},c^i_1)\wedge\cC(c^i_{k},c^i_0)$ is considered as a spectrum with $H_i$-action as all $c_{j}^i$ must be $H_i$-fixed.  Since $\cC$ is $G$-pointwise cofibrant, this spectrum with $H_i$-action is a cofibrant spectrum with $H_i$-action.  Since induction preserves cofibrancy we observe that the $k$-simplices are a wedge of cofibrant spectra with $G$-action hence are cofibrant. 
\end{proof}
 \begin{cor}
     Suppose $\cC$ is a $G$-spectral category which is either cofibrant or $G$-pointwise cofibrant.  Then for any weak equivalence $X\to Y$ of spectra with $G$-action the map $|\Ncyc_{\bullet}(\cC)|\wedge X\to |\Ncyc_{\bullet}(\cC)|\wedge Y$ is also a weak equivalence.
 \end{cor}
 \begin{proof}
    When $\cC$ is $G$-pointwise cofibrant this follows from the fact that $|\Ncyc_{\bullet}(\cC)|$ is cofibrant.  WHen $\cC$ is cofibrant it suffices to check that $|\Ncyc_{\bullet}(\cC)|^H$ is cofibrant for all $H\leq G$.  This follows from \autoref{proposition: fixed points of cyclic nerve} and the fact that $\cC^H$ is a pointwise cofibrant spectral category.
 \end{proof}

 Furthermore, weak equivalences of cofibrant $G$-spectral categories induce weak equivalences on cyclic bar constructions.

 \begin{prop}\label{prop-cof-bar-invariance}
     Let $F\colon \mathcal{C} \to \mathcal{D}$ be a weak equivalence of $G$-spectral categories, where $\mathcal{C}$ and $\mathcal{D}$ are either cofibrant or $G$-pointwise cofibrant.  Then $F$ induces a weak equivalence 
     $$|\Ncyc_{\bullet}(\mathcal{C})|\simeq |\Ncyc_{\bullet}(\mathcal{D})|$$
     of spectra with $G$-action.
 \end{prop}

 \begin{proof}
     We prove the case when $\cC$ is cofibrant, the other case is essentially the same. By Proposition \ref{proposition: fixed points of cyclic nerve}, $|\Ncyc_{\bullet}(\mathcal{C})|^H \cong |\Ncyc_{\bullet}(\mathcal{C}^H)|$. By Theorem \ref{thm-model-str-G-spec-cat}, since $\mathcal{C}$ and $\mathcal{D}$ are cofibrant, for any $H$-fixed objects $c,c' \in \mathcal{C}$ we have that $\mathcal{C}(c,c')^H$ and $\mathcal{D}(F(c), F(c'))^H$ are cofibrant spectra. Furthermore, $F^H$ induces a weak equivalence between these cofibrant spectra. Therefore each level of the cyclic bar constructions $|\Ncyc_{\bullet}(\mathcal{C}^H)|$ and $|\Ncyc_{\bullet}(\mathcal{D}^H)|$ consists of smash products of cofibrant spectra. Thus $F^H$ induces a weak equivalence on the $H$-fixed points of $|\Ncyc_{\bullet}(\mathcal{C})|$ and $ |\Ncyc_{\bullet}(\mathcal{D})|$, and therefore a weak equivalence of spectra with $G$-action.
 \end{proof}

When $\cC$ is cofibrant we can identify maps of bimodules which induce stable equivalences on cyclic nerves.  
\begin{defn}\label{defn:weak equivs of bimodules}
    For $\mathcal{C}$ a $G$-spectral category, a map $f\colon \mathcal{M}\to \mathcal{N}$ of $(\mathcal{C},\mathcal{C})$-bimodules is a \emph{weak equivalence} if for all $H\leq G$ it induces a weak equivalence of spectra $\mathcal{M}^H(c,d)\to \mathcal{N}^H(c,d)$ for all $H$-fixed pairs $(c,d)$ of objects in $\mathcal{C}$.
\end{defn}

\begin{rem}\label{rem: ETHH on homotopy category}
    Let ${_{\cC}\bimod_{\cC}}$ denote the category of equivariant  $(\cC,\cC)$-bimodules. One could proceed as in \cite[Proposition 6.1]{schwede-shipley:equivmonmodcat} to show that ${_{\cC}\bimod_{\cC}}$ is a model category with weak equivalences as defined above. For our purposes, we only need to know what the weak equivalences are and we write $\Ho({_{\cC}\bimod_{\cC}})$ for the associated homotopy category. 
\end{rem}

 \begin{cor}\label{cor: weak equivalences and cyclic nerves}
    Let $\mathcal{C}$ be a $G$-spectral category and let $\alpha\colon \mathcal{M}\to \mathcal{N}$ be a weak equivalence of $(\mathcal{C},\mathcal{C})$-bimodules.  If $\cC$ is cofibrant or $G$-pointwise cofibrant then $\alpha$ induces a weak equivalence  of spectra with $G$-action $|\Ncyc_{\bullet}(\mathcal{C};\mathcal{M})|\simeq |\Ncyc_{\bullet}(\mathcal{C};\mathcal{N})|$.
\end{cor}
\begin{proof}
     We prove the case when $\cC$ is cofibrant, the other case is essentially the same.  By definition, it suffices to check that the map induced by $\alpha$ gives a weak equivalence of spectra on $H$-fixed points for all $H\leq G$.  Commuting the fixed points through the geometric realization and applying \autoref{proposition: fixed points of cyclic nerve} we see that the induced map on $H$-fixed points is the map $|\Ncyc_{\bullet}(\cC^H,\cM^H)|\to |\Ncyc_{\bullet}(\cC^H,\cN^H)|$.  Since $\cC$ is cofibrant the spectral category $\cC^H$ is pointwise cofibrant, and this is a weak equivalence by \cite[Theorem 6.4]{blumberg-mandell:localization}.
\end{proof}

\section{Spectral Waldhausen $G$-categories}\label{sec:SpecWaldG}
Our eventual goal is to construct an equivariant version of the Dennis trace connecting algebraic $K$-theory and THH.  In order to construct this comparison map we must address the fact that the inputs to these theories, Waldhausen and spectral categories, respectively, are not the same.  To bridge this gap we follow the approach of \cite{CLMPZ2} and develop a theory of \emph{spectral Waldhausen $G$-categories}, which have enough structure to define both equivariant algebraic $K$-theory and ETHH.

\begin{defn}
    Let $\mathcal{C}_0$ be a pointed category (i.e. a category with a zero object, which implies that it is enriched over pointed sets). We define its suspension category $\Sigma^\infty \mathcal{C}_0$ to be the spectral category with the same objects as $\mathcal{C}_0$, mapping spectra given by the suspension spectra of the mapping sets of $\mathcal{C}_0$, and composition induced by the composition in $\mathcal{C}_0$.
\end{defn}

\begin{defn}
    A base category of a spectral category $\mathcal{C}$ is a pair $(\mathcal{C}_0, F: \Sigma^\infty \mathcal{C}_0 \to \mathcal{C})$, where $\mathcal{C}_0$ is a pointed category and $F: \Sigma^\infty \mathcal{C}_0 \to \mathcal{C}$ is a spectral functor which is the identity on objects.
\end{defn}

\begin{exmp}\label{ex-can}
    Let $\mathcal{C}$ be a spectral category. Take $\mathcal{C}_0$ to have the same objects as $\mathcal{C}$. For objects $a,b$, take the mappings set $\mathcal{C}_0(a,b)$ to be the 0-space of the mapping spectrum $\mathcal{C}(a,b)$ given the discrete topology. Then the maps $X^{\mathrm{discrete}} \to X$ for any space $X$ and $\Sigma^\infty Y(0) \to Y$ for any spectrum $Y$ give $\mathcal{C}_0$ the structure of a base category for the spectral category $\mathcal{C}$. We denote this functor $can\colon \Sigma^\infty \mathcal{C}_0 \to \mathcal{C}$.
\end{exmp}

\begin{defn}[\cite{CLMPZ}, Definition 3.9]
    A spectral Waldhausen category is a spectral category $\mathcal{C}$ together with a base category $\mathcal{C}_0$, where $\mathcal{C}_0$ has a Waldhausen structure. This data is subject to the conditions:
    \begin{enumerate}
        \item The zero object of $\mathcal{C}_0$ is also a zero object for $\mathcal{C}$.
        \item Every weak equivalence $c \to c'$ in $\mathcal{C}_0$ induces stable equivalences of orthogonal spectra
        $$\mathcal{C}(c', d) \to \mathcal{C}(c, d),\quad \mathcal{C}(d,c) \to \mathcal{C}(d, c').$$
        \item Let $i: a \to b$ be a cofibration in $\mathcal{C}_0$. Then every pushout square
        $$\xymatrix{
        a \ar[r]^-i \ar[d] & b \ar[d]\\
        c \ar[r] & d
        }$$
        induces homotopy pushouts squares
         $$\xymatrix{
        \mathcal{C}(e,a) \ar[r]^-{i_*} \ar[d] & \mathcal{C}(e,b) \ar[d] & \mathcal{C}(a,e)   & \ar[l]^-{i^*} \mathcal{C}(b,e) \\
        \mathcal{C}(e,c) \ar[r] & \mathcal{C}(e,d)& \mathcal{C}(c,e) \ar[u] &   \ar[l] \mathcal{C}(d,e) \ar[u]
        }$$

    \end{enumerate}
\end{defn} 

\begin{defn}\label{def-spwaldcat-fun}
    A functor of spectral Waldhausen categories $F\colon (\mathcal{C}, \mathcal{C}_0) \to (\mathcal{D}, \mathcal{D}_0)$ consists of an exact functor $F_0: \mathcal{C}_0 \to \mathcal{D}_0$ and a spectral functor $F: \mathcal{C} \to \mathcal{D}$ such that the diagram 
    $$\xymatrix{
    \Sigma^\infty \mathcal{C}_0 \ar[r]^-{\Sigma^\infty F_0} \ar[d] & \Sigma^\infty \mathcal{D}_0 \ar[d] \\
    \mathcal{C} \ar[r]^-F & \mathcal{D}
    }$$
    commutes.
\end{defn}

In \cite{CLMPZ}, the authors construct a Dennis trace map $K(\mathcal{C}_0) \to \THH(\mathcal{C})$.
\medskip

For the purposes of algebraic $K$-theory, it is instrumental to define a spectral Waldhausen category of perfect $R$-modules, for $R$ an orthogonal ring spectrum. Due to the technical requirements placed on spectral Waldhausen categories, it will be helpful to pass through EKMM spectra \cite{ekmm}. Let us write $\mathbb{S}$-Mod for the category of EKMM spectra.  Recall there is a symmetric monoidal Quillen adjunction $(\mathbb{N},\mathbb{N}^{\sharp})$

\[\begin{tikzcd}
	\Sp && {\mathbb{S}\mathrm{-Mod}}
	\arrow[""{name=0, anchor=center, inner sep=0}, "{\mathbb{N}}", shift left=2, from=1-1, to=1-3]
	\arrow[""{name=1, anchor=center, inner sep=0}, "{\mathbb{N}^{\sharp}}", shift left=2, from=1-3, to=1-1]
	\arrow["\dashv"{anchor=center, rotate=-90}, draw=none, from=0, to=1]
\end{tikzcd}\]
relating EKMM spectra to orthogonal spectra.  Note that both $\mathbb{N}$ and $\mathbb{N}^{\sharp}$ are lax monoidal.  

If $R$ is a cofibrant orthogonal ring spectrum and $N$ and $M$ are cofibrant right $R$-modules then the mapping spectrum $F_R(M,N)$ need not be fibrant.   We use the adjunction $\mathbb{N}\dashv \mathbb{N}^{\sharp}$ to produce homotopically well behaved mapping spectra $\mathbb{N}^{\sharp}F_{\mathbb{N}R}(\mathbb{N}M,\mathbb{N}N)$.  This spectrum is always fibrant since all EKMM spectra are fibrant and $\mathbb{N}^{\sharp}$ is a right Quillen adjoint.

\medskip

The following is the key example in this section.

\begin{exmp}[\cite{CLMPZ2} 3.8-3.12, 5.2]\label{ex-stperf}
    If $R$ is an orthogonal ring spectrum whose underlying spectrum is cofibrant, there is a spectral Waldhausen category $\stperf_R$ whose base category is the Waldhausen category $\perf_R$ of perfect cofibrant $R$-module spectra. Take the objects of $\stperf_R$ to be the perfect cofibrant modules over $R$. For objects $P,P'$ of $\stperf_R$, define $\stperf_R(P, P')$ to be the orthogonal spectrum $ \mathbb{N}^{\sharp} F_{\mathbb{N} R}(\mathbb{N}P, \mathbb{N}P')$. The spectral functor $\Sigma^\infty \perf_R \to \stperf_R$ is given by the adjoint of
    \[ \mathbb{N}\Sigma_{\mathrm{orth}}^{\infty}\perf_R(P,P')\xrightarrow{\cong} \Sigma_{\mathrm{EKMM}}^{\infty}\perf_R(P,P')\xrightarrow{\Sigma_{\mathrm{EKMM}}^{\infty}\mathbb{N}} \Sigma_{\mathrm{EKMM}}^{\infty}\perf_{\mathbb{N}R}(\mathbb{N}P,\mathbb{N}P')\xrightarrow{can} F_{\mathbb{N}R}(\mathbb{N}P,\mathbb{N}P')
    \]
    where $can$ is as in \autoref{ex-can}. All EKMM spectra are fibrant, and $\mathbb{N}^\sharp$ is a right Quillen adjoint, so all mapping spectra in this category are fibrant.
    \end{exmp}

    \begin{rem}\label{remark: spectral functor from stperf to Mod}
        Later we will need a comparison of $\stperf_R$ and the more natural spectral category $\perf_{R}$ defined in \autoref{example: perf as G-spec Cat}.  One can give a spectral functor $\perf_R\to \stperf_R$, which is the identity on objects, but for technical reasons we would like a map which goes the other way.  It seems difficult to construct such a functor directly, however it is straightforward to construct a functor
        \[
            \Omega\colon \stperf_R\to \Mod_{\mathbb{N}^{\sharp}\mathbb{N}R}
        \]
        which sends an object $P$ to $\mathbb{N}^{\sharp}\mathbb{N}P$.  On morphisms, we use the map
        \[
            \mathbb{N}^{\sharp}F_{\mathbb{N}R}(\mathbb{N}P,\mathbb{N}Q)\to F_{\mathbb{N}^{\sharp}\mathbb{N}R}(\mathbb{N}^{\sharp}\mathbb{N}P,\mathbb{N}^{\sharp}\mathbb{N}Q)
        \]
        which comes from the fact that $\mathbb{N}^{\sharp}$ is lax monoidal.
    \end{rem}

\begin{exmp}[\cite{CLMPZ}, 3.7 and 3.12]
    If $(\mathcal{C}, \mathcal{C}_0)$ is a spectral Waldhausen category and $I$ is a small category, then there is a spectral Waldhausen structure on the functor categories $(\Fun(I,\mathcal{C}), \Fun(I, \mathcal{C}_0))$.
\end{exmp}

This functor category construction is useful for applying the $S_\bullet$-construction to spectral Waldhausen categories. Denote by $[k]$ the category $\{ 0 \to 1 \to ... \to k \}$. If $(\mathcal{C},\mathcal{C}_0)$ is a spectral Waldhausen category, consider $S_k \mathcal{C}_0$ as a subcategory of the functor category $\Fun([k] \times [k], \mathcal{C}_0)$ as in Definition 5.8 (and the discussion that follows it) of \cite{CLMPZ2}.

\begin{prop}[\cite{CLMPZ2}]\label{prop-Sdot-spectral}
  Let $(\mathcal{C}, \mathcal{C}_0)$ be a spectral Waldhausen category. Take $S_k \mathcal{C}$ to be the full subcategory of $\Fun([k] \times [k], \mathcal{C})$ spanned by the objects of $S_k \mathcal{C}_0$. Then $(S_k \mathcal{C}, S_k \mathcal{C}_0)$ is a spectral Waldhausen category. Furthermore, as $k$ varies, this defines a simplicial object in $\mathrm{SpWaldCat}$, $(S_\bullet \mathcal{C}, S_\bullet \mathcal{C}_0)$.
\end{prop}

The functor category construction also gives a spectral Waldhausen category of weak equivalences in $(\mathcal{C}, \mathcal{C}_0)$.

\begin{prop}[\cite{CLMPZ2}, Definition 5.6]\label{prop-wdot-spectral} 
    Let $(\mathcal{C}, \mathcal{C}_0)$ be a spectral Waldhausen category. Denote by $w_k \mathcal{C}_0$ the full subcategory of $\Fun([k], \mathcal{C}_0)$ spanned by the functors that take each morphism in $[k]$ to a weak equivalence in $\mathcal{C}_0$. Denote by $w_k \mathcal{C}$ the full subcategory of $\Fun([k], \mathcal{C})$ spanned by the objects of $w_k \mathcal{C}_0$. Then $(w_k \mathcal{C}, w_k \mathcal{C}_0)$ is a spectral Waldhausen category. Furthermore, as $k$ varies, this defines a simplicial object in $\mathrm{SpWaldCat}$, $(w_\bullet \mathcal{C}, w_\bullet \mathcal{C}_0)$.
\end{prop}
\begin{rem}\label{remark: iterated S dot for spectral Wald cats}
    We can also iterate the $S_{\bullet}$ construction as is \autoref{definition: iterated S dot}.  Letting $S^{(n)} _{k_1, ..., k_n} \mathcal{C} = S_{k_1} ... S_{k_n} \mathcal{C}$, we therefore obtain a multisimplicial object $(w_\bullet S^{(n)} _\bullet \mathcal{C}, w_\bullet S^{(n)} _\bullet \mathcal{C}_0)$ in  $\mathrm{SpWaldCat}$, whose $(k_0, ..., k_n)$ level is $(w_{k_0}S^{(n)} _{k_1, ..., k_n} \mathcal{C}, w_{k_0} S^{(n)} _{k_1, ..., k_n} \mathcal{C}_0)$.
\end{rem}

We will next define Waldhausen $G$-categories and spectral Waldhausen $G$-categories. Recall that Malkiewich--Merling define a notion of Waldhausen $G$-categories.

\begin{defn}[\cite{MM1}, Section 2.3]\label{defn: Waldhausen G-cat}
Let WaldCat denote the category of Waldhausen categories and exact functors. A \emph{Waldhausen $G$-category} is a functor
\[
\mathrm{B}G \to \mathrm{WaldCat}.
\]  

\begin{exmp}\label{ex-perf-action}
    If $R$ is an orthogonal $G$-ring spectrum whose underlying spectrum is cofibrant, then $\perf_R$ as defined in \autoref{ex-stperf} inherits the structure of a Waldhausen $G$-category. The $G$-action sends a right $R$-module $P$ to the $g$-twisted right $R$-module $P^{g^{-1}}$.
\end{exmp}

\begin{exmp}[\cite{MM1}, Section 2.3]\label{example: Fun(EG,C)}
    If $\mathcal{C}$ is a Waldhausen $G$-category, then $\Fun(\EG, \mathcal{C})$ is a Waldhausen $G$-category with $G$-action given by conjugation.  That is, if $F\colon \EG\to \mathcal{C}$ is any functor then we define $(g\cdot F)(h) = gF(g^{-1}h)$.
\end{exmp}

\end{defn}
We now define spectral Waldhausen $G$-categories. Let $\mathrm{SpWaldCat}$ denote the category of spectral Waldhausen categories.
\begin{defn}
 A \emph{spectral Waldhausen $G$-category} is a functor 
 \[
\mathrm{B}G \to \mathrm{SpWaldCat}.
 \]
\end{defn}

Note that if $\mathcal{C}$ is a spectral Waldhausen $G$-category, then the base category $\mathcal{C}_0$ is a Waldhausen $G$-category and the spectral category $\cC$ is a $G$-spectral category.

\begin{exmp}
    Let $R$ be an orthogonal ring spectrum with $G$-action, and denote $g\colon R \to R$ the action of an element $g \in G$ on $R$. Then $\stperf_R$ obtains a $G$-action which sends a right $R$-module $P$ to the $g$-twisted right $R$-module $P^{g^{-1}}$.  Applying $\mathbb{N}$ to the map $g\colon F_{R}(P,Q)\to F_R(P^{g^{-1}},Q^{g^{-1}})$ from \autoref{example: modules as G-spec cat} yields a map 
    \begin{equation}\label{equation: G aciton on stperf mapping spectra}
        F_{\mathbb{N}R}(\mathbb{N}P,\mathbb{N}Q)\to F_{\mathbb{N}R}((\mathbb{N}P)^{g^{-1}},(\mathbb{N}Q)^{g^{-1}})
    \end{equation} 
    where we abuse notation slightly and write $g\colon \mathbb{N}R\to \mathbb{N}R$ for the map obtained by applying $\mathbb{N}$ to $g\colon R\to R$.  Applying $\mathbb{N}^{\sharp}$ to \eqref{equation: G aciton on stperf mapping spectra} gives the action of $G$ on the mapping spectra of $\stperf_R$.

    This $G$-action is via functors of spectral Waldhausen categories, as in \autoref{def-spwaldcat-fun}. This implies that $\stperf_R$ is a spectral Waldhausen $G$-category. Its base category, $\perf_R$, is then a Waldhausen $G$-category.
\end{exmp}

\begin{exmp}\label{example: fun(EG) spectral cat}
    If $(\mathcal{C}, \mathcal{C}_0)$ is a spectral Waldhausen $G$-category, there is a spectral Waldhausen $G$-category $\Fun(\mathcal{E}G, \mathcal{C})$  with base category $\Fun(\mathcal{E}G, \mathcal{C}_0)$. This is because Sections 4 and 5 of \cite{CLMPZ2} construct a functor
    $$\Fun(-,-): \mathrm{Cat}^{\mathrm{op}} \times \mathrm{SpWaldCat} \to \mathrm{SpWaldCat}.$$
    If the base category of $\mathcal{C}$ is $\mathcal{C}_0$, then the base category of $\Fun(I,\mathcal{C})$ is $\Fun(I, \mathcal{C}_0)$.
    The functoriality of this construction ensures that elements of $G$ act by exact functors, so that $\Fun(\mathcal{E}G, \mathcal{C})$ is a spectral Waldhausen $G$-category.
\end{exmp} 

We therefore obtain a spectral Waldhausen $G$-category $\Fun(\mathcal{E}G, \stperf_R)$, whose base category is $\Fun(\mathcal{E}G, \perf_R)$.

\begin{exmp}
    The $S_\bullet$ and $w_\bullet$ constructions of Proposition \ref{prop-Sdot-spectral} and Proposition \ref{prop-wdot-spectral} respect $G$-actions, giving $S_\bullet$ and $w_\bullet$ constructions for spectral Waldhausen $G$-categories. Applying these constructions to $\Fun(\mathcal{E}G, \mathcal{C})$ will be especially useful in section 8.
\end{exmp}

\section{Morita adjunctions}\label{sec: Morita}

One of the essential properties of topological Hochschild homology is that it is Morita invariant.  In this section we develop a theory of Morita adjunctions and Morita equivalences for $G$-spectral categories and show that $\ETHH$ is Morita invariant.  The results in this section are used in the construction of the Dennis trace to produce maps on $\ETHH$.

Let $\mathcal{C}$, $\mathcal{D}$, and $\mathcal{E}$ be three $G$-spectral categories and let $\mathcal{N}$ and $\mathcal{M}$ be $(\mathcal{C},\mathcal{E})$- and $(\mathcal{D},\mathcal{C})$-bimodules respectively.

\begin{defn}
    The \emph{two-sided bar construction} of $\mathcal{M}$, $\mathcal{C}$, and $\mathcal{N}$, written $B(\mathcal{M};\mathcal{C},\mathcal{N})$, is the $(\mathcal{D},\mathcal{E})$-bimodule which sends $(e,d)\in \mathcal{E}^{\op}\wedge \mathcal{D}$ to the realization of the simplicial spectrum 
    \[
        \bigvee\limits \mathcal{M}(c_q,d)\wedge \mathcal{C}(c_{q-1},c_q)\wedge\dots\wedge \mathcal{C}(c_0,c_1)\wedge \mathcal{N}(e,c_0)
    \]
    where the wedge is over $(q+1)$-tuples $(c_0,\dots,c_q)$  of objects in $\mathcal{C}$.
\end{defn}

One should think about the two-sided bar construction as a way to tensor $(\mathcal{C},\mathcal{E})$- and $(\mathcal{D},\mathcal{C})$-bimodules together over $\mathcal{C}$.  To illustrate this further, suppose $R$ is a ring spectrum with $G$-action which is cofibrant, write $\mathcal{C}_R$ for the one object $G$-spectral category associated to $R$, and suppose $M$ and $N$ are $(\mathcal{C}_R,\mathcal{C}_R)$-bimodules, i.e.\ $(R,R)$-bimodules in spectra.  Then the simplicial object underlying the two-sided bar construction is exactly the usual two-sided bar resolution of $M\wedge_R N$.  In particular, its geometric realization is exactly the left derived smash product of $M$ and $N$ over $R$.

Note that when $N = R$, and $R$ is cofibrant, we get an equivalence $B(M;R;R)\simeq M\wedge_R^{\mathbb{L}} R\simeq M$. The next lemma says that the analogous statement is true for arbitrary $G$-spectral categories.

\begin{lem}[{Two Sided Bar Lemma, \cite[Lemma 1.40]{BlumbergMandell:localizationLongOne}}]
    Let $\mathcal{C}$, $\mathcal{D}$, and $\mathcal{E}$ be $G$-spectral categories and let $\mathcal{M}$ and $\mathcal{N}$ be a $(\mathcal{D},\mathcal{C})$ and $(\mathcal{C},\mathcal{E})$-bimodules respectively.  There are weak equivalences
    \begin{align*}
        B(\mathcal{M};\mathcal{C};\mathcal{C}) & \xrightarrow{\sim} \mathcal{M}\\
        B(\mathcal{C};\mathcal{C};\mathcal{N})& \xrightarrow{\sim} \mathcal{N}
    \end{align*}
    of $(\mathcal{D},\mathcal{C})$ and $(\mathcal{C},\mathcal{E})$-bimodules respectively.
\end{lem}
\begin{proof}
    The maps
    \[
        \cM(c_q,d)\wedge\cC(c_{q-1},c_q)\wedge\dots\wedge\cC(e,c_0)\to  \cM(e,d)
    \]
    given by composition assemble into a map of simplicial spectra with $G$-action from $B(\cM;\cC;\cC)$ to the constant simplicial object on $\cM(e,d)$. We call this the composition map. Taking realizations gives a map on bimodules, which we immediately see is compatible with the $G$-actions.  It is classical (see, for instance, \cite[Proposition 9.8]{May:GeometryIteratedLoops}) that the composition map admits a section which is essentially given by the inclusion of many identity morphisms.  Moreover, one checks that the composition map and the section form a homotopy equivalence after geometric realization.
    
    If $d$ and $e$ are $H$-fixed for some $H \leq G$, then both the composition map and the section are $H$-equivariant, thus the homotopy equivalence descends to fixed points and this map is a weak equivalence of bimodules.
\end{proof}

\begin{cor}[{Dennis--Waldhausen--Morita argument, \cite[Proposition 1.39]{BlumbergMandell:localizationLongOne}}]\label{cor:Dennis--Wald--Morita}
    Let $\mathcal{C}$ and $\mathcal{D}$ be small $G$-spectral categories and let $\mathcal{N}$ and $\mathcal{M}$ be $(\mathcal{C},\mathcal{D})$ and $(\mathcal{D},\mathcal{C})$-bimodules, respectively.  Then there are natural isomorphisms of spectra with $G$-action
    \[
        |\Ncyc_{\bullet}(\mathcal{C};B(\mathcal{N};\mathcal{D};\mathcal{M}))|\cong |\Ncyc_{\bullet}(\mathcal{D};B(\mathcal{M};\mathcal{C};\mathcal{N}))|
    \]
    and hence isomorphisms 
    \[
    \ETHH(\mathcal{C};B(\mathcal{N};\mathcal{D};\mathcal{M}))\cong \ETHH(\mathcal{D};B(\mathcal{M};\mathcal{C};\mathcal{N}))
    \]
\end{cor}
\begin{proof}
    The proof is identical to that of \cite[Proposition 1.39]{BlumbergMandell:localizationLongOne}.
\end{proof}

The above corollary gives a quick proof of Morita invariance for $\ETHH$.
\begin{defn}
    A \emph{Morita adjunction} between two $G$-spectral categories $\mathcal{C}$ and $\mathcal{D}$ is a pair $(\mathcal{N}\dashv\mathcal{M})$ where   $\mathcal{N}$ is a $(\mathcal{C},\mathcal{D})$-bimodule, $\mathcal{M}$ is a $(\mathcal{D},\mathcal{C})$-bimodule, and morphisms
    \begin{align*}
        \eta & \colon \mathcal{D}  \to B(\mathcal{M};\mathcal{C};\mathcal{N})\\
        \epsilon &  \colon B(\mathcal{N};\mathcal{D};\mathcal{M})\to \mathcal{C}. 
    \end{align*}
 in the homotopy categories $\Ho(_{\cD}\bimod_{\cD})$ and $\Ho(_{\cC}\bimod_{\cC})$ which satisfy the triangle identities. A Morita adjunction is called a \emph{Morita equivalence} if in addition the maps $\eta$ and $\epsilon$ are equivalences.
\end{defn}

Suppose we have a Morita adjunction $(\mathcal{N}\dashv\mathcal{M})$ between $G$-spectral categories $\mathcal{C}$ and $\mathcal{D}$ which are either cofibrant or $G$-pointwise cofibrant.  While $\eta$ and $\epsilon$ are only defined in the homotopy category, the cofibrancy assumptions on $\cC$ and $\cD$ allow us to apply \autoref{cor: weak equivalences and cyclic nerves} to obtain a map $|\Ncyc_{\bullet}(\mathcal{D})|\to |\Ncyc_{\bullet}(\mathcal{C})|$ in the homotopy category of spectra with $G$-action given by the composite
    \[
        \Ncyc_{\bullet}(\cD;\cD)\xrightarrow{\eta_*} \Ncyc_{\bullet}(\cD;B(\mathcal{M};\mathcal{C};\mathcal{N})) \cong \Ncyc_{\bullet}(\cC;B(\mathcal{N};\mathcal{D};\mathcal{M}))\xrightarrow{\epsilon_*} \Ncyc_{\bullet}(\cC;\cC)
    \]
    where the isomorphism comes from \autoref{cor:Dennis--Wald--Morita}.

    \begin{notation}
        We will call a Morita adjunction $(\cN\dashv \cM)$ as above a Morita adjunction \emph{from $\cD$ to $\cC$} to emphasize the direction of the induced map on cyclic nerves.
    \end{notation}
    
    By Remarks 7.6 and 7.9 of \cite{CLMPZ}, using $n$-fold subdivision, this is a map in the homotopy category of spectra with $G \times C_n$-action for every $n$. When $(\mathcal{N}\dashv\mathcal{M})$ is a Morita equivalence the maps $\eta_*$ and $\epsilon_*$ are isomorphisms, hence this map is an isomorphism.  If, in addition, $\cC$ and $\cD$ are $G$-pointwise cofibrant, applying the change of universe functor computes $\ETHH$ and we obtain Morita invariance.
    \begin{thm}\label{theorem: Morita invariance}
        If $\cC$ and $\cD$ are $G$-pointwise cofibrant $G$-spectral categories and $(\mathcal{N}\dashv \mathcal{M})$ is a Morita equivalence then there is a natural isomorphism $\ETHH(\cC)\cong \ETHH(\cD)$ in the homotopy category of spectra with $(G\times C_n)$-action for any $n$.
    \end{thm}

Let $R$ be a cofibrant orthogonal ring spectrum with $G$-action.  Our next goal is to construct a Morita adjunction from $\Fun(\EG,\stperf_R)$ to $\cC_R$.  Non-equivariantly this adjunction is a Morita equivalence, but that will not be the case here.  Nevertheless, the adjunction is sufficient to produce the morphism on $\ETHH$ that we need to produce the equivariant Dennis trace.

Recall that $\mathcal{C}_R$ is the one object (call it $\bullet$) $G$-spectral category with morphisms $\mathcal{C}_R(\bullet,\bullet) =R$ and $\stperf_R$ is the $G$-spectral category of perfect \emph{right} $R$-modules from \autoref{ex-stperf}.  Recall that the objects of $\stperf_R$ are cofibrant perfect right $R$-modules and the morphism spectra are given by 
\[
    \stperf_R(P,P') = \mathbb{N}^{\sharp}\Mod_{\mathbb{N}R}(\mathbb{N}P,\mathbb{N}P')
\]  
for any objects $P$ and $P'\in \stperf_R$.  Finally, for any $G$-spectral category $\cC$ we will need the functor $G$-spectral category $\Fun(\EG,\cC)$ constructed in \autoref{example: fun(EG) spectral cat}.

\begin{prop}\label{proposition: Morita adjunction one}
    Let $\cC$ be a small, cofibrant $G$-spectral category. Let $\cD = Q\Fun(\EG,\cC)$ where $Q$ denotes a cofibrant replacement in the model structure of \autoref{thm-model-str-G-spec-cat}. Then there is a Morita adjunction $(\cN \dashv \cM)$ from $Q\Fun(\EG,\cC)$ to $\cC$.
\end{prop}
\begin{proof}
    There is a constant spectral functor $\cC\to \Fun(\EG,\cC)$, and because $\cC$ is cofibrant this lifts to a functor $f\colon \cC\to \cD$ which, on objects, is given by the constant spectral functor.  Given any object $\Theta\in \cD$, we obtain an object $\Theta(1)\in \cC$, though we should note that this assignment is not a $G$-spectral functor since it is not $G$-equivariant on the nose. 

    Define a $(\cC,\cD)$-bimodule $\cN$ by
    \[
        \cN(\Theta,c) = \cC(\Theta(1),c)
    \]
    and a $(\cD,\cC)$-bimodule $\cM$ by
    \[
        \cM(c,\Theta) = \cC(c,\Theta(1)).
    \]
    To see that $\cN$ is a spectral functor observe that it is the composite
    \[
        Q\Fun(\EG,\cC)^{\op}\wedge \cC\to \Fun(\EG,\cC)^{\op}\wedge \cC\xrightarrow{\mathrm{ev}_1\wedge \mathrm{id}} \cC^{\op}\wedge \cC\xrightarrow{\cC(-,-)}\Sp
    \]
    where the unlabeled map is the map which witnesses $Q$ as a cofibrant replacement in $G$-spectral categories and $\mathrm{ev}_1$ is the evaluation at $1$ map.  This map is, by definition, a spectral functor although not a $G$-spectral functor.

    For the structure maps, for any $g\in G$ we need to produce maps
    \[
        \cC(\Theta(1),c) = \cN(\Theta,c)\to\cN(g\cdot \Theta,gc) = \cC(g\Theta(g^{-1}),gc)
    \]
    and 
    \[
        \cC(c,\Theta(1)) = \cM(c,\Theta)\to \cM(gc,g \cdot\Theta) = \cC(gc,g\Theta(g^{-1})).
    \]
    To construct these, first observe that part of the structure of $\Theta$ is a chosen isomorphism $\Theta(1)\cong \Theta(g^{-1})$.  With this, we have maps
    \begin{align*}
      \cC(\Theta(1),c) & \cong \cC(\Theta(g^{-1}),c)\xrightarrow{g} \cC(g\Theta(g^{-1}),gc)  \\
      \cC(c,\Theta(1)) & \cong \cC(c,\Theta(g^{-1}))\xrightarrow{g} \cC(gc,g\Theta(g^{-1})) 
    \end{align*} 
    which give the structure maps.

    The bar construction $B(\cM;\cC;\cN)$ is the $(\cD,\cD)$-bimodule whose value on a pair $(\Theta,\Psi)$ is the realization of the simplicial spectrum with $q$-simplices
    \[
        \bigvee_{(c_0,\dots,c_q)} \cC(c_q,\Psi(1))\wedge \cC(c_{q-1},c_q)\wedge\dots\wedge \cC(c_0,c_1)\wedge \cC(\Theta(1),c_0). 
    \]
    If we pick $c_0=\dots = \Theta(1)$ then we have an evident map
    \[
        \cC(\Theta(1),\Psi(1)) \cong \cC(\Theta(1),\Psi(1))\wedge \mathbb{S}\wedge\dots\wedge \mathbb{S}\to \cC(\Theta(1),\Psi(1))\wedge \cC(\Theta(1),\Theta(1))\wedge\dots\wedge \cC(\Theta(1),\Theta(1))
    \]
    where the unlabeled map is the smash product of the identity on $\cC(\Theta(1),\Psi(1))$ with the unit maps $\mathbb{S}\to \cC(\Theta(1),\Theta(1))$.  This assembles into a map from the constant simplicial spectrum on $\cC(\Theta(1),\Psi(1))$, and taking geometric realization yields a map $\cC(\Theta(1),\Psi(1))\to B(\cM;\cC;\cN)(\Theta,\Psi)$.  The map $\eta$ is, evaluated at the pair $(\Theta,\Psi)$, the composite
    \[
        \cD(\Theta,\Psi)\xrightarrow{ev_1} \cC(\Theta(1),\Psi(1))\to B(\cM;\cC;\cN)(\Theta,\Psi).
    \]

    The bar construction $B(\cN;\cD;\cM)$ is given at the pair $(x,y)\in \mathrm{ob}\cC \times \mathrm{ob}\cC$ by the realization of the simplicial spectrum with $q$-simplices
    \[
        \bigvee_{(\Theta_1,\dots,\Theta_q)} \cC(\Theta_q(1),y)\wedge \cD(\Theta_{q-1},\Theta_q)\wedge\dots\wedge \cC(x,,\Theta_0(1)).
    \]
    Applying the evaluation maps $\cD(\Theta_{i-1},\Theta_i)\to \cC(\Theta_{i-1}(1),\Theta_i(1))$ followed by the composition yields a map to the constant simplicial spectrum on $\cC(x,y)$ and geometrically realizing gives the $(x,y)$-component of the map $\epsilon$.  

    Checking the triangle identities is essentially immediate, since the map $\eta$ is given by including identity morphisms and $\epsilon$ is given by composition.
\end{proof}

\begin{prop}\label{proposition: Morita adjunction two}
    Let $R$ be a cofibrant ring $G$-spectrum, let $\cC_{\mathbb{N}^{\sharp}\mathbb{N}R}$ be the one object $G$-spectral category whose morphism spectrum is $\mathbb{N}^{\sharp}\mathbb{N}R$, and let $\cD_R = Q\mathrm{stPerf}_R$ be a cofibrant replacement of the $G$-spectral category of perfect right $R$-modules.  There is Morita adjunction from $\cD_R$ to $\cC_{\mathbb{N}^{\sharp}\mathbb{N}R}$. 
\end{prop}
\begin{proof}
    By \autoref{example: bimodules over one object spec cats}, a $(\cC_{\mathbb{N}^{\sharp}\mathbb{N}R},\cD_R)$-bimodule is a spectral functor $\cN\colon\cD_R^{\op}\to {_{\mathbb{N}^{\sharp}\mathbb{N}R}\Mod}$ together with isomorphisms
    \[
        \cN(P)\to {{^g}\cN(gP)}
    \]
    for all $g\in G$.  Similarly, a $(\cD_R,\cC_{\mathbb{N}^{\sharp}\mathbb{N}R})$-bimodule consists of a spectral functor $\cM\colon \cD\to \Mod_{\mathbb{N}^{\sharp}\mathbb{N}R}$  with isomorphisms
    \[
        \cM(P)\to \cM(gP)^g
    \]
    for all $g\in G$.  In this case, we take $\cN(P) = F_{\mathbb{N}^{\sharp}\mathbb{N}R}(\mathbb{N}^{\sharp}\mathbb{N}P,\mathbb{N}^{\sharp}\mathbb{N}R)$ and $\cM(P) =\mathbb{N}^{\sharp}\mathbb{N}P$. To see these are spectral functors, observe that $\cN$ is the composite
    \[
        Q\stperf_R^{\op}\to \stperf_R^{\op}\xrightarrow{\Omega^{\op}} {_{\mathbb{N}^{\sharp}\mathbb{N}R}\Mod}^{\op}\to {_{\mathbb{N}^{\sharp}\mathbb{N}R}\Mod}
    \]
    where $\Omega$ is the functor from \autoref{remark: spectral functor from stperf to Mod} and the unlabeled functor is given by taking duals.
    
    The functor $\cM$ is the composite
    \[
        Q\stperf_R\to \stperf_R\xrightarrow{\Omega} \Mod_{\mathbb{N}^{\sharp}\mathbb{N}R}.
    \]
    
    By definition, $gP = P^{g^{-1}}$, so $(gP)^g = P$.  Since $\mathbb{N}^{\sharp}\mathbb{N}$ is a functor, we have $g\cdot (\mathbb{N}^{\sharp}\mathbb{N}P)^{g} = \mathbb{N}^{\sharp}\mathbb{N}P$ so the identity provides the structure isomorphism for $\cM$.  For $\cN$, we have 
    \[
        ^{g}F_{\mathbb{N}^{\sharp}\mathbb{N}R}(g\mathbb{N}^{\sharp}\mathbb{N}P,\mathbb{N}^{\sharp}\mathbb{N}R)\cong F_{\mathbb{N}^{\sharp}\mathbb{N}R}(g\mathbb{N}^{\sharp}\mathbb{N}P,{^{g}\mathbb{N}^{\sharp}\mathbb{N}R})\cong F_{\mathbb{N}^{\sharp}\mathbb{N}R}(\mathbb{N}^{\sharp}\mathbb{N}P^{g^{-1}},\mathbb{N}^{\sharp}\mathbb{N}R^{g^{-1}})\cong F_{\mathbb{N}^{\sharp}\mathbb{N}R}(\mathbb{N}^{\sharp}\mathbb{N}P,\mathbb{N}^{\sharp}\mathbb{N}R)
    \]
    by \autoref{lemma: right twisted mapping set} and \autoref{lemma: left twisted mapping set} and we take these as our structure isomorphisms.  

    The bar construction $B(\cM;\cC_{\mathbb{N}^{\sharp}\mathbb{N}R};\cN)$ is the $(\cD_R,\cD_R)$-bimodule whose value on a pair of perfect right $R$-modules $(P,P')$  is the realization of the simplicial spectrum whose $q$-simplices are
    \[
        \mathbb{N}^{\sharp}\mathbb{N}P\wedge \mathbb{N}^{\sharp}\mathbb{N}R\wedge \dots\wedge \mathbb{N}^{\sharp}\mathbb{N}R\wedge F_{\mathbb{N}^{\sharp}\mathbb{N}R}(\mathbb{N}^{\sharp}\mathbb{N}P',\mathbb{N}^{\sharp}\mathbb{N}R).
    \]
    Since $\mathbb{N}^{\sharp}$ is lax monoidal there is a map to the simplicial spectrum with $q$-simplices
    \[
        \mathbb{N}^{\sharp}\left(\mathbb{N}P\wedge\mathbb{N}R\wedge \dots\wedge \mathbb{N}R\wedge F_{\mathbb{N}R}(\mathbb{N}P',\mathbb{N}R)\right)
    \]
    
    This complex is $\mathbb{N}^{\sharp}$ applied to the two-sided bar resolution of $\mathbb{N}P\wedge_{\mathbb{N}R} F_{\mathbb{N}R}(\mathbb{N}P',\mathbb{N}R)$, and because $\mathbb{N}P'$ is cofibrant and perfect the realization is isomorphic, in the homotopy category, to $\mathbb{N}^{\sharp}F_{\mathbb{N}R}(\mathbb{N}P',\mathbb{N}P) = \stperf_R(P,P')\cong \cD_{R}(P,P')$.  This isomorphism gives the map $\eta$.

    The bar construction $B(\cN;\cD_R;\cM)$ is the $(\cC_{\mathbb{N}^{\sharp}\mathbb{N}R},\cC_{\mathbb{N}^{\sharp}\mathbb{N}R})$-bimodule which is the realization of the simplicial spectrum with $q$-simplices
    \[
        \bigvee_{(P_0,\dots,P_q)} F_{\mathbb{N}^{\sharp}\mathbb{N}R}(\mathbb{N}^{\sharp}\mathbb{N}P_q,\mathbb{N}^{\sharp}\mathbb{N}R)\wedge \mathbb{N}^{\sharp}F_{\mathbb{N}R}(\mathbb{N}P_{q-1},\mathbb{N}P_q)\wedge\dots \mathbb{N}^{\sharp}F_{\mathbb{N}R}(\mathbb{N}P_0,\mathbb{N}P_1)\wedge \mathbb{N}^{\sharp}\mathbb{N}P_0. 
    \]
    We can bring the $\mathbb{N}^{\sharp}$ functors inside the function spectra, since $\mathbb{N}^{\sharp}$ is monoidal, and then the composition and evaluation functors provide a map from this simplicial spectrum to the constant simplicial spectrum at $\mathbb{N}^{\sharp}\mathbb{N}R$, and realizing this map gives the map $\epsilon$.

    Checking the triangle identities amounts to the observation that the derived tensor-hom adjunction in the homotopy category of $R$-modules is indeed an adjunction.
\end{proof}
\begin{rem}
    In the non-equivariant story of Morita invariance one can construct a Morita equivalence from $\cC_R$ to $\stperf_R$ via the functor $\cC_R\to \stperf_R$ which sends the single object in $\cC_R$ to $R$.  We note that this is \emph{not} a $G$-spectral functor as $g\cdot R = R^{g^{-1}}\neq R$.  Thus this approach does not work in this setting.  
\end{rem}

Let $QR\xrightarrow{\sim} \mathbb{N}^{\sharp}\mathbb{N}R$ be a cofibrant replacement in the positive complete stable model structure on orthogonal $G$-ring spectra.  By restriction of scalars we have a spectral functor
\[
    \Mod_{\mathbb{N}^{\sharp}\mathbb{N}R}\to \Mod_{QR}
\]
which is a stable equivalence on all morphism spectra.  Using this restriction functor, we can lift the Morita adjunction of \autoref{proposition: Morita adjunction two} to a Morita adjunction from $\cD_R = Q\stperf_R$ to $\cC_{QR}$.  This has the advantage that now the source is cofibrant and the target is $G$-pointwise cofibrant.
\begin{lem}
    If $R$ is a cofibrant orthogonal ring spectrum with $G$-action then there is a weak equivalence of $G$-spectral categories $\cC_R\to \cC_{QR}$.
\end{lem}
\begin{proof}
    We may assume without loss of generality that the cofibrant replacement map $QR\xrightarrow{\sim} \mathbb{N}^{\sharp}\mathbb{N}R$ is an acyclic fibration.  Since $R$ is cofibrant, it follows that the unit map $R\to \mathbb{N}^{\sharp}\mathbb{N}R$ lifts to a ring map $R\to QR$.  By the two-of-three property this map is a weak equivalence, and induces a weak equivalence of one-object $G$-spectral categories.
\end{proof}

\begin{prop}\label{prop-morita-adj-R}
    For any cofibrant ring spectrum with $G$-action $R$, there is a Morita map
    \[
        |\Ncyc_{\bullet}(Q\Fun(\EG,Q\mathrm{stPerf}_R))|\to |\Ncyc_{\bullet}(R)|
    \]
    in the homotopy category of spectrum with $G$-action, which restricts to an isomorphism in the non-equivariant homotopy category.
\end{prop}
\begin{proof}
    Using \autoref{proposition: Morita adjunction one}, \autoref{proposition: Morita adjunction two}, and the discussion which follows we have Morita maps
    \[
        |\Ncyc_{\bullet}(Q\Fun(\EG,Q\stperf_R))|\to |\Ncyc_{\bullet}(Q\stperf_R)|
    \]
    and 
    \[
        |\Ncyc_{\bullet}(Q\stperf_R)|\to |\Ncyc_{\bullet}(QR)|
    \]
    and composition in the homotopy category gives the claimed map.  It is an isomorphism of underlying homotopy types because both Morita adjunctions are Morita equivalences on underlying spectra.

    To get to $|\Ncyc_{\bullet}(R)|$ we apply the lemma and \autoref{prop-cof-bar-invariance} to see there is an isomorphism $|\Ncyc_{\bullet}(QR)|\cong |\Ncyc_{\bullet}(R)|$ in the homotopy category.
\end{proof}

\section{Additivity for equivariant topological Hochschild homology}\label{sec:Additivity}

One of the fundamental theorems of topological Hochschild homology is the additivity theorem, which mimics Waldhausen's additivity theorem for algebraic $K$-theory.  In this section we give a version of this result for equivariant topological Hochschild homology which is used in the construction of the Dennis trace in \autoref{sec:DennisTrace}.

For any $1\leq j\leq n$ and $1\leq i_j\leq k_j$, there is a map $e_{i_j}\colon\mathcal{C}\to S_{k_j}\mathcal{C}$ which sends an object $c$ to the object $(A_{a,b})$ in $S_{k_j}\mathcal{C}$ with $A_{a,b}=c$ whenever $a<i_j$ and $b\geq i_j$, all other $A_{a,b}$ the zero objects, and all maps are either identities or zero maps.  For example, if $k_j =3$ and $i_j = 2$, $e_{2}(c)$ is the diagram
\[
    \begin{tikzcd}
        0 \ar[r] & 0 \ar[r] \ar[d]&c  \ar[r] \ar[d]& c\ar[d]\\
         &  0 \ar[r] & c\ar[r] \ar[d]& c \ar[d]\\
         & & 0 \ar[r]&  0\ar[d]\\
         & & & 0
    \end{tikzcd}
\]

Using degeneracy maps liberally, we obtain a functor $f_{i_j}\colon \mathcal{C}\to S_{k_j}\mathcal{C}\to S^{(n)}_{k_1,\dots,k_n}\mathcal{C}$.  Note that this functor is evidently equivariant. 

\begin{thm}[cf. {\cite[Theorem 7.6]{CLMPZ2}}] \label{thm-additivity-bar}
    The maps $f_{i_j}$ induce an equivalence of spectra with $G$-action
    \[
    \bigvee\limits_{\substack{1 \le i_j \le k_j\\ 1 \le j \le n}} |\Ncyc_{\bullet}(\mathcal{C})|\xrightarrow{\simeq} |\Ncyc_{\bullet}(w_{k_0}S^{(n)}_{k_1,\dots, k_n}\mathcal{C})| 
    \]
    for any spectral $G$-category $\cC$.
\end{thm}
\begin{proof}
    We need to show that the map in the theorem statement is an equivalence after taking $H$-fixed points for all $H\leq G$.  First, observe that the wedge on the left is a wedge of spectra with $G$-action, and hence we are considering the map
    \[
        \bigvee\limits_{\substack{1 \le i_j \le k_j\\ 1 \le j \le n}}  f_{i_j}^H\colon \bigvee\limits_{\substack{1 \le i_j \le k_j\\ 1 \le j \le n}} |\Ncyc_{\bullet}(\mathcal{C})|^H\xrightarrow{} |\Ncyc_{\bullet}(w_{k_0}S^{(n)}_{k_1,\dots, k_n}\mathcal{C})|^H.
    \]
    Since fixed points of finite group actions commute with geometric realizations and smash products, we apply \autoref{proposition: fixed points of cyclic nerve}, and see that this map is equivalent to
    \[
        \bigvee\limits_{\substack{1 \le i_j \le k_j\\ 1 \le j \le n}} |\Ncyc_{\bullet}(\mathcal{C}^H)|\xrightarrow{} |\Ncyc_{\bullet}((w_{k_0}S^{(n)}_{k_1,\dots, k_n}\mathcal{C})^H)|.
    \]
        Finally, since the $w_{k_0}S^{(n)}_{k_1,\dots,k_n}$ is a subcategory of a functor category, we can commute the fixed points inside the iterated $S_{\bullet}$ construction and the map becomes 
        \[
        \bigvee\limits_{\substack{1 \le i_j \le k_j\\ 1 \le j \le n}} |\Ncyc_{\bullet}(\mathcal{C}^H)|\xrightarrow{\simeq} |\Ncyc_{\bullet}((w_{k_0}S^{(n)}_{k_1,\dots, k_n}(\mathcal{C}^H)|.
    \]
    and this is an equivalence by the additivity theorem for the spectral Waldhausen category $\cC^H$; see \cite[Theorem 7.6]{CLMPZ2}.
\end{proof}
\begin{rem}\label{remark: equivalences of cyclotomic spectra}
    Since the map in \autoref{thm-additivity-bar} is induced by maps of spectral Waldhausen categories we obtain for free that it is a map of cyclotomic spectra.  In particular, it is a map of $S^1$-equivariant objects in spectra with $G$-action.  Moreover, if $H< S^1$ is any finite subgroup this map remains an equivalence upon taking $H$-fixed points since a map of cyclotomic spectra is an equivalence if and only if it is an equivalence on underlying spectra \cite[Proposition 5.5]{blumberg-mandell:cyclotomic}. 
\end{rem}

By changing universes, we can now deduce additivity for $\ETHH$.

\begin{cor}\label{cor-ETHH-additivity}
    If $\mathcal{C}$ is $G$-pointwise cofibrant, then the maps $\mathcal{I}_{\R^\infty} ^ \mathcal{U}f_{i_j}$ induce an equivalence of genuine orthogonal $G$-spectra
    $$\bigvee\limits_{\substack{1 \le i_j \le k_j\\ 1 \le j \le n}} \ETHH(\mathcal{C})\xrightarrow{\simeq} \ETHH(w_{k_0}S^{(n)}_{k_1,\dots, k_n}\mathcal{C}).
    $$
\end{cor}

\begin{proof}
    The change of universe functor $\mathcal{I}_{\R^\infty} ^ \mathcal{U}$ preserves weak equivalences between cofibrant objects. If $\mathcal{C}$ is $G$-pointwise cofibrant, both sides of the equivalence in Theorem \ref{thm-additivity-bar} are cofibrant, and therefore the equivalence is preserved by the change of universe functor.
\end{proof}

We will also need a slightly more refined version of this theorem when $G= C_n =\langle \sigma \mid \sigma^n\rangle$ is a cyclic subgroup.  We write $\mathcal{F}_{\Delta_{C_n}}$ for the family of subgroups of $C_n\times S^1$ which contains all subgroups of the form $H\times 1$, $1\times C_m$, and $K\leq \Delta_{C_n}$, where $\Delta_{C_n}$ is the diagonal copy of $C_n$ in $C_n\times S^1$.  The family $\mathcal{F}_{\Delta_{C_n}}$ determines the $\mathcal{F}_{\Delta_{C_n}}$-stable model structure on orthogonal spectra with $(C_n\times S^1)$-action where weak equivalences are those maps which induces isomorphisms on $\pi^\Gamma_*$ for all subgroups $\Gamma\in \mathcal{F}_{\Delta_{C_n}}$ \cite[Theorem IV.6.5]{mandell-may}.  We call such a map an $\mathcal{F}_{\Delta_{C_n}}$-stable equivalence.

\begin{thm}\label{theorem: additivity with Delta fixed points}
    The maps $f_{i_j}$ induce an $\mathcal{F}_{\Delta_{C_n}}$-stable equivalence of spectra with $(C_n\times S^1)$-action
    \[
     \bigvee\limits_{\substack{1 \le i_j \le k_j\\ 1 \le j \le n}} |\Ncyc_{\bullet}(\mathcal{C})|\xrightarrow{\simeq} |\Ncyc_{\bullet}(w_{k_0}S^{(n)}_{k_1,\dots, k_n}\mathcal{C})| 
    \]
    for any $C_n$-spectral category $\cC$.  In particular, it induces a stable equivalence on fixed points with respect to any subgroup of $C_n\times S^1$ of the form $H\times 1$, $1\times C_m$, or $K\leq \Delta_{C_n}$.
\end{thm}
\begin{proof}
    This map is an equivalence of spectra with $(C_n\times 1)$-action by the previous theorem.  It is an equivalence of spectra with $(1\times S^1)$-action by \autoref{remark: equivalences of cyclotomic spectra}.  Thus it suffices to check that this map remains an equivalence upon taking fixed points with respect to subgroups of the diagonal subgroup. We will check that it remains an equivalence after taking $\Delta_{C_n}$-fixed points.  Since every subgroup of the diagonal is also diagonal the case of proper subgroups of $\Delta_{C_n}$ is essentially the same.

    When we take $\Delta_{C_n}$ fixed points we obtain, via \autoref{prop: twisted fixed points of cyclic nerve}, the map
    \[
         \bigvee\limits_{\substack{1 \le i_j \le k_j\\ 1 \le j \le n}} |\Ncyc_{\bullet}(\mathcal{C},\cC^{\sigma^{-1}})|\xrightarrow{\simeq} |\Ncyc_{\bullet}(w_{k_0}S^{(n)}_{k_1,\dots, k_n}\mathcal{C},(w_{k_0}S^{(n)}_{k_1,\dots, k_n}\mathcal{C})^{\sigma^{-1}})| 
    \]
    which, is an equivalence by the twisted additivity theorem \cite[Theorem 5.9]{CLMPZ}.

\end{proof}

\section{An equivariant Dennis trace}\label{sec:DennisTrace}
In this section, we will define a Dennis trace map from equivariant algebraic $K$-theory to $\eTHH$, where $\eTHH(\mathcal{C}) = |\Ncyc_{\bullet}(\mathcal{C})|$. We first recall the definition of equivariant algebraic $K$-theory. If $R$ is a $G$-ring orthogonal spectrum, we consider the Waldhausen $G$-category $\perf_R$ of perfect right  $R$-modules, as in \autoref{ex-perf-action}.

\begin{defn}[\cite{MM1}]
    The equivariant algebraic $K$-theory of a $G$-ring spectrum $R$ is the genuine orthogonal $G$-spectrum $K_G(R)$ obtained in \cite{MM1} as the equivariant delooping of the $G$-infinite loop space
    $
       \Omega|wS_{\bullet} \Fun(\EG,\perf_R)|.
    $
\end{defn}

We will prove that the underlying spectrum with $G$-action of $K_G(R)$ has a trace map to  $ \eTHH(R) = |\Ncyc_{\bullet}(R)|$. More generally,

\begin{thm}\label{thm: Dennis trace for G spec Wald cats}
    For any spectral Waldhausen $G$-category $(\mathcal{C},\mathcal{C}_0)$ there is an equivariant Dennis trace map in the homotopy category of orthogonal spectra with $G$-action
    \[
        \mathrm{tr}\colon \mathcal{I}_{\mathcal{V}} ^{\R^\infty} K_G(\mathcal{C}_0)\to \eTHH(Q\Fun(\mathcal{E}G, \mathcal{C}))
    \]
    which recovers the usual Dennis trace when $G=e$. 
\end{thm}

The construction of this trace map follows the strategy of Section 7 of \cite{CLMPZ2}.
\begin{proof}

Suppose that $(\mathcal{C},\mathcal{C}_0)$ is a spectral Waldhausen $G$-category. To reduce notational clutter we let $(\cD,\cD_0) = (Q\Fun(\EG,\cC),\Fun(\EG,\cC_0))$.  For any object $d\in \mathcal{D}_0$ there is a unit map
\[
    \mathbb{S}\xrightarrow{1_d} \mathcal{D}(d,d).
\]
Putting these maps together we obtain a map
\[
    \Sigma^{\infty}\mathrm{ob}\mathcal{D}_0 = \bigvee\limits_{d\in \mathcal{D}_0} \mathbb{S}\xrightarrow{\vee 1_d} \bigvee\limits_{d\in \mathcal{D}_0} \mathcal{D}(d,d) \to |\Ncyc_{\bullet}(\cD)|
\]
where the last map is the inclusion of the $0$-skeleton. Replacing $\mathcal{D}$ with $w_{\bullet}S^{(n)}_{\bullet}\mathcal{D}$ we obtain, for any $(n+1)$-tuple $(k_0;k_1,\dots, k_n)$, a map
\begin{equation}\label{eq: multisimplicial map}
    \Sigma^{\infty}\mathrm{ob}w_{k_0}S^{(n)}_{k_1,\dots, k_n}\mathcal{D}_0 \to |\mathrm{N}^{\cyc}_{\bullet}(w_{k_0}S^{(n)}_{k_1,\dots, k_n}\mathcal{D})|
\end{equation}
which assembles into a map of multisimplicial orthogonal spectra with $G$-action.  

Using Theorem \ref{thm-additivity-bar}, together with the map \eqref{eq: multisimplicial map}, we have a zigzag of orthogonal spectra with $G$-action
\[
    \Sigma^{\infty}\mathrm{ob}w_{k_0}S^{(n)}_{k_1,\dots, k_n}\mathcal{D}_0 \to |\mathrm{N}^{\cyc}_{\bullet}(w_{k_0}S^{(n)}_{k_1,\dots, k_n}\mathcal{D}) |\xleftarrow{\simeq} \bigvee\limits_{\substack{1 \le i_j \le k_j\\ 1 \le j \le n}} |\mathrm{N}^{\cyc}_{\bullet}(\mathcal{D})|
\]
which assembles into a zigzag of multisimplicial orthogonal spectra with $G$-action. Note that the right hand side is just the smash product of $|\Ncyc_{\bullet}(\mathcal{D})|$ with the $n$-fold smash product of the simplicial circle:
\[
    \bigvee\limits_{\substack{1 \le i_j \le k_j\\ 1 \le j \le n}} |\mathrm{N}^{\cyc}_{\bullet}(\mathcal{D})|\cong (S^1_{\bullet})^{\wedge n} \wedge |\mathrm{N}^{\cyc}_{\bullet}(\mathcal{D})|.
\]

Now fixing $n$, we take geometric realizations to obtain a zigzag
\begin{equation}\label{eq: final zigzag}
    |\Sigma^{\infty}\mathrm{ob}w_{\bullet}S^{(n)}_{\bullet}\Fun(\mathcal{E}G, \mathcal{C}_0)| \to |\Ncyc_{\bullet}(w_{\bullet}S^{(n)}_{\bullet}Q\Fun(\mathcal{E}G, \mathcal{C}))| \xleftarrow{\simeq} \Sigma^n |\Ncyc_{\bullet}(Q\Fun(\mathcal{E}G, \mathcal{C}))|
\end{equation}

of spectra with $G$-action. As $n$ varies this is a symmetric spectrum object in orthogonal spectra with $G$-action.  Such objects are called \emph{bispectra} in \cite{CLMPZ2}.  The salient feature that we need is that there are two Quillen equivalences
\[
    \begin{tikzcd}
       \mathrm{Sp}^G_{\R^\infty} \ar[r,shift left,"\Sigma^{\infty}_{Obi}"] &  \mathrm{Bispectra}_G \ar[r,shift left,"\mathbb{P}_{biO}"] \ar[l,shift left]& \mathrm{Sp}^G_{\R^\infty} \ar[l,shift left,"\mathrm{sh}"] 
    \end{tikzcd}
\]
and $\mathbb{P}_{biO} \circ \Sigma^\infty _{Obi} = id$. Denoting by $\mathrm{Sp}^G _{sym}$ the category of symmetric spectra with $G$-action, there is also a similar adjunction
\[
    \begin{tikzcd}
    \mathrm{Sp}^G_{sym}\ar[r,shift left,"\Sigma^{\infty}_{Sbi}"] &  \mathrm{Bispectra}_G. \ar[l,shift left]
    \end{tikzcd}
\]  
If we denote by $\mathbb{P}_{SO}$ the prolongation functor from symmetric spectra with $G$-action to orthogonal spectra with $G$-action, we have $\mathbb{P}_{biO} \circ \Sigma^\infty _{Sbi} = \mathbb{P}_{SO}$. In the zigzag \ref{eq: final zigzag}, as $n$ varies, the left hand term is $\Sigma^\infty _{Sbi} K_G ^{sym} (\mathcal{C}_0)$, where $K_G ^{sym} (\mathcal{C}_0)$ denotes the symmetric spectrum with $G$-action whose $n^{th}$ space is $|\mathrm{ob}w_{\bullet}S^{(n)}_{\bullet}\Fun(\mathcal{E}G, \mathcal{C}_0)|$. As $n$ varies, the right hand term of the zigzag \ref{eq: final zigzag} is $\Sigma^\infty _{Obi} |\Ncyc_{\bullet}(Q\Fun(\mathcal{E}G, \mathcal{C}))|$. Applying $\mathbb{P}_{biO}$, we thus obtain a trace map 

$$\mathbb{P}_{SO} K_G ^{sym} (\mathcal{C}_0) \to |\Ncyc_{\bullet}(Q\Fun(\mathcal{E}G, \mathcal{C}))|.$$

By Proposition 2.23 and Remark 2.24 of \cite{MM1}, $\mathbb{P}_{SO} K_G ^{sym} (\mathcal{C}_0) \simeq \mathcal{I}_{\mathcal{V}} ^{\R^\infty} K_G(\mathcal{C}_0)$ as orthogonal spectra with $G$-action. Therefore, passing to the homotopy category, we obtain a trace map
\[
    \mathrm{tr}\colon \mathcal{I}_{\mathcal{V}} ^{\R^\infty} K_G(\mathcal{C}_0)\to |\Ncyc_{\bullet}(Q\Fun(\mathcal{E}G, \mathcal{C}))|.
\]

When $G = e$, this restricts to the construction of the Dennis trace of \cite{CLMPZ2}, and therefore recovers the usual Dennis trace.

\end{proof}

\begin{rem}\label{rem-diagonal-subgroup}
    Suppose that $G = C_n$, and consider $\mathcal{I}_{\mathcal{V}} ^{\R^\infty} K_{C_n}(\mathcal{C}_0)$ as a $C_n \times S^1$-spectrum with trivial $S^1$-action. Let $\Delta_{C_n}$ denote the diagonal subgroup $\Delta_{C_n} \leq C_n \times C_n \leq C_n \times S^1$. Since the additivity map is a $\Delta_{C_n}$-equivariant equivalence, it follows that the trace map above is $\Delta_{C_n}$-equivariant. 
\end{rem}

\begin{cor}\label{cor-trace-R}
    If $R$ is a cofibrant $G$-ring spectrum, there is an equivariant Dennis trace map in the homotopy category of orthogonal spectra with $G$-action
    $$tr_R\colon \mathcal{I}_{\mathcal{V}} ^{\R^\infty} K_G(R) \to\eTHH(R)$$
    which recovers the usual Dennis trace when $G=e$.
\end{cor}
\begin{proof}
    Take $(\mathcal{C}, \mathcal{C}_0)$ to be the spectral Waldhausen $G$-category $(Q\stperf_R, \perf_R)$. Then Theorem \ref{thm: Dennis trace for G spec Wald cats} gives the Dennis trace map $K_G(R) \to |\Ncyc_{\bullet} (Q\Fun(\mathcal{E}G, Q\stperf_R)|$ in the homotopy category of orthogonal spectra with $G$-action. To get $tr_R: \mathcal{I}_\mathcal{V} ^{\R^\infty} K_G(R) \to |\Ncyc_{\bullet} (R)|$, postcompose with the map $|\Ncyc_{\bullet} (Q\Fun(\mathcal{E}G, Q\stperf_R)| \to |\Ncyc_{\bullet} (R)| = \eTHH(R)$ given by the Morita map of Proposition \ref{prop-morita-adj-R}.
\end{proof}

We now show how our equivariant Dennis trace recovers the trace map $K_{C_n}(R)^{C_n} \to \THH_{C_n}(R)$ from \cite{aghkk:shadows} for a $C_n$-ring spectrum $R$.  To do this, recall the category of $G$-twisted $R$-modules.
\begin{defn}
    Let $(\cC,\otimes,1)$ be a symmetric monoidal category and let $R$ be a monoid in $\cC$ with left $G$-action.  That is, for all $g\in G$ there are monoid maps $g\colon R\to R$ in $\cC$ which are strictly associative and unital.  A $G$-twisted right $R$-module is a left $G$-object $M$ in $\cC$ together with a right $R$-module structure $\alpha\colon M\otimes R\to M$ such that the diagram
    \[
        \begin{tikzcd}
            M\otimes R \ar[r,"\alpha"] \ar[d,"g\otimes g^{-1}"] & M\ar[d,"g"]\\
            M\otimes R\ar[r,"\alpha"] & M
        \end{tikzcd}
    \]
    commutes.  We denote the category of $G$-twisted right $R$-modules by $\Mod(R)_G$.
\end{defn}
\begin{lem}\label{lem-hfp-modules}
    Let $(\cC,\otimes,1)$ be a symmetric monoidal category and let $R$ be a monoid in $\cC$ with left $G$-action.   Let $\Mod(R)$ denote the category of right $R$-modules.  Then there is an equivalence of categories between $\Fun(\EG,\Mod(R))^G$ and the category $\mathrm{Mod}(R)_G$.
\end{lem}
\begin{proof}
    An object $F\in \Fun(\EG,\Mod(R))$ consists of a collection of $R$-modules $F(g)$ for all $g\in G$ together with choices of $R$-module isomorphisms
    \[
        \phi_{h,g}\colon F(h)\xrightarrow{\cong} F(g)
    \]
    for all $g,h$. These $\phi_{h, g}$ satisfy the relations $\phi_{g,g} =\mathrm{id}_{F(g)}$ and $\phi_{g,k}\circ \phi_{h,g} = \phi_{h,k}$ for any $g,h,k\in G$. 
    
    Specializing to $h=1$ we obtain isomorphisms $\phi_{g} = \phi_{1,g}\colon F(1)\xrightarrow{\cong} F(g)$ such that the diagram
    \begin{equation}\label{equation: data of cofree module}
    \begin{tikzcd}
        F(1)\otimes R \ar[d,"\alpha_1"] \ar[r,"\phi_g\otimes 1"] & F(g)\otimes R \ar[d,"\alpha_g"]\\
        F(1) \ar[r,"\phi_g"] & F(g)
        \end{tikzcd}
    \end{equation}
    commutes, where $\alpha_g\colon F(g)\otimes R\to F(g)$ is the $R$-action map.
    
    The $G$-action is given by: $(x\cdot F)(g) = 
 F(x^{-1}g)^{x^{-1}}$ with structure morphisms 
 \[
    \phi_{x^{-1}g,x^{-1}h}^{x^{-1}}\colon F(x^{-1}h)^{x^{-1}}\to F(x^{-1}g)^{x^{-1}}.
\]  
If $F$ is $G$-fixed that means that for any $x,g\in G$ we have an equality of $R$-modules
    \[
        F(x^{-1}g)^{x^{-1}} = F(g)
    \]
    and hence, there is an equality of underlying objects in $\cC$, $F(x^{-1}g) = F(g)$. Taking $x=g$, this yields an equality of $R$-modules $F(1)^{g^{-1}} = F(g)$, hence the diagram
\begin{equation}\label{equation: data of fixed cofree module}
\begin{tikzcd}
	F(1)\otimes R &  F(g)\otimes R & {F(g)} \\
	{F(1)\otimes R} && {F(1)}
	\arrow["{=}", from=1-1, to=1-2]
	\arrow["{1\otimes g^{-1}}"', from=1-1, to=2-1]
	\arrow["{\alpha_g}", from=1-2, to=1-3]
	\arrow["{=}", from=1-3, to=2-3]
	\arrow["{\alpha_1}"', from=2-1, to=2-3]
\end{tikzcd}
\end{equation}
commutes.  

Define a $G$-action on $F(1)$, in $\cC$, by $F(1)\xrightarrow{\phi_{g}} F(g)=F(1)$.  We emphasize that this is an action in $\cC$, but not in the category of $R$-modules. Indeed, we claim that $F(1)$ is a $G$-twisted $R$-module. To see this, we need to show that the diagram

\[\begin{tikzcd}[column sep = large]
	{F(1)\otimes R} & {F(g)\otimes R} & {F(1)\otimes R} & {F(1)\otimes R} \\
	{F(1)} & {F(g)} & {F(1)} & {F(1)}
	\arrow["{\phi_{g}\otimes 1}", from=1-1, to=1-2]
	\arrow["{\alpha_1}"', from=1-1, to=2-1]
	\arrow["{=}", from=1-2, to=1-3]
	\arrow["{\alpha_{g}}", from=1-2, to=2-2]
	\arrow["{1\otimes g^{-1}}", from=1-3, to=1-4]
	\arrow["{\alpha_1}", from=1-4, to=2-4]
	\arrow["{\phi_{g}}"', from=2-1, to=2-2]
	\arrow["{=}"', from=2-2, to=2-3]
	\arrow["{=}"', from=2-3, to=2-4]
\end{tikzcd}\]
commutes. This follows from the commutativity of the diagrams \eqref{equation: data of cofree module} and \eqref{equation: data of fixed cofree module}.  Thus the assignment $F\mapsto F(1)$ is a function on objects
\[
    \Fun(\EG,\Mod(R))^G\to \Mod(R)_G
\]
and this is functorial in maps.  

We now show that  $\Psi(F) = F(1)$ is an equivalence of categories by constructing an inverse. Suppose that $M$ is a $G$-twisted $R$-module.  Define a functor $F_M\colon \EG\to \Mod(R)$ by $F_M(g) = M^{g^{-1}}$.  The structure map $\phi_{h,g}\colon M^{h^{-1}}\to M^{g^{-1}}$ is the map in $\cC$ given by acting by $gh^{-1}$.  This functor is $G$-fixed since, by definition, $F_M(x^{-1}g)^{x^{-1}} =F_M(1)^{g^{-1}} = F_M(g)$. 

Let $\Phi(M) =F_M$. Since $\Psi(F_M)(1) = M$, and since $\phi_{1,g}$ is multiplication by $g^{-1}$ we see that $\Psi\circ \Phi$ is the identity on $\Mod(R)_G$.  On the other hand, given an arbitrary $G$-fixed functor $B\colon \EG\to \Mod(R)$ observe that there is a natural isomorphism $F_{B(1)}\xrightarrow{\simeq} B$ which is given component-wise by the identity
\[  
    B(1)^{g^{-1}} = B(g^{-1}g)^{g^{-1}} =  B(g).
\]
Checking this is a well defined natural transformation is straightforward, hence $\Phi\circ \Psi$ is naturally isomorphic to the identity and these functors are mutually inverse equivalences.
\end{proof}
\begin{rem}\label{remark: homotopy fixed points and twisted modules}
    If $\cD\subset \Mod(R)$ is any full subcategory then the equivalence above restricts to an equivalence of categories between $\Fun(\EG,\cD)$ and $G$-twisted $R$-modules which are in objects in $\cD$.  This applies, in particular, when $R$ is a ring spectrum and $\cD$ is the category of perfect $R$-modules. Furthermore, in that case, $\Phi$ and $\Psi$ are both exact functors of Waldhausen categories, and therefore this equivalence of categories gives an equivalence on algebraic $K$-theory.
\end{rem}

\begin{thm}\label{thm-trace-to-twisted}
    Let $R$ be a cofibrant $C_n$-ring spectrum. Upon taking fixed points of the equivariant Dennis trace map of Corollary \ref{cor-trace-R}, we obtain the trace map 
    $$K_{C_n}(R)^{C_n} \to \THH_{C_n}(R)$$
    of \cite{aghkk:shadows}.
\end{thm}

\begin{proof}
 Since the change of universe to a smaller universe is a right adjoint, it commutes with limits, and thus $K_{C_n}(R)^{C_n} = (\mathcal{I}_{\mathcal{V}} ^{\R^\infty} K_{C_n}(R))^{C_n}$. Considering $\mathcal{I}_{\mathcal{V}} ^{\R^\infty} K_{C_n}(R)$ as a $C_n \times S^1$-spectrum with trivial $S^1$-action, we have $(\mathcal{I}_{\mathcal{V}} ^{\R^\infty} K_{C_n}(R))^{C_n} = (\mathcal{I}_{\mathcal{V}} ^{\R^\infty} K_{C_n}(R))^{\Delta_{C_n}} $, which is equivalent to $(\mathbb{P}_{SO} K_{C_n} ^{sym}(R))^{\Delta_{C_n}}$. The prolongation functor $\mathbb{P}_{biO}$ commutes with fixed points of finite group actions (one can check this using the formula above Proposition A.7 in \cite{CLMPZ2}), therefore the trace map on the fixed points comes from the zigzag 

\[
    \begin{tikzcd}[column sep = large]
\abs{\Sigma^{\infty}\mathrm{ob}w_{\bullet}S^{(m)}_{\bullet}\Fun(\mathcal{E}C_n, \perf_R)}^{\Delta_{C_n}} \ar[r] & \abs{\Ncyc_{\bullet}(w_{\bullet}S^{(m)}_{\bullet}Q\Fun(\mathcal{E}C_n, Q\stperf_R))}^{\Delta_{C_n}} \\       
& (\Sigma^m \abs{\Ncyc_{\bullet}(Q\Fun(\mathcal{E}C_n, Q\stperf_R))})^{\Delta_{C_n}}_.\ar[u,"\simeq"] 
    \end{tikzcd}
\]

Identifying the $\Delta_{C_n}$-fixed points of the cyclic bar construction with the twisted cyclic bar construction as in \autoref{prop: twisted fixed points of cyclic nerve}, and commuting the fixed points past the suspension spectrum and the $w_\bullet S_\bullet$-construction, we obtain
\[
\begin{tikzcd}
    \abs{\Sigma^{\infty}\mathrm{ob}w_{\bullet}S^{(m)}_{\bullet}(\Fun(\mathcal{E}C_n, \perf_R)^{C_n})} \ar[r] & \abs{\Ncyc_{\bullet} {}^{,C_n}(w_{\bullet}S^{(m)}_{\bullet}Q\Fun(\mathcal{E}C_n, Q\stperf_R))}\\
    & \Sigma^m \abs{\Ncyc_{\bullet} {}^{,C_n}(Q\Fun(\mathcal{E}C_n, Q\stperf_R))}.\ar[u,"\simeq"]
\end{tikzcd}
\]

As in Proposition \ref{prop-morita-adj-R}, the Morita map $|\Ncyc_{\bullet}(Q\Fun(\mathcal{E}C_n,Q\mathrm{stPerf}_R))|\to |\Ncyc_{\bullet}(R)|$ is an equivalence of spectra, and the same is true for the twisted cyclic bar constructions. Therefore we can replace  $Q\Fun(\mathcal{E}C_n, Q\stperf_R)$ above with $R$. Identifying $\Fun(\mathcal{E}C_n, \perf_R)^{C_n}$ with the category $\perf(R)_{C_n}$ of $C_n$-twisted perfect $R$-modules as in Lemma \ref{lem-hfp-modules} or Proposition 4.8 of \cite{Merling}, we obtain 

\begin{equation*}
    |\Sigma^{\infty}\mathrm{ob}w_{\bullet}S^{(m)}_{\bullet}(\perf(R)_{C_n})| \to |\Ncyc_{\bullet} {}^{,C_n}(w_{\bullet}S^{(m)}_{\bullet}Q\Fun(\mathcal{E}C_n, Q\stperf_R)| \xleftarrow{\simeq} \Sigma^m |\Ncyc_{\bullet} {}^{,C_n}(R))|
\end{equation*}

Applying prolongation to this zigzag recovers the twisted Dennis trace 
\[
    K(\perf(R)_{C_n}) \to |\Ncyc_{\bullet} {}^{,C_n}(R)|
\] 
of Definition 6.16 of \cite{CLMPZ}. This is how the trace map $K_{C_n}(R)^{C_n} \to \THH_{C_n}(R)$ in \cite{aghkk:shadows} was defined, which concludes the proof.  

\end{proof}

Interestingly enough, if $R$ is a genuine $G$-spectrum the trace map of Corollary \ref{cor-trace-R} does \textit{not} refine in general to a trace map of genuine $G$-spectra $K_G (R) \to \ETHH(R)$. Note that our trace map of spectra with $G$-action in the homotopy category gives an honest trace map from the cofibrant replacement $(I_{\mathcal{V}} ^{\R^\infty} K_G(R))^c$ to the cyclic bar construction, and therefore (after changing to a complete universe) a trace map
\begin{equation}\label{equation: genuine trace}
\mathcal{I}^{\mathcal{U}}_{\mathbb{R}^{\infty}} (\mathcal{I}_{\mathcal{V}} ^{\R^\infty} K_G(R))^c = \mathcal{I}_{\mathcal{V}} ^\mathcal{U} \mathcal{I}^{\mathcal{V}}_{\mathbb{R}^{\infty}} (\mathcal{I}_{\mathcal{V}} ^{\R^\infty} K_G(R))^c \to \ETHH(R)
\end{equation}
but the map 
$$\mathcal{I}^{\mathcal{V}}_{\mathbb{R}^{\infty}} (\mathcal{I}_{\mathcal{V}} ^{\R^\infty} K_G(R))^c \to \mathcal{I}^{\mathcal{V}}_{\mathbb{R}^{\infty}} \mathcal{I}_{\mathcal{V}} ^{\R^\infty} K_G(R)  = K_G(R)$$
does not admit a section, as the following example shows.

\begin{exmp}\label{example: no section}
    We consider an example where the map of genuine $G$-spectra \[\epsilon\colon \mathcal{I}^\mathcal{V}_{\mathbb{R}^{\infty}}((\mathcal{I}_\mathcal{V}^{\mathbb{R}^{\infty}}K_G(R))^c)\to K_G(R)\] cannot admit a section.  We will do this by showing that the map of $C_2$-Mackey functors induced by $\epsilon$ admits no section.  Let $G=C_2$ and $R = \mathbb{F}_4$, with its standard Galois action. 

    First, note that because change of universe preserves equivalences of underlying spectra we have that $\epsilon$ is an underlying equivalence. Thus any section of $\epsilon$ would necessarily be an underlying equivalence.  We claim there is no map of $C_2$-spectra which runs anti-parallel to $\epsilon$ and is an underlying equivalence. 

    To prove this claim, we compute homotopy Mackey functors.  By \cite[Theorem 1.2 (4)]{Merling}  we have 
    \begin{align*}
        \pi_0^{C_2}(K_{C_2}(\mathbb{F}_4)) \cong K_0(\mathbb{F}_2) & \cong \mathbb{Z}\\
        \pi_0^e(K_{C_2}(\mathbb{F}_4))  \cong K_0(\mathbb{F}_4)& \cong \mathbb{Z}
    \end{align*}
    and this is precisely the coefficient system $M$ considered in \autoref{example: free Mackey functor}.  By \autoref{proposition: computing change of universe on pi zero} we can compute the source of $\epsilon$ to be the Mackey functor $j_!(M)$ from \autoref{example: free Mackey functor}. On $\pi_0$, any section of $\epsilon$ would be a map of Mackey functors of the form
\[\begin{tikzcd}
	\Z && {\Z^2} \\
	\Z && \Z
	\arrow["f", from=1-1, to=1-3]
	\arrow["1"', shift right, from=1-1, to=2-1]
	\arrow["{(1,2)}"', shift right, from=1-3, to=2-3]
	\arrow["2"', shift right, from=2-1, to=1-1]
	\arrow["{\pm 1}"', from=2-1, to=2-3]
	\arrow["i_2"', shift right, from=2-3, to=1-3]
\end{tikzcd}\]
where $f$ is some group homomorphism and $i_2\colon \mathbb{Z}\to \mathbb{Z}^2$ is the inclusion of the second component. But this would require $f(2) = (0,\pm1)$, which is impossible since $1$ is not $2$-divisible in $\mathbb{Z}$.
\end{exmp}
\begin{rem}\label{remark: trace on genuine fixed points}
    As noted in \autoref{remark: fixed points are retract of change universe fixed points}, there is a map
    \[
        K_{G}(R)^G\cong (\mathcal{I}_{\mathcal{V}}^{\R^{\infty}}K_{G}(R)^c)^G \to (\mathcal{I}^{\mathcal{V}}_{\R^{\infty}}(\mathcal{I}_{\mathcal{V}}^{\R^{\infty}}K_G(R))^c)^G
    \]
    in the homotopy category of spectra which is the inclusion of a retract.  Extending along the map \eqref{equation: genuine trace}, we obtain a trace map
    \[
        K_{G}(R)^G\to \ETHH(R)^G
    \]
    on the level of fixed point spectra for any genuine cofibrant $G$-ring spectrum.
\end{rem}

We end with applications of the equivariant Dennis trace to equivariant $A$-theory.

\subsection{Applications to equivariant $A$-theory}

We now show how our trace map gives a map from the fixed points of equivariant $A$-theory to the free loop space. Malkiewich--Merling \cite{MM1} define two versions of equivariant $A$-theory: coarse equivariant $A$-theory $A_G^{\mathrm{coarse}}(X)$, whose fixed points are related to bivariant $A$-theory, and genuine equivariant $A$-theory $A_G(X)$, which is related to the equivariant parametrized stable $h$-cobordism theorem. In forthcoming work of the first-named author, Calle, Chedalavada, and Mejia, it is shown that the coarse $A$-theory spectrum of a $G$-space $X$ is equivalent to the equivariant algebraic $K$-theory of $\Sigma^{\infty}_G(\Omega X)_+$.

\begin{prop}[\cite{CCCM}]
    For any finite group $G$ and any $G$-connected based $G$-space $X$ there is an isomorphism in the equivariant stable homotopy category
    \[
        A^{\mathrm{coarse}}_G(X)\cong K_G(\Sigma^{\infty}_G(\Omega X)_+).
    \]
\end{prop}

In particular, our equivariant Dennis trace yields a map in the homotopy category of spectra with $G$-action 
\[
    \mathcal{I}_{\mathcal{V}} ^{\R^\infty} A^{\mathrm{coarse}}_G(X)\to \eTHH(\Sigma^{\infty}(\Omega X)_+)\simeq \Sigma^{\infty}(\mathcal{L}X)_+.
\]
Non-equivariantly this map plays an important role in computations of $A$-theory, as well as the proof of a piece of the parametrized stable $h$-cobordism theorem \cite{WaldhausenJahrenRognes,BHM}.  Taking fixed points, we obtain for every $H \leq G$ a map in the homotopy category of spectra
\[
   (A^{\mathrm{coarse}}_G(X))^H =  (I_{\mathcal{V}} ^{\R^\infty} A^{\mathrm{coarse}}_G(X))^H\to (\Sigma^{\infty}(\mathcal{L}X)_+)^H = \Sigma^{\infty}(\mathcal{L}X^H)_+.
\]
In particular, since the fixed points of coarse $A$-theory give bivariant $A$-theory, we obtain a trace map from bivariant $A$-theory to the free loop space.

With an eye toward the equivariant parametrized stable $h$-cobordism theorem, it is natural to wonder about such a map for genuine $A$-theory. In \cite{MM1}, it is shown that there is a comparison map
\[
    A_G(X)\to A^{\mathrm{coarse}}_G(X).
\]
On fixed points, postcomposition with our Dennis trace yields
\[
    A_G(X)^H\to \Sigma^{\infty}(\mathcal{L}X^H)_+.
\]

To summarize,

\begin{prop}\label{prop-trace-A-loop}
    Let $H \leq G$. The equivariant Dennis trace map of Theorem \ref{cor-trace-R} gives trace maps
    \[
   (A^{\mathrm{coarse}}_G(X))^H \to \Sigma^{\infty}(\mathcal{L}X^H)_+\hspace{.5cm} \textrm{and}\hspace{.5cm} A_G(X)^H\to \Sigma^{\infty}(\mathcal{L}X^H)_+
\] 
in the homotopy category of spectra.
\end{prop}

\begin{rem}\label{rem-A-to-ETHH}
We conjecture that the trace map $A_G(X)^H \to \Sigma^{\infty}(\mathcal{L}X^H)_+$ arises from a genuine equivariant trace map $A_G(X) \to \ETHH(\Sigma^{\infty}_G(\Omega X)_+) \simeq \Sigma^{\infty} _{G \times S^1}(\mathcal{L}X)_+$. Indeed, by the tom Dieck splitting, for any $H \leq G$,
$$(\Sigma^{\infty}_{G \times S^1}(\mathcal{L}X)_+)^H \simeq \bigvee_{(H') \leq H} \Sigma^{\infty}_+(\mathcal{L}X)^{H'}_{hW_H(H')}$$
whose $H' = H$ summand is $\Sigma^{\infty}(\mathcal{L}X)^H_+ = \Sigma^{\infty}(\mathcal{L}X^H)_+$.
\end{rem}

\bibliographystyle{plain}
\bibliography{bib}

\begin{thebibliography}{10}

\bibitem{aghkk}
Katharine Adamyk, Teena Gerhardt, Kathryn Hess, Inbar Klang, and Hana~Jia Kong.
\newblock Computational tools for twisted topological {H}ochschild homology of
  equivariant spectra.
\newblock {\em Topology Appl.}, 316:Paper No. 108102, 26, 2022.

\bibitem{aghkk:shadows}
Katharine Adamyk, Teena Gerhardt, Kathryn Hess, Inbar Klang, and Hana~Jia Kong.
\newblock A shadow perspective on equivariant {H}ochschild homologies.
\newblock {\em Int. Math. Res. Not. IMRN}, 18:15299--15357, 2023.

\bibitem{AKGH}
Gabriel Angelini-Knoll, Teena Gerhardt, and Michael Hill.
\newblock Real topological {H}ochschild homology, norms, and {W}itt vectors.
\newblock {\em arxiv 2111.06970}, 2021.

\bibitem{AKMP}
Gabriel Angelini-Knoll, Mona Merling, and Maximilien Péroux.
\newblock Topological {$\Delta G$} homology of rings with twisted {$G$}-action.
\newblock {\em arxiv 2409.18187}, 2024.

\bibitem{AnBlGeHiLaMa}
Vigleik Angeltveit, Andrew~J. Blumberg, Teena Gerhardt, Michael~A. Hill, Tyler
  Lawson, and Michael~A. Mandell.
\newblock Topological cyclic homology via the norm.
\newblock {\em Doc. Math.}, 23:2101--2163, 2018.

\bibitem{Barwick}
Clark Barwick.
\newblock Spectral {M}ackey functors and equivariant algebraic {$K$}-theory
  ({I}).
\newblock {\em Adv. Math.}, 304:646--727, 2017.

\bibitem{BGS}
Clark Barwick, Saul Glasman, and Jay Shah.
\newblock Spectral {M}ackey functors and equivariant algebraic {$K$}-theory,
  {II}.
\newblock {\em Tunis. J. Math.}, 2(1):97--146, 2020.

\bibitem{BlGeHiLa}
Andrew~J. Blumberg, Teena Gerhardt, Michael~A. Hill, and Tyler Lawson.
\newblock The {W}itt vectors for {G}reen functors.
\newblock {\em J. Algebra}, 537:197--244, 2019.

\bibitem{BlumbergHillMandellNorms}
Andrew~J. Blumberg, Michael~A. Hill, and Michael~A. Mandell.
\newblock Norms for compact {L}ie groups in equivariant stable homotopy theory.
\newblock {\em arxiv 2212.11404}, 2022.

\bibitem{blumberg-mandell:localization}
Andrew~J. Blumberg and Michael~A. Mandell.
\newblock Localization theorems in topological {H}ochschild homology and
  topological cyclic homology.
\newblock {\em Geom. Topol.}, 16(2):1053--1120, 2012.

\bibitem{blumberg-mandell:cyclotomic}
Andrew~J. Blumberg and Michael~A. Mandell.
\newblock The homotopy theory of cyclotomic spectra.
\newblock {\em Geom. Topol.}, 19(6):3105--3147, 2015.

\bibitem{BlumbergMandell:localizationLongOne}
Andrew~J. Blumberg and Michael~A. Mandell.
\newblock Localization for {$THH(ku)$} and the topological {H}ochschild and
  cyclic homology of {W}aldhausen categories.
\newblock {\em Mem. Amer. Math. Soc.}, 265(1286):v+100, 2020.

\bibitem{BMOOP:modelStructureGCat}
Anna~Marie Bohmann, Kristen Mazur, Ang\'elica~M. Osorno, Viktoriya Ozornova,
  Kate Ponto, and Carolyn Yarnall.
\newblock A model structure on {$G$Cat}.
\newblock In {\em Women in topology: collaborations in homotopy theory}, volume
  641 of {\em Contemp. Math.}, pages 123--134. Amer. Math. Soc., Providence,
  RI, 2015.

\bibitem{BHM}
M.~B{\"o}kstedt, W.~C. Hsiang, and I.~Madsen.
\newblock The cyclotomic trace and algebraic ${K}$-theory of spaces.
\newblock {\em Invent. Math.}, 111(3):465--539, 1993.

\bibitem{CCCM}
Maxine Calle, David Chan, Anish Chedalavada, and Andres Mejia.
\newblock {Spherical group ring models for the equivariant algebraic $K$-theory
  of spaces}.
\newblock forthcoming.

\bibitem{CLMPZ}
Jonathan Campbell, John Lind, Cary Malkiewich, Kate Ponto, and Inna
  Zakharevich.
\newblock {$K$}-theory of endomorphisms, the {$TR$}-trace, and zeta functions.
\newblock {\em arXiv preprint:2005.04334}, 2020.

\bibitem{CLMPZ2}
Jonathan~A. Campbell, John~A. Lind, Cary Malkiewich, Kate Ponto, and Inna
  Zakharevich.
\newblock Spectral {W}aldhausen categories, the {$S_\bullet$}-construction, and
  the {D}ennis trace.
\newblock {\em Grad. J. Math.}, 9(2):27--60, 2024.

\bibitem{ChanVogeli}
David Chan and Chase Vogeli.
\newblock The {G}alois-equivariant {$K$}-theory of finite fields.
\newblock {\em Proc. Lond. Math. Soc. (3)}, 130(1):Paper No. e70012, 36, 2025.

\bibitem{CMM21}
Dustin Clausen, Akhil Mathew, and Matthew Morrow.
\newblock {$K$}-theory and topological cyclic homology of henselian pairs.
\newblock {\em J. Amer. Math. Soc.}, 34(2):411--473, 2021.

\bibitem{CMNN}
Dustin Clausen, Akhil Mathew, Niko Naumann, and Justin Noel.
\newblock Descent and vanishing in chromatic algebraic {$K$}-theory via group
  actions.
\newblock {\em Ann. Sci. \'Ec. Norm. Sup\'er. (4)}, 57(4):1135--1190, 2024.

\bibitem{DottoMalkPatchSagaveWoo}
Emanuele Dotto, Cary Malkiewich, Irakli Patchkoria, Steffen Sagave, and Calvin
  Woo.
\newblock Comparing cyclotomic structures on different models for topological
  {H}ochschild homology.
\newblock {\em J. Topol.}, 12(4):1146--1173, 2019.

\bibitem{DMPR}
Emanuele Dotto, Kristian Moi, Irakli Patchkoria, and Sune~Precht Reeh.
\newblock Real topological {H}ochschild homology.
\newblock {\em Journal of the European Mathematical Society}, 23(1):63--152,
  2020.

\bibitem{DGM}
Bj\o rn~Ian Dundas, Thomas~G. Goodwillie, and Randy McCarthy.
\newblock {\em The local structure of algebraic {K}-theory}, volume~18 of {\em
  Algebra and Applications}.
\newblock Springer-Verlag London, Ltd., London, 2013.

\bibitem{ElmantoZhang}
Elden Elmanto and Ningchuan Zhang.
\newblock Equivariant algebraic $\mathrm{K}$-theory and {A}rtin
  {$L$}-functions.
\newblock {\em arXiv 2405.03578}, 2024.

\bibitem{ekmm}
Anthony~D Elmendorf, I~Kriz, MA~Mandell, and JP~May.
\newblock Rings, modules, and algebras in stable homotopy theory.
\newblock In {\em American Mathematical Society Surveys and Monographs,
  American Mathematical Society}. Citeseer, 1995.

\bibitem{GeHe06}
Thomas Geisser and Lars Hesselholt.
\newblock Bi-relative algebraic {$K$}-theory and topological cyclic homology.
\newblock {\em Invent. Math.}, 166(2):359--395, 2006.

\bibitem{GoodwillieIgusaMalkiewichMerling}
Thomas Goodwillie, Kiyoshi Igusa, Cary Malkiewich, and Mona Merling.
\newblock On the functoriality of the space of equivariant smooth
  $h$-cobordisms, 2023.

\bibitem{GuillouMay:IteratedLoopSpaceTheory}
Bertrand~J. Guillou and J.~Peter May.
\newblock Equivariant iterated loop space theory and permutative
  {$G$}-categories.
\newblock {\em Algebr. Geom. Topol.}, 17(6):3259--3339, 2017.

\bibitem{Guilloumay}
Bertrand~J. Guillou and J.~Peter May.
\newblock Models of {$G$}-spectra as presheaves of spectra.
\newblock {\em Algebr. Geom. Topol.}, 24(3):1225--1275, 2024.

\bibitem{HHR}
M.~A. Hill, M.~J. Hopkins, and D.~C. Ravenel.
\newblock On the nonexistence of elements of {K}ervaire invariant one.
\newblock {\em Ann. of Math. (2)}, 184(1):1--262, 2016.

\bibitem{HHR:Book}
Michael~A. Hill, Michael~J. Hopkins, and Douglas~C. Ravenel.
\newblock {\em Equivariant stable homotopy theory and the {K}ervaire invariant
  problem}, volume~40 of {\em New Mathematical Monographs}.
\newblock Cambridge University Press, Cambridge, 2021.

\bibitem{Lewis:Splitting}
L.~Gaunce Lewis, Jr.
\newblock Splitting theorems for certain equivariant spectra.
\newblock {\em Mem. Amer. Math. Soc.}, 144(686):x+89, 2000.

\bibitem{Malkiewich:Coassembly}
Cary Malkiewich.
\newblock Coassembly and the {$K$}-theory of finite groups.
\newblock {\em Adv. Math.}, 307:100--146, 2017.

\bibitem{MM1}
Cary Malkiewich and Mona Merling.
\newblock Equivariant {$A$}-theory.
\newblock {\em Doc. Math.}, 24:815--855, 2019.

\bibitem{MM2}
Cary Malkiewich and Mona Merling.
\newblock The equivariant parametrized {$h$}-cobordism theorem, the
  non-manifold part.
\newblock {\em Adv. Math.}, 399:Paper No. 108242, 42, 2022.

\bibitem{MalkiewichPonto}
Cary Malkiewich and Kate Ponto.
\newblock Periodic points and topological restriction homology.
\newblock {\em Int. Math. Res. Not. IMRN}, 4:2401--2459, 2022.

\bibitem{mandell-may}
M.~A. Mandell and J.~P. May.
\newblock Equivariant orthogonal spectra and {$S$}-modules.
\newblock {\em Mem. Amer. Math. Soc.}, 159(755):x+108, 2002.

\bibitem{May:GeometryIteratedLoops}
J.~P. May.
\newblock {\em The geometry of iterated loop spaces}, volume Vol. 271 of {\em
  Lecture Notes in Mathematics}.
\newblock Springer-Verlag, Berlin-New York, 1972.

\bibitem{MayThomason}
J.~P. May and R.~Thomason.
\newblock The uniqueness of infinite loop space machines.
\newblock {\em Topology}, 17(3):205--224, 1978.

\bibitem{Merling}
Mona Merling.
\newblock Equivariant algebraic {K}-theory of {$G$}-rings.
\newblock {\em Math. Z.}, 285(3-4):1205--1248, 2017.

\bibitem{NS}
Thomas Nikolaus and Peter Scholze.
\newblock On topological cyclic homology.
\newblock {\em Acta Math.}, 221(2):203--409, 2018.

\bibitem{schwede-shipley:equivmonmodcat}
Stefan Schwede and Brooke Shipley.
\newblock Equivalences of monoidal model categories.
\newblock {\em Algebr. Geom. Topol.}, 3:287--334, 2003.

\bibitem{Stephan}
Marc Stephan.
\newblock On equivariant homotopy theory for model categories.
\newblock {\em Homology Homotopy Appl.}, 18(2):183--208, 2016.

\bibitem{Waldhausen}
Friedhelm Waldhausen.
\newblock Algebraic {$K$}-theory of spaces.
\newblock In {\em Algebraic and geometric topology ({N}ew {B}runswick,
  {N}.{J}., 1983)}, volume 1126 of {\em Lecture Notes in Math.}, pages
  318--419. Springer, Berlin, 1985.

\bibitem{WaldhausenJahrenRognes}
Friedhelm Waldhausen, Bj\o~rn Jahren, and John Rognes.
\newblock {\em Spaces of {PL} manifolds and categories of simple maps}, volume
  186 of {\em Annals of Mathematics Studies}.
\newblock Princeton University Press, Princeton, NJ, 2013.

\bibitem{Williams}
Bruce Williams.
\newblock Bivariant {R}iemann {R}och theorems.
\newblock In {\em Geometry and topology: {A}arhus (1998)}, volume 258 of {\em
  Contemp. Math.}, pages 377--393. Amer. Math. Soc., Providence, RI, 2000.

\end{thebibliography}
\end{document}